\documentclass[a4paper,oneside]{amsart}

\usepackage{graphicx}

\usepackage{amsmath}
\usepackage{amssymb}
\usepackage{amsthm}
\usepackage{bbm,euscript,mathrsfs}
\usepackage{amsfonts}
\usepackage{amsxtra}
\usepackage{color}

\usepackage{enumitem}
\setenumerate{label={\rm (\alph{*})}}

\usepackage{xifthen}

\usepackage{aliascnt}
\usepackage[bookmarksopen,bookmarksdepth=2,backref=none,hidelinks]{hyperref}

\usepackage[frak]{paper_diening}

\theoremstyle{definition}
\newtheorem{definition}{Definition}[section]
\newaliascnt{lemma}{definition}
\newaliascnt{theorem}{definition}
\newaliascnt{corollary}{definition}
\newaliascnt{proposition}{definition}
\newaliascnt{assumption}{definition}
\newaliascnt{remark}{definition}
\newaliascnt{example}{definition}
\newaliascnt{conjecture}{definition}
\usepackage{esint}

\newtheoremstyle{remark2}{}{}{}{}{\bfseries}{.}{ }{}
\theoremstyle{remark2}

\newtheorem{remarkx}[remark]{Remark}
\newenvironment{remark}
  {\pushQED{\qed}\remarkx}
  {\popQED\endremarkx}

\theoremstyle{plain}
\newtheorem{lemma}[lemma]{Lemma}
\newtheorem{theorem}[theorem]{Theorem}
\newtheorem{corollary}[corollary]{Corollary}

\newtheorem{assumption}[assumption]{Assumption}

\newcommand{\toetai}{\to^\eta}
\newcommand{\weaktoetai}{\weakto^\eta}

\aliascntresetthe{lemma}
\aliascntresetthe{theorem}
\aliascntresetthe{corollary}
\aliascntresetthe{proposition}
\aliascntresetthe{remark}
\aliascntresetthe{conjecture}
\aliascntresetthe{assumption}

\newcommand{\R}{\ensuremath{\mathbb{R}}}

\newcommand{\N}{\ensuremath{\mathbb{N}}}

\newcommand{\dd}{\mathrm{d}}
\newcommand{\dx}{\,\mathrm{d}x}
\newcommand{\dy}{\,\mathrm{d}y}
\newcommand{\dt}{\,\mathrm{d}t}
\newcommand{\dxt}{\,\mathrm{d}x\,\mathrm{d}t}
\newcommand{\ds}{\,\mathrm{d}s}
\newcommand{\dxs}{\,\mathrm{d}x\,\mathrm{d}s}

\renewcommand{\abs}[1]{\left|{#1}\right|}
\renewcommand{\norm}[2][]{\left\|#2\right\|_{{#1}}} %
\newcommand{\inner}[3][]{\left\langle #2,#3 \right \rangle_{#1}} %

\DeclareMathOperator{\Div}{div}
\DeclareMathOperator{\cof}{cof}
\DeclareMathOperator{\tr}{tr}

\DeclareMathOperator{\Lip}{Lip}

\newcommand{\seb}[1]{{\color{blue}#1}}

\numberwithin{equation}{section}
\allowdisplaybreaks

\usepackage{todonotes}

\begin{document}
\title[Compressible fluids and 3D visco-elastic bulks]{Compressible fluids interacting with 3D visco-elastic bulk solids} 

\author{Dominic Breit}
\address[D. Breit]{
Department of Mathematics, Heriot-Watt University, Riccarton Edinburgh EH14 4AS, UK}
\email{d.breit@hw.ac.uk}

\author{Malte Kampschulte}

\author{Sebastian Schwarzacher}
\address[M. Kampschulte \& S. Schwarzacher]{Department of Mathematical Analysis, Faculty of Mathematics and Physics, Charles University, Sokolovsk\'a 83, 186 75 Praha, Czech Republic}
\email{kampschulte@karlin.mff.cuni.cz, schwarz@karlin.mff.cuni.cz}

\begin{abstract}
We consider the physical setup of a three-dimensional fluid-structure interaction problem. A viscous compressible gas or liquid interacts with a nonlinear, visco-elastic, three-dimensional bulk solid. The latter is described by a hyperbolic evolution with a non-convex elastic energy functional. The fluid is modelled by the compressible Navier--Stokes equations with a barotropic pressure law. Due to the motion of the solid, the fluid domain is time-changing.
 Our main result is the long-time existence of a weak solution to the coupled system until the time of a collision.
 The nonlinear coupling between the motions of the two different matters is established via the method of minimising movements. The motion of both the solid and the fluid is chosen via an incrimental minimization with respect to dissipative and static potentials. These variational choices together with a careful construction of an underlying flow map for our approximation then directly result in the pressure gradient and the material time derivatives.
\end{abstract}


\date{\today}

 \maketitle

 \section{Introduction}
 The interaction of fluids with elastic strucures has attracted the interest of scientists from various fields due to the immense potential for applications ranging from hydro- and aero-elasticity \cite{Do} over bio-mechanics \cite{BGN} to hydrodynamics
\cite{Su}. In last two decades there has also been a huge development in the understanding of the mathematical models -- typically highly coupled systems of nonlinear partial differential equations (PDEs).\\
The interaction of
incompressible viscous fluids with elastic structures has been studied intensively. Results concerning the existence of weak solutions to the coupled system (which exist as long as the moving part of the structure does not touch the fixed part of the fluid boundary) can be found, for instance, in \cite{Ch, Gr, LeRu, MuCa1, MuCa, MuSc}. 
The first result in \cite{Ch} is concerned with a flexible elastic plate located on one
part of the fluid boundary. The shell equation is linearised and the shell is assumed to be one-dimensional.
The existence of a weak solution to the incompressible Navier--Stokes equations coupled with a plate in flexion is proved in \cite{Gr}. In \cite{MuCa1} the incompressible Navier--Stokes equations are studied in a cylindrical wall. The movement of the latter is modelled by the one-dimensional cylindrical linearised Koiter shell model. The elastodynamics of the cylinder wall in \cite{MuCa} is governed by the one-dimensional linear wave equation modelling a thin structural layer, and by the two-dimensional equations of linear elasticity modelling a thick structural layer.
The interaction with a linear-elastic shell of Koiter-type in a general geometric set-up (where the middle surface of the shell serves as the mathematical boundary of the three-dimensional fluid domain) is studied in \cite{LeRu}. 
The result has recently been extended
to original (and fully nonlinear) Koiter model, cf. \cite{MuSc}.
 \\
What all these results have in common is that the incompressible Navier--Stokes equations are coupled to a lower dimensional hyperbolic equation for the structure. We are interested in models where fluid domain and structure have the same dimension which is the case for the interaction with an elastic bulk. A major difficulty is that the associated elastic energy involves a non-convex functional. Consequently, it is not possible to apply fixed point arguments to obtain a solution 
and one is forced to use variational methods instead. A corresponding approach has been developed recently by the second and third author together with B.\ Bene\v{s}ov\'a in \cite{benesovaVariationalApproachHyperbolic2020}. 
They provide an approach to the fully coupled system which is based on De Giorgi's celebrated minimising movement method.
The key idea in \cite{benesovaVariationalApproachHyperbolic2020} is that the second order time derivative of the structure displacement is discretised in two steps (yielding a velocity scale and an acceleration scale) and to use a Lagrangian approximation
of the material time derivative on the time-discrete level.
The result is the existence of a global-in-time weak solution to
coupled system describing the interaction of an incompressible fluid with a three-dimensional visco-elastic bulk solid.\\
 The situation in the compressible case is completely different and the analysis of problems from fluid-structure interaction is still at the beginning. A first result has been achieved by the first and third author in \cite{BrSc}. They prove the existence of a weak solution to a coupled system describing the motion of a three-dimensional compressible fluid interacting with a two-dimensional linear elastic shell of Koiter-type. Recently, they extended the result to heat-conducting fluids and fully nonlinear shell models in \cite{BrSc2}. Results on the existence of local strong solutions appeared recently in  
 \cite{Ma,Ma2,Mi,Tr}. In this paper we aim at the natural next step and consider the interaction of a compressible fluid with a three-dimensional visco-elastic bulk solid in order to arrive at a compressible counterpart of the result in \cite{benesovaVariationalApproachHyperbolic2020}.

Let us present the model in detail. The
fluid together with the elastic structure are both confined to a fixed container -- a bounded Lipschitz domain $\Omega\subset\R^n$ with $n\geq 2$ (where $n=3$ is the most interesting case, but also $n=2$ has physical relevance). The deformation of the solid
is described by the deformation function $\eta:Q\rightarrow\Omega$, where $Q\subset\R^n$ is the reference configuration of the solid, which is assumed to be a bounded Lipschitz domain as well. We denote by $M$ the part of $Q$ which is mapped to the contact interface between the 
fluid
and the solid and set $P:=\partial Q\setminus M$.
For a given deformation, we then deal with a fluid domain $\Omega_{\eta}:=\Omega\setminus \eta(Q)$, which will be time-dependent. In $\Omega_\eta$ we observe the flow of a viscous compressible fluid subject to the volume force $f_f:I\times \Omega_\eta\rightarrow\R^n$. We seek the velocity field $v:I\times \Omega_\eta\rightarrow\R^n$ and the density $\varrho:I\times \Omega_\eta\rightarrow\R$ solving the system
 \begin{align}
 \partial_t\varrho+\Div(\varrho v)&=0&\text{ in } &I\times \Omega_\eta,\label{eq1}
 \\ \partial_t(\varrho v)+\Div(\varrho v\otimes v)&=\Div\mathbb S(\nabla v)-\nabla p(\varrho)+{\varrho}f_f&\text{ in } &I\times\Omega_\eta,\label{eq2}
 \\
 v(t,\eta(t,x))&=\partial_t\eta(t,x)&\text{ in } &I\times M,\label{eq3a}
 \\
  v(t,x)&=0&\text{ in } &I\times \partial\Omega\setminus\eta(P),\label{eq3b}
 \\
  \varrho(0)=\varrho_0,\quad (\varrho v)(0)&=q_0&\text{ in } &\Omega_{\eta(0)}.\label{eq3}
 \end{align}
 Here, $p(\varrho)$ is the pressure which is assumed to follow the $\gamma$-law, that is $p(\varrho)\sim\varrho^\gamma$ for large $\varrho$, where $\gamma>1$ (see Section \ref{sec:ass} for the precise assumptions).
Further, we suppose Newton's rheological law
\begin{align*}
\mathbb S(\nabla v)=2\mu\Big(\frac{\nabla v+\nabla v^\top}{2}-\frac{1}{3}\Div v\,\mathbb{I}\Big)+\lambda \Div v\,\mathbb{I}
\end{align*} 
with strictly positive viscosity coefficients $\mu,\,\lambda$.
The balance of linear momentum for the solid is given by
 \begin{align}\label{eq:4}
  \varrho_s\partial_t^2\eta+\Div\sigma&=f_s\quad\text{in}\quad I\times Q,
\end{align}
where $\varrho_s>0$ is the density of the shell and $f_s:I\times Q\rightarrow \R^3$ a given external force.
We suppose that the first Piola--Kirchhoff stress tensor $\sigma$ can be derived from underlying
energy and dissipation potentials; that is
\begin{align*}
\Div\sigma=DE(\eta)+D_2 R(\eta,\partial_t\eta)
\end{align*}
for some energy- and dissipation functionals $E$ and $R$.
Prototypical example examples are given by
\begin{align*}
R(\eta,\partial_t\eta)=\int_Q|\partial_t\nabla\eta^\top\nabla\eta+\nabla\eta^\top\partial_t\nabla\eta|^2\dy,
\end{align*}
and for some $q>n$ and $a>\frac{nq}{q-n}$
\begin{align*}
E(\eta)=\frac{1}{8}\int_Q\Big(\mathbb C(\nabla\eta^\top\nabla\eta- \mathbb{I}):(\nabla\eta^\top\nabla\eta-\mathbb{I})+\frac{1}{(\det\nabla\eta)^a}+\frac{1}{q}|\nabla^2\eta|^q\Big)\dx.
\end{align*}
The latter is defined provided $\det \nabla\eta>0$ a.e.\ in $Q$ and we set $E(\eta)=\infty$ otherwise. Here $\mathbb C$ is the
positive definite tensor of elastic constants. The general assumptions for $E$ and $R$ are collected in Section \ref{sec:ass}.
 Equation \eqref{eq:4} is supplemented with the boundary conditions
\begin{align}\label{eq5}
\begin{aligned}
\sigma(t,x) \nu(x)&=\big(\mathbb S(\nabla v)-p(\varrho)\mathcal I\big)\hat{\nu}(t,\eta(t,x))\quad\text{in}\quad I\times M,\\
\eta(t,x)&=\gamma\quad\text{in} \quad I\times P.
\end{aligned}
\end{align}
Here $\nu(x)$ is the unit normal to $M$ while $\hat{\nu}(t,\eta(t,x))=\cof(\nabla\eta(t,x))\nu(x)$ is the normal transformed
to the actual configuration and $\gamma:P\rightarrow\Omega$ is a given function. Finally, we assume the initial conditions
\begin{align}
\label{initial}
\eta(0,\cdot)=\eta_0,\quad \partial_t\eta(0,\cdot)=\eta_1\quad\text{in}\quad Q,
\end{align}
where $\eta_0,\eta_1:Q\rightarrow\R$ are given functions. We aim to prove the existence of a weak solution to \eqref{eq1}--\eqref{initial}, where the precise formulation can be found in Section \ref{sec:weak}.
A simplified versions of our main result reads as follows and we refer to Theorem \ref{thm:main} for the complete statement.
\begin{theorem} \label{thm:MAINa}
Under natural assumptions on the data there exists a weak solution $(\eta,v,\varrho)$ to \eqref{eq1}--\eqref{initial} which satisfies the energy inequality
\begin{align}\label{eq:apriori0}
\begin{aligned}
\mathscr E(s)&+\int_0^s\int_{\Omega_\eta}\mathbb S(\nabla v):\nabla v\dxt+2\int_0^s R(\eta,\partial_t\eta)\dt\\ &\leq
\mathscr E(0)+\int_{\Omega_{\eta}}\varrho f_f\cdot v\dx+\int_Q f_s\cdot\partial_t\eta\dy,\\
\mathscr E(t)&= \int_{\Omega_\eta(t)}\left( \varrho(t) \tfrac{\abs{ v(t)}^2}{2} + H(\varrho(t))\right)\dx+\varrho_s\int_Q \tfrac{|\partial_t\eta(t)|^2}{2}\,\dy+ E(\eta(t)),\end{aligned}
\end{align} 
for all $s\in I$, where $H$ is the pressure potential related to $p$ by $p(\varrho) = H'(\varrho)\varrho-H(\varrho)$.
The interval of existence is of the form $I=(0,T)$, where $T$ is the first time of collision or $\infty$ if there is none.
\end{theorem}
To prove Theorem \ref{thm:MAINa} we aim to apply a variational approach in the spirit of 
 \cite{benesovaVariationalApproachHyperbolic2020}, where the same problem was solved in the easier case of an incompressible fluid. This approach is based on a time-delayed approximation: One seeks a continuous solution for which the inertial terms (i.e.\ $\partial_{tt} \eta$ and the material derivative $\tfrac{\mathrm{D}v}{\mathrm{D}t}$) occur in form of a difference quotient involving a time step $h>0$. The resulting equation is of gradient-flow type, which allows us to construct solutions using the minimizing movements scheme. On a formal level, this method is well suited to the compressible case: The flow map, that is constructed to obtain the material derivative, can directly be used to transport the density. It turns out that the variational nature of the scheme allows us to generate the pressure term directly from the pressure potential which we include in the functional we are minimising in each step. The material derivative of the density also occurs as a difference quotient which leads to the desired equation of continuity in the limit.
 
Let us now explain how to make these ideas rigorous.
When solving the compressible Navier--Stokes equations it is common to work with an artificial viscosity $\varepsilon$ in \eqref{eq1} and add $\varepsilon\Delta\varrho$ on the right-hand side to make it a proper parabolic equation. As it turns out the limit $\varepsilon\rightarrow 0$ can only be performed if the integrability of the density (which results from the parameter $\gamma$ in \eqref{eq2}) is large and thus outside the realm of physical interest. Consequently, a second regularisation ($\delta$-level with artificial pressure) is used and one adds an additional term $\delta\varrho^\beta$ (with $\beta$ sufficiently large) to the pressure in \eqref{eq2}. This approach has been introduced in \cite{F}. In order to solve the regularised
problem with $\varepsilon,\delta>0$ fixed, we aim at a variational approach as described above. Due to the additional term $\varepsilon\Delta\varrho$, we need to modify the mass transport from a simple update of the densities via the flow map. This change also necessitates a rather specific choice of the discretised inertial term in order to obtain the correct energy inequality (see Section \ref{subsec:h-construction} for details).
Additionally, when combined with the pressure potential, this perturbation creates in the limit of the time-step $h\rightarrow0$ the term
\begin{align*}
\frac{\varepsilon}{2}\int_I\int_{\Omega_\eta}\nabla\varrho\cdot(\nabla v \phi+\nabla\phi v)\dxt
\end{align*}
in the momentum equation (as well as a similar term in the energy inequality). This corresponds to a regularisation used similarly in \cite{F} and many subsequent papers. In order to recover this term, it turns out that we have to exclude the vacuum for the limit passage
$h\rightarrow0$, see equation \eqref{limregh}. Excluding the hypothetical vacuum is a big open problem for the compressible Navier--Stokes system even in presence of positive $\varepsilon$.
To overcome this difficulty, we split the regularisation of $v$, that was already present in \cite{benesovaVariationalApproachHyperbolic2020}, on the $h$-level into its own $\kappa$-level: For fixed $\kappa>0$ we add a higher order dissipation to the momentum equation. This gives sufficient regularity of the velocity which ultimately yields a minumum principle for the equation of continuity.\\
For fixed $\kappa,\varepsilon,\delta>0$ we are able to apply the ideas just explained and to obtain a solution to the approximate problem by the minimising movement method (with the acceleration scale limit $\tau\rightarrow 0$ in Section \ref{subsec:tau} and the velocity scale limit $h\rightarrow0$ in Section \ref{sec:h}).
Eventually, we pass to the limit with respect to the regularisation parameters $\kappa,\varepsilon$ and $\delta$ in Section \ref{sec:5}. For technical reasons this has to be done in three independent steps. The limit $\kappa\rightarrow0$ is rather straightforward as the density remains compact for $\varepsilon>0$. The compactness of the density becomes critical in the subsequence limit procedures where we pass to the limit in $\varepsilon$ and $\delta$ respectively.
This is done via the method of effective viscous flux which is due to Lions \cite{Li2} (with important extensions by Feireisl et al.\ \cite{F}). This method has been extended to the setting of variable domains in \cite{BrSc}.\\
It is important to note that in each of these approximations, our approximate solutions are already confined to the set of admissible states. In particular, on each level the solid deformation $\eta$ is injective, the solid and the fluid move with the same velocities at the interface and the total mass of the fluid is conserved at all times. Additionally, an energy inequality holds on each level. It mirrors the physical energy inequality and will be our main tool to obtain a priori estimates. 

 \section{Preliminaries}
 
 \subsection{Assumptions}
 \label{sec:ass}
 The assumptions on the solid and its dissipation are identical to those in \cite{benesovaVariationalApproachHyperbolic2020}:
 \begin{assumption}[Elastic energy] 
\label{ass:energy}
We assume that $q>n$ and $E:W^{2,q}(Q;\Omega) \to \overline{\R}$ satisfies:
\begin{enumerate}
 \item[S1] Lower bound: There exists a number $E_{min} > -\infty$ such that 
 \[E(\eta) \geq E_{min} \text{ for all } \eta \in W^{2,q}(Q; \R^n).\]
 \item[S2] Lower bound in the determinant: For any $E_0>0$ there exists $\varepsilon_0 >0$ such that $\det \nabla \eta \geq \varepsilon_0$ for all $\eta \in \{\eta \in W^{2,q}(Q; \R^n):E(\eta) <E_0\}$.
 \item[S3] Weak lower semi-continuity: If $\eta_l \rightharpoonup \eta$ in $W^{2,q}(Q; \R^n)$ then $E(\eta) \leq \liminf_{l\to\infty} E(\eta_l)$.
 \item[S4] Coercivity: All sublevel-sets $\{\eta \in \mathcal{E} :E(\eta) <E_0\}$ are bounded in $W^{2,q}(Q; \R^n)$.
 \item[S5] Existence of derivatives: For finite values $E$ has a well defined derivative which we will formally denote by 
 \[DE:  \{ \eta \in \mathcal{E}: E(\eta) < \infty\} \to (W^{2,q}(Q; \R^n))'.\]
 Furthermore, on any sublevel-set of $E$, $DE$ is bounded and continuous with respect to strong $W^{2,q}(Q;\R^n)$-convergence.
 \item[S6] Monotonicity and Minty-type property: If $\eta_l \rightharpoonup \eta$ in $W^{2,q}(Q; \R^n)$ with $\sup_lE(\eta_l) < \infty$, then \[\liminf_{l\to \infty} \inner{DE(\eta_l)-DE(\eta)}{(\eta_l-\eta)\psi} \geq 0 \text{ for all }\psi\in C^\infty_0(Q;[0,1]).\] If additionally $\limsup_{l\to \infty} \inner{DE(\eta_l)-DE(\eta)}{(\eta_l-\eta)\psi} \leq 0$ then $\eta_l\to \eta$ in $W^{2,q}(Q; \R^n)$.
\end{enumerate}
\end{assumption}
 
\begin{definition}[Domain of definition] \label{rem:ciarletNecas}
The set of admissible deformations in $W^{2,q}(Q; \Omega)$ (injective a.e.\ and satisfying the Dirichlet boundary condition) can be expressed as
\begin{align}
\label{eq:etaspace}
  \mathcal{E} := \left\{\eta \in W^{2,q}(Q;\Omega): E(\eta) < \infty,\, \abs{\eta(Q)} = \int_Q \det \nabla \eta \dx, \eta|_{P} = \gamma\right\}
\end{align}
where $\gamma: P \to \partial \Omega$ is a fixed injective function of sufficient regularity so that $\mathcal{E}$ is non-empty. By a slight abuse of notation, we use $L^\infty(I;\mathcal{E})$ to denote the set of functions $\eta: I \times Q \to \Omega$, for which $\eta(t) \in \mathcal{E}$ for all $t \in I$ and $t\mapsto E(\eta(t))$ is bounded.
\end{definition}

\begin{assumption}[Solid dissipation] The dissipation functional $R: \mathcal{E} \times W^{1,2}(Q;\R^n) \to \R$ satisfies: \label{ass:dissipation}
\begin{enumerate}
  \item[R1] Weak lower semi-continuity: If $b_l \rightharpoonup b$ in $W^{1,2}(Q;\R^n)$ then 
  \[\liminf_{l\to\infty} R(\eta,b_l) \geq R(\eta,b).\]
  \item[R2] Homogeneity of degree two: The dissipation is homogeneous of degree two in its second argument , i.e.,
  \[R(\eta, \lambda b) = \lambda^2 R(\eta, b) \quad \forall \lambda \in \R.\]
  In particular, this implies $R(\eta,b) \geq 0$ and $R(\eta,0) = 0$.
  \item[R3] Energy-dependent Korn-type inequality: Fix $E_0 >0$. Then there exists a constant $c_K = c_K(E_0) >0$ such that for all $\eta \in W^{2,q}(Q;\R^n)$ with $E(\eta) \leq E_0$ and all $b \in W^{1,2}(Q;\R^n)$ with $b|_{P} = 0$ we have
  \[c_K \norm[W^{1,2}(Q)]{b}^2 \leq R(\eta,b).\]
  \item[R4] Existence of a continuous derivative: The derivative $D_2R(\eta,b) \in (W^{1,2}(Q;\R^n))'$ given by
  \[ \frac{\dd}{\dd\varepsilon} |_{\varepsilon = 0} R(\eta,b+\varepsilon \phi) =: \inner{D_2R(\eta,b)}{ \phi} \quad \forall \phi \in W^{1,2}(Q;\R^n)\]
  exists and is weakly continuous in its two arguments. Due to the homogeneity of degree two this in particular implies
  \[\inner{D_2R(\eta,b)}{ b } = 2R(\eta,b).\]
 \end{enumerate}
\end{assumption}
 
 Additionally, we need to state some assumptions on the potential energy of the density, which will also result in the pressure response.

  \begin{assumption}[Pressure]
  The function $p:[0,\infty)\rightarrow[0,\infty)$ satisfies the following.
  \begin{itemize}
   \item[P1:] $p\in C^2((0,\infty))\cap C^1([0,\infty))$;
   \item[P2:] $p'(\varrho)>0$ for all $\varrho>0$; 
   \item[P3:] $p$ has $\gamma$-growth with $\gamma > \frac{2n(n-1)}{3n-2}$, i.e.\ there exists $a > 0$ and $\gamma > \frac{2n(n-1)}{3n-2}$ such that
   \begin{align*}
\lim_{\varrho\rightarrow\infty}\frac{p'(\varrho)}{\varrho^{\gamma-1}}=a.
   \end{align*}
  \end{itemize}
  \end{assumption}
 We will also consider the regularised pressure $p_\delta$
 given by $$p_\delta(\varrho):=p(\varrho)+\delta\varrho^\beta+\delta \varrho^2$$
 for $\delta>0$, where $\beta>\max\{4,\gamma\}$. The function $p_\delta$ clearly inherits P1 and P2 but has $\beta$-growth instead of $\gamma$-growth. Additionally, one also has that
 \begin{align*}
 p''_\delta(\varrho)\geq\,2\delta\quad \text{for all }\varrho>0.
 \end{align*}
 Given the pressure through the function $p$ the potential energy of the density can be described by the fluid potential
   \begin{align*}
   U_\eta(\varrho) := \int_{\Omega \setminus \eta(Q)} H(\varrho) \dx,\quad H(\varrho):=\varrho\int_0^\varrho \frac{p(z)}{z^2}\,\dd z,
  \end{align*}
  which satisfies the relation $p(\varrho) = H'(\varrho) \varrho - H(\varrho)$. For the pressure potential $H$ we also obtain a regularised version given by $H_\delta(\varrho):=\varrho\int_c^\varrho \frac{p_\delta(z)}{z^2}\,\dd z = H(\varrho) + \frac{\delta}{\beta-1} \varrho^\beta + \delta \varrho^2$ for a fixed $c\geq 0$. During the minimising movement approach in Sections \ref{sec:tau} and \ref{sec:h}
  only $H_\delta$ appears. Some of its properties, which follow directly from P1--P3 above, are listed in the following lemma.
 \begin{lemma}[Fluid potential]\label{lem:fluidpotential}
 For fixed $\delta>0$ the function $H_\delta$ satisfies the following for some $c_\delta>0$.
  \begin{itemize}
   \item[H1:] $H_\delta$ is bounded from below;
   \item[H2:] $H_\delta$ is strictly convex, i.e.\ $H_\delta''(\varrho) \geq c_{\delta} > 0$ for all $\varrho>0$;
   \item[H3:] $H_\delta$ has $\beta$-growth, i.e., we have
   \begin{align*}
    H_\delta(\varrho) \geq c_\delta (\varrho^{\beta}-1 ) &\text{ for all } \varrho \in [0,\infty).
   \end{align*}
  \end{itemize}
 \end{lemma}

%
%
 
 \subsection{Function spaces on variable domains}
\label{ssec:geom}

For a given deformation $\eta : Q \rightarrow \Omega$ we parametrise the deformed fluid domain by
\begin{align*}
\Omega_\eta := \Omega \setminus \eta(Q).
\end{align*}
If $\eta : I \times Q \rightarrow \Omega$ is additionally taken to be time-dependent, this defines a deformed space-time cylinder $I\times\Omega_\eta :=\bigcup_{t\in I}\set{t}\times\Omega_{\eta(t)} \subset I \times \Omega$. To save on notation, we will sometimes define shorter notation in the presence of indices and parameters, e.g. 
\[\Omega_l^{(h)}(t) := \Omega_{\eta^{(h)}_l(t)}.\]

Recall that the moving part of the boundary of $\Omega_{\eta(t)}$ is given by $\eta(t)|_{M}:M\rightarrow \Omega$.
The corresponding function spaces for variable domains are defined as follows.
\begin{definition}{(Function spaces)}
For $I=(0,T)$, $T>0$, and $\eta\in C(\overline{I}\times Q; \Omega)$ defining a changing domain $\Omega(t) := \Omega \setminus \eta(t,Q)$ we define for $1\leq p,r\leq\infty$
\begin{align*}
L^p(I;L^r(\Omega(\cdot))&:=\big\{v\in L^1(I\times\Omega(\cdot)):\,\,v(t,\cdot)\in L^r(\Omega(t))\,\,\text{for a.e. }t,\,\,\|v(t,\cdot)\|_{L^r(\Omega(t))}\in L^p(I)\big\},\\
L^p(I;W^{1,r}(\Omega(\cdot)))&:=\big\{v\in L^p(I;L^r(\Omega(\cdot))):\,\,\nabla v\in L^p(I;L^r(\Omega(\cdot)))\big\}.
\end{align*}
\end{definition}
Function spaces of vector- or matrix valued functions are defined accordingly. 
We now give a definition of convergence in variable domains. Convergence in Lebesgue spaces follows from an extension by zero.
\begin{definition}
\label{def:conv}
Let $(\eta_i)\subset C(\overline{I}\times Q;\Omega)$ with $\eta_i\rightarrow \eta$ uniformly in $\overline I\times Q$.
 Let $p\in [1,\infty]$ and $k\in\N_0$.
\begin{enumerate}
\item We say that  a sequence $(g_i) \subset L^p(I,L^q(\Omega_{\eta_i}))$ converges to $g$ in $L^p(I,L^q(\Omega_{\eta}))$ strongly with respect to $(\eta_i)$, in symbols
$
g_i\toetai g \,\text{in}\, L^p(I,L^q(\Omega_{\eta})),
$
if
\begin{align*}
\chi_{\Omega_{\eta_i}}g_i\to \chi_{\Omega_\eta}g \quad\text{in}\quad L^p(I,L^q(\R^n)).
\end{align*}
\item Let $p,q<\infty$. We say that  a sequence $(g_i) \subset L^p(I,L^q(\Omega_{\eta_i}))$ converges to $g$ in $L^p(I,L^q(\Omega_{\eta}))$ weakly with respect to $(\eta_i)$, in symbols
$
g_i\weaktoetai g \,\text{in}\, L^p(I,L^q(\Omega_{\eta})),
$
if
\begin{align*}
\chi_{\Omega_{\eta_i}}g_i\weakto \chi_{\Omega_\eta}g \quad\text{in}\quad L^p(I,L^q(\R^n)).
\end{align*}
\item Let $p=\infty$ and $q<\infty$. We say that  a sequence $(g_i) \subset L^\infty(I,L^q(\Omega_{\eta_i}))$ converges to $g$ in $L^\infty(I,L^q(\Omega_{\eta}))$ weakly$^*$ with respect to $(\eta_i)$, in symbols
$
g_i\weakto^{\ast,\eta} g \,\text{in}\, L^\infty(I,L^q(\Omega_{\eta})),
$
if
\begin{align*}
\chi_{\Omega_{\eta_i}}g_i\weakto^* \chi_{\Omega_\eta}g \quad\text{in}\quad L^\infty(I,L^q(\R^n)).
\end{align*}
\end{enumerate}
\end{definition}
Next we state a compactness lemma from \cite[Lemma 2.8]{BrSc} which gives a variant of the classical result by Aubin-Lions for PDEs in variables domains. It allows to pass to the limit in the product
of two weakly converging sequences provided one is more regular in space and the other one in time.
 Let us list the required assumptions.
\begin{enumerate}[label={\rm (A\arabic{*})}, leftmargin=*]
\item\label{A1} The sequence $(\eta_i)\subset C(\overline I\times Q;\Omega)$  satisfies $\eta_i\rightarrow \eta$ uniformly in $\overline I\times Q$ and for some $\alpha>0$
\begin{align*}
\eta_i&\rightarrow\eta\quad\text{in}\quad C^\alpha(\overline I\times Q).
\end{align*}
\item\label{A2} Let $(v_i)$ be a sequence such that for some $p,s\in[1,\infty)$ we have
$$v_i\weaktoetai v\quad\text{in}\quad L^p(I;W^{1,s}(\Omega_{\eta})).$$
\item\label{A3} Let $(r_i)$ be a sequence such that for some $m,b\in[1,\infty)$ we have
$$r_i\weaktoetai r\quad\text{in}\quad L^m(I;L^{b}(\Omega_{\eta})).$$
Assume further that $(\partial_t r_i)$ is bounded in the sense of distributions, i.e., there is $c>0$ and $m\in\N$ such that
\begin{align*}
\int_I\int_{\Omega_{\eta_i}}r_i\,\partial_t\phi\dxt\leq\,c\int_I\|\phi\|_{W^{m,2}(\Omega_{\eta_i})}\dt
\end{align*}
uniformly in $i$ for all $\phi\in C^\infty_c(I\times\Omega_{\eta_i})$.
\end{enumerate}
In \cite[Lemma 2.8]{BrSc} the corresponding version of \ref{A3} assumes $m=2$. But the very same argument is also valid in the general case as in the classical Aubin-Lions argument.
\begin{lemma}
\label{thm:weakstrong}
Let $(\eta_i)$, $(v_i)$ and $(r_i)$ be sequences satisfying \ref{A1}--\ref{A3}
 where $\frac{1}{s^*}+\frac{1}{b}=\frac{1}{a}<1$ (with $s^*=\frac{ns}{n-s\alpha}$ if $s\in (1,3/\alpha)$ and $s^*\in(1,\infty)$ arbitrarily otherwise) and $\frac{1}{m}+\frac{1}{p}=\frac{1}{q}<1$. Then there is a subsequence with
\begin{align}
\label{al}
v_i r_i\weakto^\eta v r\text{ weakly in }L^{q}(I,L^a(\Omega_{\eta})).
\end{align}
\end{lemma}

\begin{corollary}
\label{rem:strong}
In the case $r_i=v_i$ we find that
\[
v_i\to^\eta v\text{ strongly in }L^{2}(I, L^{2}(\Omega_{\eta})).
\]
\end{corollary}

 \subsection{Weak solutions and the main theorem}
 \label{sec:weak}
 In this session we make the concept of a weak solution to
 \eqref{eq1}--\eqref{initial} rigorous. We begin with the function spaces to which the triple $(\eta,v,\varrho)$ belongs.
\begin{itemize}
\item For the solid deformation $\eta:I\times Q\to \Omega$ we consider the space
$$Y^I:=\{\zeta\in W^{1,2}(I;W^{1,2}(Q;\R^n))\cap L^\infty(I;\mathcal{E})\,\}.$$
\item Given $\eta\in Y^I$, for the fluid velocity $v:I\times \Omega_\eta\to\mathbb{R}^n$ we define the space $$X_\eta^I:=L^2(I;W^{1,2}(\Omega_{\eta};\R^n)).$$
\item Given $\eta\in Y^I$, for the fluid density $\varrho: I\times \Omega_\eta\to [0,\infty)$ we define the space $$
 Z_\eta^I:= C_w(\overline{I};L^\gamma(\Omega_\eta)),$$ where the subscript $w$ refers to continuity with respect to the weak topology. 
\end{itemize}

A weak solution to \eqref{eq1}--\eqref{initial} is a triple $(\eta,v,\varrho)\in Y^I\times X_{\eta}^I \times Z_\eta^I $ that satisfies the following.
\begin{itemize}
\item\label{*1} The momentum equation holds in the sense that\begin{align}\label{eq:mom}
\begin{aligned}
&\int_I\frac{\dd}{\dt}\int_{\Omega(t)}\varrho v \cdot b\dx-\int_{\Omega(t)} \Big(\varrho v\cdot \partial_t b +\varrho v\otimes v:\nabla b \Big)\dxt
\\
&+\int_I\int_{\Omega(t)}\mathbb S(\nabla v):\nabla b \dxt-\int_I\int_{\Omega(t)}
p_\delta(\varrho)\,\Div b\dxt\\
&+\int_I\bigg(-\int_Q \varrho_s \partial_t\eta\,\partial_t \phi\dy + \langle DE(\eta),\phi\rangle+\langle D_2R(\eta,\partial_t\eta),\phi\rangle\bigg)\dt
\\&=\int_I\int_{\Omega(t)}\varrho f_f\cdot b\dxt+\int_I\int_Q f_s\cdot \phi\,\dd x\dt
\end{aligned}
\end{align} 
for all $(\phi,b)\in L^2(I; W^{2,q}(Q;\R^n))\cap W^{1,2}(I;L^2(Q;\R^n))\times C_0^\infty(\overline{I}\times \Omega; \R^n)$ with  $\phi(t)=b(t)\circ\eta(t)$ in $Q$ and $b(t)=0$ on $P$ for a.a. $t\in I$. Moreover, we have $(\varrho v)(0)=q_0$, $\eta(0)=\eta_0$ and $\partial_t\eta(0)=\eta_1$ we well as $\partial_t\eta(t)=v(t)\circ\eta(t)$ in $Q$, $\eta(t)\in\mathcal E$ and $v(t)=0$ on $\partial\Omega$ for a.a. $t\in I$.
\item\label{*2}  The continuity equation holds in the sense that
\begin{align}\label{eq:con}
\begin{aligned}
&\int_I\frac{\dd}{\dt}\int_{\Omega(t)}\varrho \psi\dxt-\int_I\int_{\Omega(t)}\Big(\varrho\partial_t\psi
+\varrho v\cdot\nabla\psi\Big)\dxt=0
\end{aligned}
\end{align}
for all $\psi\in C^\infty(\overline{I}\times\R^3)$ and we have $\varrho(0)=\varrho_0$. 
\item \label{*3} The energy inequality is satisfied in the sense that
\begin{align} \label{eq:ene}
\begin{aligned}
- \int_I &\partial_t \psi \,
\mathscr E \dt+\int_I\psi\int_{\Omega(t)}\mathbb S(\nabla v):\nabla v\dxs+2\int_I\psi R(\eta,\partial_t\eta)\ds \\&\leq
\psi(0) \mathscr E(0)+\int_I\int_{\Omega(t)}\varrho f_f\cdot v\dxt+\int_I\psi\int_Q f_s\,\partial_t\eta\,\dd y\dt
\end{aligned}
\end{align}
holds for any $\psi \in C^\infty_c([0, T))$.
Here, we abbreviated
$$\mathscr E(t)= \int_{\Omega(t)}\Big(\frac{1}{2} \varrho(t) | {v}(t) |^2 + H(\varrho(t))\Big)\dx+\int_Q\varrho_s\frac{|\partial_t\eta|^2}{2}\dy+ E(\eta(t)).$$
\end{itemize}

\begin{remark}\label{rem:malte}
 Due to the bulk setting, it is possible to extend the fluid variables onto the fixed domain $\Omega$ using their solid counterparts. We will do so quite often for the velocity, where it is convenient to set $v(t,\eta(t,y)) := \partial_t \eta(t,y)$ for all $y \in Q$, as this results in the complete Eulerian velocity $v \in L^2(I;W^{1,2}_0(\Omega;\R^n))$. For the density, one can similarly consider pushing the solid density forward onto $\Omega$ using $\eta$. However in practice this is less useful, as the resulting function will still have a jump at the interface between solid and fluid. Instead, it is often more convenient to keep the inertial effects of the solid in their Lagrangian description and to think of $\varrho$ as extended by $0$ onto $\Omega$, as this allows for the removal of most of the time-dependent domains in the weak solution.
\end{remark}
We are finally in the position to state our main result in a complete form.
\begin{theorem}\label{thm:main}
Assume that we have
\begin{align*}
\frac{|q_0|^2}{\varrho_0}&\in L^1(\Omega_{\eta_0}),\ \varrho_0\in L^{\gamma}(\Omega_{\eta_0}), \ \eta_0\in \mathcal E,\ \eta_1\in L^2(Q;\R^n),\\
f_f&\in L^2(I;L^\infty(\R^n;\R^n)),\ f_s\in L^2(I\times Q;\R^n).
\end{align*}
There is a weak solution $(\eta,v,\varrho)\in Y^I\times X_\eta^I\times Z_\eta^I$ to \eqref{eq1}--\eqref{initial} in the sense of \eqref{eq:mom}--\eqref{eq:ene}. 
Here, we have $I=(0,T_*)$, where $T_*<T$ only if the time $T_*$ is the time of the first contact
of the free boundary of the solid body either with itself or $\partial\Omega$ (i.e. $\eta(T_*)\in\partial\mathcal E$). 
\end{theorem}

As will be apparent by the analysis we will show that the renormalised continuity equation in the sense of DiPerna and Lions is satisfied, cf. \cite{DL,Li2}.
\begin{definition}[Renormalized continuity equation]
\label{def:ren}
Let $\eta\in Y^I$ and $v\in X_\eta^I$. We say that the function $\varrho\in W_\eta^I$ solves the continuity equation~\eqref{eq:con} in the renormalized sense if we have
\begin{align}\label{eq:final3b}
\begin{aligned}
\int_I\frac{\dd}{\dt}\int_{\Omega(t)}\theta(\varrho)\psi\dxt&-\int_I\int_{\Omega(t)}\Big(\theta(\varrho)\partial_t\psi +\theta(\varrho) v\cdot \nabla \psi\Big)\dxt
\\
& =- \int_I\int_{\Omega(t)}(\varrho\theta'(\varrho)-\theta(\varrho)) \Div v \,\psi\dxt
\end{aligned}
\end{align}
for all $\psi\in C^\infty(\overline{I}\times\R^3)$ and all $\theta\in C^1(\R)$ with 
$\theta(0)=0$ and 
$\theta'(z)=0$ for  $z\geq M_\theta$.
\end{definition}

\subsection{The damped continuity equation in variable domains}

Here we present some results concerning the continuity equation
in moving domains. Some of the results are direct consequences of our previous papers \cite{BrSc,BrSc2}, but Lemma \ref{lem:warme2} below is new and contains some improvements.\\
  We assume that the moving boundary is prescribed by a function $\eta:\overline{I}\times Q\rightarrow \Omega$. 
 For a given function $w\in L^2(I;W^{1,2}(\Omega_\eta;\R^n))$ with $\partial_t\eta(t)=w(t)\circ\eta(t)$ in $Q$ for a.a. $t\in I$ and some $\varepsilon>0$ we consider the equation
\begin{align}\label{eq:warme}\begin{aligned}
\partial_t \varrho&+\Div(\varrho w)=\varepsilon\Delta\varrho\quad\text{in}\quad I\times\Omega_\eta,\\
\varrho(0)&=\varrho_0\text{ in }\Omega_{\eta(0)},\quad\partial_{\nu_\eta}\varrho\big|_{\partial\Omega_\eta}=0\quad\text{on}\quad I\times\partial\Omega_\eta.
\end{aligned}
\end{align}
 A weak solution to \eqref{eq:warme} satisfies
\begin{align}
\int_I\frac{\dd}{\dt}\int_{\Omega_\eta}\varrho\psi\dx\dt
-\int_I\int_{\Omega_\eta}\big(\varrho \partial_t\psi +\varrho w\cdot\nabla \psi\big)\dxt=-\int_I\int_{\Omega_\eta}\varepsilon\nabla\varrho\cdot\nabla\psi\dxt\label{eq:weak-masse}
\end{align}
for all $\psi\in C^\infty(\overline{I}\times \R^3)$. We have the following version of \cite[Thm. 3.1]{BrSc}.

\begin{lemma}\label{lem:warme} Let $\eta\in L^2(I;W^{1,\infty}(Q;\R^n))$ be the function describing the boundary. Assume that $w\in L^2(I;W^{1,2}(\Omega_\eta;\R^n))$ with $\partial_t\eta(t)=w(t)\circ\eta(t)$ in $Q$ and $w(t)=0$ in $P$ for a.a. $t\in I$ and
$\varrho_0\in L^{2}(\Omega_{\eta(0)})$.
\begin{enumerate}
\item[a)] There is at most one weak solution $\varrho$ to \eqref{eq:warme} such that
$$\varrho\in L^\infty(I;L^2(\Omega_{\eta}))\cap L^2(I;W^{1,2}(\Omega_{\eta})).$$
\item[b)] Let $\theta\in C^2(\setR_+;\setR_+)$ be such that $\theta'(s)= 0$
 for large values of $s$ and $\theta(0)=0$.
Then the following holds:
\begin{align}
\label{eq:renormz}
\begin{aligned}
\int_I\frac{\dd}{\dt}&\int_{\Omega_{\eta}} \theta(\varrho)\,\psi\dxt-\int_{I\times \Omega_{\eta}}\theta(\varrho)\,\partial_t\psi\dxt
\\
=&-\int_{I\times \Omega_\eta}\big(\varrho\theta'(\varrho)-\theta(\varrho)\big)\Div w\,\psi\dx+\int_{I\times \Omega_{\eta}}\theta(\varrho) w\cdot\nabla\psi\dxt\\ &-\int_{I\times\Omega_{\eta}}\varepsilon\nabla \theta(\varrho)\cdot\nabla\psi\dxt-\int_{I\times \Omega_{\eta}}\varepsilon\theta''(\varrho)|\nabla\varrho|^2\psi\dxt
\end{aligned}
\end{align}
for all $\psi\in C^\infty(\overline{I}\times\R^3)$.
\item[c)] Assume that $\varrho_0\geq0$ a.e. in $\Omega_{\eta(0)}$. Then we have $\varrho\geq0$ a.e. in $I\times\Omega_\eta$. 
\end{enumerate}
\end{lemma}
\begin{proof}
The proof follows by the lines of \cite[Thm. 3.1]{BrSc}, where much stronger conditions on the regularity of $\eta$ and also sightly more on $w$ is assumed. One easily checks that these assumptions are only used there to prove the existence of a solution, which we do not claim here. The statements from a)--c) do not require it. Note that the assumption $\eta(t)\in W^{1,\infty}(Q;\R^n)$ for a.a. $t\in I$ is needed to guarantee the existence of an extension operator $ W^{1,2}(\Omega_\eta;\R^n)\rightarrow W^{1,2}(\R^n;\R^n)$.  The latter is required for the proof of b).
\end{proof}
\begin{lemma}\label{lem:warme2}
Let the assumptions of Lemma \ref{lem:warme} be satisfied and suppose additionally that $w\in L^2(I;W^{k_0,2}(\Omega_\eta))$ and $\eta,\partial_t\eta\in L^2(I;W^{k_0,2}(Q))$ for some $k_0>\frac{n}{2}+2$ and $\eta_0\in W^{2,q}(Q;\R^n)$.
\begin{enumerate}
\item[a)] We have
\begin{align*}
\inf_{\Omega_{\eta(0)}}\varrho_0\exp\bigg(-\int_0^T\|\Div w\|_{L^\infty(\Omega_\eta)}\dt\bigg)  \leq \varrho(t,x)\leq \sup_{\Omega_{\eta(0)}}\varrho_0\exp\bigg(\int_0^T\|\Div w\|_{L^\infty(\Omega_\eta)}\dt\bigg) 
\end{align*}
for a.a. $(t,x)\in I\times\Omega_\eta$.
\item[b)] We have 
$$\varrho\in L^\infty(I;W^{1,2}(\Omega_{\eta}))\cap L^2(I;W^{2,2}(\Omega_{\eta}))$$
and it holds
   \begin{align*}
\sup_{t\in I}\int_{\Omega_\eta}|\nabla\varrho|^2\dx&+\int_I\int_{\Omega_\eta}|\nabla^2\varrho|^2\dx\dt\\&\leq\,c(\varepsilon)\exp\bigg(\int_0^T\big(\|w\|_{L^\infty(\Omega_\eta)}^2+\|\partial_t\eta\|_{L^\infty(Q)}^2\big)\dt\bigg)\\&\qquad\qquad\qquad\qquad\times\bigg(\int_{\Omega_\eta}\big(|\nabla\varrho_0|^2\dx+c\int_0^T\|\Div w\|_{L^\infty(\Omega_\eta)}^2\|\varrho\|_{L^2(\Omega_\eta)}^2\dt\bigg).
\end{align*}
\end{enumerate}
\end{lemma}
\begin{proof}
Let us initially suppose that $\eta,w$ and $\varrho$ are sufficiently smooth and $\varrho$ is stricly positive such that all the following manipulations are justified.\\
Ad~a). Multiplying \eqref{eq:warme} by $\theta'(\varrho)$ shows
\begin{align*}
\partial_t \theta(\varrho)+w\cdot\nabla\theta(\varrho)+\Div w\varrho\theta'(\varrho)=\varepsilon\Div\big(\theta'(\varrho)\nabla\varrho \big)-\varepsilon \theta''(\varrho)|\nabla\varrho|^2
\end{align*}
as well as
\begin{align*}
\frac{\dd}{\dt}\int_{\Omega_\eta} \theta(\varrho)\dx+\int_{\Omega_\eta}(\varrho\theta'(\varrho)-\theta(\varrho))\Div w\dx=-\varepsilon\int_{\Omega_\eta} \theta''(\varrho)|\nabla\varrho|^2\dx
\end{align*}
using Reynold's transport theorem, $\partial_{\nu_\eta}\varrho=0$ and $w(t)\circ\eta(t)=\partial_t\eta(t)$ in $Q$ for a.a. $t\in I$. If $\theta''\geq 0$
and $|\theta'(z)z|\leq c_\theta |\theta(z)|$ for $z\geq0$ we clearly get
\begin{align*}
\frac{\dd}{\dt}\int_{\Omega_\eta} \theta(\varrho)\dx\leq\int_{\Omega_\eta}(\theta(\varrho)-\varrho\theta'(\varrho))\Div w\dx\leq\,(c_\theta+1) \|\Div w\|_{L^\infty(\Omega_\eta)}\int_{\Omega_\eta}\theta(\varrho)\dx.
\end{align*}
Gronwall's lemma yields
\begin{align*}
\int_{\Omega_\eta} \theta(\varrho(t))\dx\leq\exp\Big(C_\theta\int_I \|\Div w\|_{L^\infty(\Omega_\eta)}\dt\Big)\int_{\Omega_\eta} \theta(\varrho_0)\dx,
\end{align*}
where $C_\theta=c_\theta+1$.
Choosing $\theta(z)=z^m$ for $m\gg1$ we get $c_\theta=m$ and hence
\begin{align*}
\bigg(\int_{\Omega_\eta}|\varrho(t)|^m\dx\bigg)^{\frac{1}{m}}\leq\exp\Big(\frac{m+1}{m}\int_I \|\Div w\|_{L^\infty(\Omega_\eta)}\dt\Big)\bigg(\int_{\Omega_\eta} |\varrho_0|^m\dx\bigg)^{\frac{1}{m}}
\end{align*}
for all $t\in I$.
Passing with $m\rightarrow\infty$ we obtain
\begin{align*}
\sup_{t\in I}\|\varrho(t)\|_{L^\infty(\Omega_\eta(t))}\leq\|\varrho_0\|_{L^\infty(\Omega_{\eta_0})}\exp\Big(\int_I \|\Div w\|_{L^\infty(\Omega_\eta)}\dt\Big).
\end{align*}
Similarly, choosing $\theta(z)=z^{-m}$, we have
\begin{align*}
\int_{\Omega_{\eta(t)}} (\varrho(t))^{-m}\dx\leq \exp\Big((m+1)\int_I \|\Div w\|_{L^\infty(\Omega_\eta)}\dt\Big)\int_{\Omega_{\eta_0}} (\varrho_0)^{-m}\dx.
\end{align*}
Taking the $m$-th root shows
\begin{align*}
\bigg(\int_{\Omega_{\eta(t)}} (\varrho(t))\Big)^{-m}\dx\bigg)^{\frac{1}{m}}\leq \exp\Big(\frac{m+1}{m}\int_I \|\Div w\|_{L^\infty(\Omega_\eta)}\dt\Big)\bigg(\int_{\Omega_{\eta_0}} (\varrho_0)^{-m}\dx\bigg)^{\frac{1}{m}}.
\end{align*}
Passing with $m\rightarrow \infty$ implies
\begin{align*}
\sup_{\Omega_{\eta(t)}} \frac{1}{\varrho(t)}\leq \exp\Big(\int_I \|\Div w\|_{L^\infty(\Omega_\eta)}\dt\Big)\sup_{\Omega_{\eta_0}} \frac{1}{\varrho_0}
\end{align*}
or, equivalently,
\begin{align*}
\exp\Big(-\int_I \|\Div w\|_{L^\infty(\Omega_\eta)}\dt\Big)\inf_{\Omega_{\eta_0}} \varrho_0\leq \inf_{\Omega_{\eta(t)}} \varrho(t).
\end{align*}
Ad~(b). Multiplying \eqref{eq:warme} by $\Delta\varrho$ shows
 \begin{align*}
- \partial_t\varrho\,\Delta\varrho+\varepsilon|\Delta\varrho|^2=\nabla\varrho\cdot w\,\Delta\varrho+\Div w\,\varrho\Delta\varrho
 \end{align*}
 as well as
  \begin{align*}
& \quad \quad \frac{1}{2}\frac{\dd}{\dt}\int_{\Omega_\eta}|\nabla\varrho|^2\dx+\varepsilon\int_{\Omega_\eta}|\nabla^2\varrho|^2\dx\\&=\int_{\partial\Omega_\eta}|\nabla\varrho|^2  \partial_t\eta \cdot\nu_\eta \dd x+\int_{\Omega_\eta}\nabla\varrho\cdot w\,\Delta\varrho\dx+\int_{\Omega_\eta}\Div w\,\varrho\Delta\varrho\dx - \varepsilon\int_{\partial \Omega_{\eta}} (\nabla \nu_\eta)\nabla \varrho \cdot  \nabla \varrho \dx\\
&=:\mathrm{I}+\mathrm{II}+\mathrm{III} +\mathrm{IV}
 \end{align*}
 using Reynold's transport theorem and $\partial_{\nu_\eta}\varrho=0$ (note that we used $\int_{\Omega_\eta}|\Delta\varrho|^2\dx=\int_{\Omega_\eta}|\nabla^2\varrho|^2\dx + \int_{\partial \Omega_{\eta}}  (\nabla \nu_\eta)\nabla \varrho \cdot \nabla \varrho \dx $).
 The first and last terms are estimated by
 \begin{align*}
 \mathrm{I}+\mathrm{IV}&\leq \,(c+\|\partial_t\eta\|_{L^\infty(Q)})\|\nabla\varrho\|^2_{L^2(\partial\Omega_\eta)}\leq \,c(1+\|\partial_t\eta\|_{L^\infty(Q)})\|\nabla\varrho\|_{L^2(\Omega_\eta)}\|\nabla^2\varrho\|_{L^2(\Omega_\eta)}\\
 &\leq \,c(\xi)(1+\|\partial_t\eta\|_{L^\infty(Q)})^2\|\nabla\varrho\|_{L^2(\Omega_\eta)}^2+\xi\|\nabla^2\varrho\|_{L^2(\Omega_\eta)}^2,
 \end{align*}
 where $\xi>0$ is arbitrary and where we used the trace theorem which requires\footnote{We even have $\eta\in L^\infty(I;W^{1,\infty}(Q;\R^n))$ by our assumptions and Sobolev's embedding.} $\eta\in L^\infty(I;W^{1,\infty}(Q;\R^n))$ together with an interpolation argument. We also have
  \begin{align*}
 \mathrm{II}
 &\leq \,c(\xi)\|w\|^2_{L^\infty(\Omega_\eta)}\|\nabla\varrho\|_{L^2(\Omega_\eta)}^2+\xi\|\nabla^2\varrho\|_{L^2(\Omega_\eta)}^2,
 \end{align*}
 as well as
   \begin{align*}
 \mathrm{III}
 &\leq \,c(\xi)\|\Div w\|_{L^\infty(\Omega_\eta)}^2\|\varrho\|_{L^2(\Omega_\eta)}^2+\xi\|\nabla^2\varrho\|_{L^2(\Omega_\eta)}^2.
 \end{align*}
 Absorbing the $\xi$-terms and applying Gronwall's lemma proves
   \begin{align*}
\sup_{t\in I}\int_{\Omega_\eta}|\nabla\varrho|^2\dx+\int_I\int_{\Omega_\eta}|\nabla^2\varrho|^2\dx\dt\leq\,c\int_{\Omega_\eta}|\nabla\varrho_0|^2\dx+c\int_0^T\|\Div w\|_{L^\infty(\Omega_\eta)}^2\|\varrho\|_{L^2(\Omega_\eta)}^2\dt,
\end{align*}
where the constant depends on $\int_I\|\Div w\|_{L^2(\Omega_\eta)}^2\dt$, $\int_I\|\partial_t\eta\|_{L^\infty(Q)}^2\dt$ and $\int_I\|w\|_{L^\infty(\Omega_\eta)}^2\dt$.\\
Let us now remove the regularity assumptions on $\eta,w$ and $\varrho$. 
We can regularise $\eta$ and $w$ by smooth approximation as follows. First, we extend $w$ by $\partial_t\eta\circ\eta^{-1}$ to $\Omega$ and regularise $w$ by a standard smooth approximation in space-time. This yields a smooth sequence
$(w)_\xi$ which converges to $w$ in the $L^2(I;W^{2,\infty}(\Omega;\R^n))$-norm. We define 
$(\eta)_\xi$ as the solution to the ODE $w_\xi(t,\eta_\xi(\cdot,x))=\partial_t\eta_\xi(\cdot,x)$ for each given $x\in Q$ with $\eta(0,x)=(\eta_0)_\xi(x)$, where $(\eta_0)_\xi$ is a regularisation of $\eta_0$ in space. The function $(\eta)_\xi(\cdot ,x)$ does indeed exist for all given $x\in Q$ on the interval $[0,T]$ by the Picard-Lindel\"off theorem as $\norm[\infty]{\nabla w_\xi}$ is uniformly bounded (in dependence of $\xi$).
 By its equation one directly deduces that $(\eta)_\xi$ is smooth.
Now \cite[Theorem 3.3]{BrSc2} applies to the regularised problem
and we obtain a solution $\varrho_\xi$ with $\varrho_\xi\in C^1(\overline I\times\overline \Omega_\eta)$ and $\nabla^2\varrho_\xi\in C(\overline I\times\overline \Omega_\eta)$. Also $\varrho_\xi$ is stricly positive. 
One easily checks that the estimates derived above do not depend on $\xi$. It remains to pass to the limit $\xi\rightarrow0$. Since $\eta_0\in W^{2,q}(Q;\R^n)$ by assumption and $q>n$ we have $(\eta_0)_\xi\rightarrow \eta_0$ in $W^{1,\infty}(Q;\R^n)$.
Hence, (using the ODE, the properties of $(w)_\xi$ and Gronwall's lemma) one deduces that $(\eta)_\xi\rightarrow \eta$ as $\xi\rightarrow0$ uniformly in $Q$. Actually, for this purpose convergence
of $w$ in the $L^2(I;W^{1,\infty}(\Omega;\R^n))$ is sufficient. Since we have convergence
in $L^2(I;W^{2,\infty}(\Omega;\R^n))$ we have similarly
$\nabla(\eta)_\xi\rightarrow \nabla\eta$ as $\xi\rightarrow0$ uniformly in $Q$. Also, passing to the limit with the ODE implies that $\partial_t(\eta)_\xi\rightarrow\partial_t\eta $ in $L^2(I;L^\infty(Q;\R^n))$. These convergences are sufficient to pass to the limit in the estimate.
Since \eqref{eq:warme} is linear the limit procedure
$\xi\rightarrow0$ in \eqref{eq:weak-masse} is straightforward and the limit is indeed the unique solution (see Lemma \ref{lem:warme} a)).
\end{proof}


 \section{The time delayed problem}\label{sec:tau}
Following \cite{benesovaVariationalApproachHyperbolic2020}
we begin constructing a solution to a time-delayed problem, where the material derivative in the momentum equation and its solid counterpart are discretised at level $h>0$, but all other variables are already continuous. We initially solve only on the interval $[0,h]$ and then iterate this in the next section to a time-delayed solution on $I$ by decomposing it into intervals $[0,h],[h,2h],\dots$. As it is common in the literature on the compressible Navier--Stokes system we use an artificial diffusion in the continuity equation and an approximate pressure. As it turns out this alone does not yield sufficient regularity to pass to the limit in the time-step $h$, which we are going to do in the next section.
In order to overcome this problem we add additional terms and set
for $\kappa>0$ and $k_0\in\mathbb N$ large enough
\begin{align*}
E_\kappa(\eta)&=E(\eta)+\kappa\|\nabla^{k_0}\eta\|_{L^2(Q)}^2,\quad R_\kappa(\eta,b)=R(\eta,b)+\kappa\|\nabla^{k_0}b\|_{L^2(Q)}^2.
\end{align*}
In the model for the bulk we replace $E$ by $E_\kappa$ and
$R$ by $R_\kappa$ respectively.
Additionally, we add $\kappa\int_{\Omega(t)}|\nabla^{k_0}v|^2\dx$ to the dissipation of the fluid.
Similar terms are also used in \cite{benesovaVariationalApproachHyperbolic2020} with $h$ instead of $\kappa$. Hence there they disappear already the limit $h\rightarrow 0$, simultaneously with turning the differential quotient into the material derivative. In our case this is split into two limit procedures.

  Given suitable initial state $\varrho_0$, $\eta_0$, $\Omega_0 :=\Omega \setminus \eta(Q)$, a previous solid velocity $\zeta: [0,h] \times Q \to \R^n$ and a corresponding quantity $w:[0,h]\times \Omega_0 \to \R^n$ for the fluid\footnote{Due to the dissipation of density in the equation of continuity, we cannot simply work with the previous fluid velocity. Additionally, we need to deal with the transport effects inherent in the Eulerian description of the fluid. As a result, the natural quantity $w(t)$ turns out to be $\sqrt{\varrho(t-h)} v(t-h)$ transported forward along the flow to the domain $\Omega_0$ at time $t=0$. This will become more apparent in the subsequent section, when we extend the problem to $[0,T]$ and perform the limit $h\to 0$. For the purpose of the current section, all terms involving $w$ can be thought of as a given force term.}, we call a triple $(\eta,v,\varrho)$ solution to the time delayed problem 
in $[0,h]$ provided
\begin{itemize}
\item The time delayed momentum equation holds, that is we have
  \begin{align}\label{eq:h1}
  \begin{aligned}
&\inner[Q]{DE_\kappa(\eta(t))}{\phi} -\int_{\Omega(t)} p_\delta(\varrho)\Div b \dx  \\
 &\quad+ \inner[Q]{D_2R_\kappa\left(\eta,\partial_t \eta\right)}{\phi} + \int_{\Omega(t)}\mathbb S(\nabla v):\nabla b \dx + \kappa \int_{\Omega(t)}\nabla^{k_0} v:\nabla^{k_0}b \dx  \\
 &\quad+ \frac{1}{h} \int_Q \varrho_s(\partial_t \eta - \zeta)\, \phi\dy + \frac{1}{h} \int_{\Omega(t)}(\varrho v-\sqrt{\varrho}\sqrt{\det \nabla \Phi^{-1}} w  \circ \Phi^{-1})\cdot b\dx \\
 &\quad= \int_Qf_s\,\phi\dy + \int_{\Omega(t)}\varrho f_f\cdot b\dx
 \end{aligned}
  \end{align}
  for almost any $t\in[0,h]$ and all $\phi\in C_c([0,h];W^{k_0,2}(Q;\R^n))$, $b\in C_c([0,h];W^{k_0,2}_{0}(\Omega;\R^n))$ satisfying
$\phi|_P=0$ and
\begin{align*}
b\circ\eta= \phi\quad\text{and}\quad v\circ\eta=\partial_t\eta\quad\text{in}\quad Q,
\end{align*}
here we have set $\Omega(t)=\Omega\setminus \eta(t,Q)$ and $\Phi:[0,h]\times \Omega_0\rightarrow\Omega(t)$ solves $\partial_t\Phi=v\circ\Phi$ and $\Phi(0,\cdot)=\mathrm{id}$ where $\circ$ is meant with respect to space (i.e. $(v \circ \Phi) (t,x) := v(t,\Phi(t,x))$).
  \item The approximate equation of continuity holds, that is we have
  \begin{align}
   \partial_t \varrho = -\Div(v \varrho) + \varepsilon \Delta \varrho\label{eq:h2}
  \end{align}
  in $(0,h)\times \Omega(t)$ and $\partial_{\nu(t)}\varrho(t)=0$ in $\partial \Omega(t)$ for all $t \in (0,h)$ as well as $\varrho(0)=\varrho_0$.
  \item The energy balance holds in the sense that
  \begin{align}\label{eq:h3}
  \begin{aligned}
E_\kappa&(\eta(t_1)) + U_{\eta}^\delta(\varrho(t_1))  + h\int_0^{t_1}\int_{\Omega(t)}|\nabla^{k_0} v|^2\dxt \\
&+ \int_0^{t_1}  \Big(2R_\kappa\left(\eta,\partial_t \eta\right) + \int_{\Omega(t)}\mathbb S(\nabla v):\nabla v\dx + \varepsilon \int_{\Omega(t)}H_\delta''(\varrho)|\nabla\varrho|^2\dx\Big) \dt \\ 
  & + \int_0^{t_1} \frac{1}{2h}\left[\varrho_s\int_Q|\partial_t \eta|^2\dy + \int_{\Omega(t)} \varrho\abs{v}^2 \dx  \right] \dt\\
 &\leq E_\kappa(\eta_{0}) + U_{\eta_0}^\delta(\varrho_0) + \int_0^{t_1} \frac{1}{2h}\left[\varrho_s\int_Q|\zeta|^2\dy + \int_{\Omega_0} \abs{ w}^2 \dx  \right] \dt \\
 &+  \int_0^{t_1} \left[\int_Q \partial_t \eta \, f_s \dy + \int_{\Omega(t)} \varrho v \cdot f_f \dx\right] \dt
 \end{aligned}
\end{align}
for a.a. $t_1\in(0,h)$, where $U^\delta_\eta(\varrho)=\int_{\Omega_\eta}H_\delta(\varrho)\dx$.
 \end{itemize}
 We now have the following result.
\begin{theorem}\label{thm:h}
Suppose that there are
\begin{align*}
&\zeta\in L^2(0,h;L^2(Q;\R^n)) ,\quad w\in L^2(0,h;L^2(\Omega_0;\R^n)) ,\\
&\eta_0\in\mathcal E \cap W^{k_0,2}(Q;\Omega),\quad \varrho_0\in L^\infty(\Omega_0),\quad \essinf_{\Omega_0}\varrho_0>0,\\
&f_s\in C([0,h];L^2(Q;\R^n)),\quad f_f\in C([0,h];L^2(\Omega;\R^n)).
\end{align*}
Then there is a triple $(\eta,v,\varrho)$ with
\begin{align*}
&\eta \in L^\infty(0,h;\mathcal{E}) \cap W^{1,2}(0,h;W^{k_0,2}(Q;\R^n)),\\
&v\in L^2(0,h;W^{k_0,2}(\Omega;\R^n)), \\
&\varrho \in L^\infty(0,h;L^\beta(\Omega_\eta)) \cap L^2(0,h;W^{2,2}(\Omega_\eta))
\end{align*}
 which solves the time delayed problem in the sense of \eqref{eq:h1}--\eqref{eq:h3}.
\end{theorem}
 
 The rest of this section is devoted to the proof of Theorem \ref{thm:h}. Some aspects are reminiscent to Theorem 4.2 and its predecessors Theorems 2.2 (for the FSI) and 3.5 (for the time delayed terms) in \cite{benesovaVariationalApproachHyperbolic2020} to which we refer when possible. The main effort is to understand to contribution of $\varrho$ and its behaviour.\\
In order to construct a solution we now splitting $[0,h]$ again into small time-steps of length $\tau\ll h$ and discretise in time. For each of these steps we solve a stationary minimisation problem (the iterative problem), prove a discrete version of the energy inequality and finally pass to the limit $\tau\rightarrow0$ with the resulting piecewise constant and affine approximations.
 
 \subsection{The iterative approximation}
 
 Assume that $\tau,h,\kappa,\varepsilon,\delta >0$ are fixed and that the forces $f_f$ and $f_s$ as well as $\zeta_k:= \fint_{\tau k}^{\tau(k+1)} \zeta \dt : Q \to \R^n $ and $w_k:= \fint_{\tau k}^{\tau (k+1)} w\dt: \Omega_0 \to \R^n$ are given. Suppose that
  also $\eta_k:Q\to \Omega$, $\Omega_k := \Omega \setminus \eta_k(Q), \varrho_k: \Omega_k \to \R$ and a diffeomorphism $\Phi_k: \Omega_0 \to \Omega_k$ are known. Then we define $\eta_{k+1}:Q\to \Omega$, $v_{k+1}:\Omega_k \to \R^n$ and $\varrho_{k+1}:\Omega_{k+1} \to \R$ to be a minimizing triple (not necessarily unique) of
 \begin{align} \label{eq:DiscreteMinProblem}
 \begin{aligned}
  (\eta,v,\varrho)&\mapsto E_\kappa(\eta) + U^\delta_\eta(\varrho) + \tau \left[ R_\kappa\left(\eta_k,\tfrac{\eta-\eta_k}{\tau}\right) + \int_{\Omega_k} \frac{1}{2}\mathbb S(\nabla v):\nabla v\dx + \frac{\kappa}{2} \int_{\Omega_k}|\nabla^{k_0} v|^2\dx \right] \\ 
  & + \frac{\tau}{2h}\left[\varrho_s\int_Q \abs{\tfrac{\eta-\eta_k}{\tau} - \zeta_k}^2 + \int_{\Omega_k} \abs{  \sqrt{\varrho_k} v - \sqrt{\det \nabla \Phi_k^{-1}}w_k \circ \Phi_k^{-1} }^2 \dx  \right] \\
  &- \tau \int_Q \frac{\eta-\eta_k}{\tau} \cdot f_s(\tau k) \dy - \tau\int_{\Omega_k} \varrho_k v  \cdot f_f(\tau k) \dx,
  \end{aligned}
 \end{align}
 where we require $\eta \in \mathcal{E}\cap W^{k_0,2}(Q)$, $v \in W^{k_0,2}(\Omega_k;\R^n)$ with $v|_{\partial \Omega} = 0$ and
 subject to the coupling of velocities
 \begin{align}\label{eq:boundaryhtau}
  \frac{\eta-\eta_k}{\tau} = v \circ \eta_k \quad\text{in}\quad M.
 \end{align}
 For the fluid, we require a regularised condition for mass transport, that is $\varrho$ and $v$ are related through
 \begin{align}\label{eq:vrho}
  \varrho \circ \Psi_v = \frac{(\mathrm{id}-\tau\varepsilon \Delta)^{-1}\varrho_k}{\det \nabla\Psi_v}.
 \end{align}
 Here and in the future $\Psi_v := \mathrm{id}+\tau v$ is a helpful shorthand and $(\mathrm{id}-\tau\varepsilon \Delta)^{-1}$ is to be understood as the solution operator to the respective Neumann problem on $\Omega_k$, i.e. $(\mathrm{id}-\tau\varepsilon \Delta)^{-1}\varrho_k$ is the unique function $\tilde{\varrho}_k:\Omega_k \to \R$ solving
 \begin{align}
 \label{eq:rhok}
  \begin{cases} \tilde{\varrho}_k - \tau \varepsilon \Delta \tilde{\varrho}_k &= \varrho_{k} \quad\text{in}\quad \Omega_k, \\
   \partial_\nu \tilde{\varrho}_k &= 0  \quad\text{on}\quad \partial \Omega_k.
  \end{cases}
 \end{align}

Clearly, there is a unique solution $\tilde{\varrho}_k \in W^{2,\beta}(\Omega_k)$ since $\varrho_k$ can be assumed to be in $L^\beta(\Omega_k)$ and $\Omega_k$ is a $C^{1,\alpha}$-domain.\footnote{In fact, we can infer from Lemma \ref{lem:dmcv} that $\tilde{\varrho}_k \in W^{3,\beta}(\Omega_k)$, but note that this cannot be iterated for increasing $k$, as the $\Psi_{v_{k+1}}$ appearing in the definition of $\varrho_{k+1}$ presents an upper limit to regularity.} Furthermore, due to the fact that $v$ uniquely determines $\varrho$ through \eqref{eq:vrho} we can rewrite \eqref{eq:DiscreteMinProblem} as a minimisation problem in $(\eta,v)$ only. Specifically, we can write
  \begin{align}\label{eq:rho-v}
 U^\delta_{\eta_{k+1}}(\varrho) &= \int_{\Omega_{k+1}} H_\delta(\varrho)\dx = \int_{\Omega_k} H_\delta(\varrho(\Psi_{v}(y))) \det \nabla\Psi_{v}(y) \dx  \nonumber \\
 &= \int_{\Omega_k} H_\delta\left(\frac{(\mathrm{id}-\tau\varepsilon \Delta)^{-1}\varrho_k}{ \det (\mathbb{I}+\tau \nabla v)}  \right) \det (\mathbb{I}+\tau \nabla v) \dx =: \tilde{U}_{\eta_k,\varrho_k}^{\delta,\varepsilon}(v). 
 \end{align}
 
Finally, we update $\Phi_k$ to $\Phi_{k+1}$ by setting
\begin{align*}
\Phi_{k+1}=\Psi_{v_{k+1}}\circ\Phi_k.
\end{align*} 

\begin{lemma}\label{lem:dmcv}
Suppose that $\eta_k$ and $\varrho_k$ are given, where
\begin{align*}
&\eta_k\in \mathcal E \cap W^{k_0,2}(Q;\Omega) ,\quad \varrho_k\in L^\beta(\Omega_k), \quad \essinf_{\Omega_k}\varrho_k  >0,\\
&w_k\in L^2(\Omega_0;\R^n), \quad \zeta_k \in L^2(Q;\R^n),\quad f_s\in L^2(Q;\R^n),\quad f_f\in L^2(\Omega_k;\R^n).
\end{align*}
Then the minimisation problem \eqref{eq:DiscreteMinProblem}--\eqref{eq:vrho} has a solution
$(\eta_{k+1},v_{k+1},\varrho_{k+1})$, where
\begin{align*}
\eta_{k+1}\in \mathcal E \cap W^{k_0,2}(Q;\Omega),&\quad v_{k+1}\in W^{k_0,2}(\Omega_k;\R^n),\quad \Omega_{k+1}=\Omega\setminus \eta_{k+1}(Q),
\quad \varrho_{k+1}\in L^{\beta}(\Omega_{k+1}),\\
  \essinf_{\Omega_{k+1}}\varrho_{k+1}  >0,& \quad \inf_{\Omega_k} \det(\mathbb{I}+\tau \nabla v_{k+1})  > 0.
\end{align*}

\end{lemma}
\begin{proof}
We argue by the direct method in the calculus of variation. The functional is clearly well-defined
by the choice of the function spaces. Using Young's inequality to estimate the force terms we can also show that it is bounded from below in each term and inserting $(\eta_k,0,\tilde{\varrho}_k)$ as a candidate shows that the minimiser must have a finite value.
Using \eqref{eq:rho-v} we can rewrite the problem as a minimisation in $(\eta,v)$ only. 
Coercivity in the relevant functional spaces is now obvious, recalling assumption \ref{ass:energy} S4. However, we still need to verify that any limit obeys the lower bounds.

A standard application of the minimum principle\footnote{Here we use the assumption that $\rho_k$ is strictly bounded from above and below. Due to the Neumann boundary conditions, we find that $\tilde{\rho}_k$ satisfies the same upper and lower bounds as $\rho_k$.}  to \eqref{eq:rhok} implies that $\inf_{\Omega_k} \tilde{\varrho}_k > 0$. Since $\tau \norm[\Omega_k]{\nabla^{k_0} v}^2$ is part of the functional, we know that $\norm[C^{\alpha}(\Omega_k)]{\nabla v}$ is uniformly bounded along any minimising sequence. Thus $\det(\mathbb{I}+\tau \nabla v)$ is trivially bounded from above, which gives us $\inf_{\Omega_k} \varrho >0$ for any $\varrho$ of finite energy in the functional.

Additionally, by the growth condition H3 of Lemma \ref{lem:fluidpotential} and \eqref{eq:rho-v} we find the following bound on the determinant
\begin{align*}
 U^\delta_{\eta_{k+1}}(\varrho) = \tilde{U}_{\eta_k,\varrho_k}^{\delta,\varepsilon}(v) \geq c\left( \int_{\Omega_k} \det(\mathbb{I}+\tau\nabla v)^{-\beta}\dx -1 \right)
\end{align*}
 in dependence of the lower bound of $\tilde{\varrho}_k$. Combining this with the $C^{\alpha}$ bound on $\nabla v$ and choosing $\beta$ sufficiently large, we get a nonzero lower bound on $\det(\mathbb{I}+\tau \nabla v)$ by \cite[Prop. 2.24 (S2)]{benesovaVariationalApproachHyperbolic2020}, which ultimately goes back to \cite{HealeyKroemer}. Note that together with the uniform $C^{1,\alpha}$-bounds on $\eta$, this also implies that $\operatorname{id} +\tau v$ is a diffeomorphism up to the boundary of the fluid domain, which in turn implies that there cannot be a collision. 


The final point which needs clarification is lower semi-continuity, which for most terms does not differ from the analysis in the incompressible case and for these we refer to \cite[Propositions 2.13 and 4.3]{benesovaVariationalApproachHyperbolic2020}. For the new term $\tilde{U}_{\varrho_k}^{\delta,\varepsilon}$ however, weak convergence in $W^{k_0,2}(\Omega_k)$ implies strong convergence in $C^1(\Omega_k)$. Since the functional $\tilde{U}_{\varrho_k}^{\delta,\varepsilon}$ is continuous on $C^1(\Omega)$ (as $\det (\mathbb{I}+\tau \nabla v$) is strictly positive for all velocities for which the functional is of finite value) the proof is complete.
\end{proof}

Next we calculate the corresponding Euler-Lagrange equation, which will give us the discrete, time-delayed momentum-balance. Assuming that $(\eta_{k+1},v_{k+1})$ is in the interior of the admissible set (i.e. the solid has no collisions and $\det\nabla \Psi_{v_k} > 0$), we can vary with $(\phi,\frac{b}{\tau})$ such that $b \circ \eta_k = \phi$ in $Q$. 
For the fluid potential we now calculate
\begin{align*}
 & \phantom{{}={}} \delta \tilde{U}_{\eta_k,\varrho_k}^{\delta,\varepsilon}(v_{k+1}) \left( \frac{b}{\tau} \right) = \frac{\dd}{\dd s} \int_{\Omega_k} H_\delta\left(\tfrac{(\mathrm{id}-\tau\varepsilon \Delta)^{-1}\varrho_k}{\det (\mathbb{I}+\tau \nabla v_{k+1} +s \nabla b)}\right) \det (\mathbb{I}+\tau \nabla v_{k+1} +s \nabla b) \dx\bigg|_{s = 0} \\
 &= \int_{\Omega_k} \tr( \nabla b  \cdot \cof (\nabla \Psi_{v_{k+1}})^\top) \left[-\frac{H_\delta'\left(\tfrac{(\mathrm{id}-\tau\varepsilon \Delta)^{-1}\varrho_k}{\det(\nabla \Psi_{v_{k+1}})}\right)}{\det (\nabla \Psi_{v_{k+1}})} (\mathrm{id}-\tau\varepsilon \Delta)^{-1}\varrho_k + H_\delta\left(\tfrac{(\mathrm{id}-\tau\varepsilon \Delta)^{-1}\varrho_k}{\det(\nabla \Psi_{v_{k+1}})}\right)\right] \dx\\
 &=-\int_{\Omega_k}  \tr (\nabla b  \cdot \nabla \Psi_{v_{k+1}}^{-1}) \left[H_\delta'\left(\tfrac{(\mathrm{id}-\tau\varepsilon \Delta)^{-1}\varrho_k}{\det(\nabla \Psi_{v_{k+1}})}\right) \frac{(\mathrm{id}-\tau\varepsilon \Delta)^{-1}\varrho_k}{\det \nabla \Psi_{v_{k+1}}}-  H_\delta\left(\tfrac{(\mathrm{id}-\tau\varepsilon \Delta)^{-1}\varrho_k}{\det(\nabla \Psi_{v_{k+1}})}\right) \right]  \det \nabla \Psi_{v_{k+1}} \dx \\
 &= -\int_{\Omega_{k+1}} \Div  (b \circ (\Psi_{v_{k+1}})^{-1}) \left( H_\delta'(\varrho_{k+1}) \varrho_{k+1}  - H_\delta(\varrho_{k+1}) \right) \dx \\
  &= -\int_{\Omega_{k+1}} \Div  (b \circ (\Psi_{v_{k+1}})^{-1})\,p_\delta(\varrho_{k+1})\dx 
\end{align*} 
The first integral is the discrete version of the pressure term and the second is a dissipation related error that will be shown to vanish in the limit $\tau \to 0$. Furthermore we note that a short calculation reveals 
\begin{align*}
 \delta \left(\frac{\tau}{2}\mathbb S(\nabla v_{k+1}):\nabla v_{k+1} \right) \left(\frac{b}{\tau}\right) = \mathbb S(\nabla v_{k+1}):\nabla b
\end{align*}

Refering
to \cite[Propositions 4.3]{benesovaVariationalApproachHyperbolic2020} for the remaining terms in the 
Euler-Lagrange equation we then conclude:

\begin{corollary}\label{cor:EL}
Suppose that the assumptions of Lemma \ref{lem:dmcv} hold.
Any solution of the minimisation problem \eqref{eq:DiscreteMinProblem}--\eqref{eq:vrho} satisfies
\begin{align} \label{eq:discreteEL}
\begin{aligned}
 &\phantom{{}={}} \inner[Q]{DE_\kappa(\eta_{k+1})}{\phi} -\int_{\Omega_{k+1}} \nabla \cdot  (b \circ (\Psi_{v_{k+1}})^{-1}) p_\delta(\varrho_{k+1}) dy \\
 &+\int_{\Omega_k} \tr (\nabla b  \nabla \Psi_{v_{k+1}}^{-1})  H_\delta'\left(\frac{(\mathrm{id}-\tau\varepsilon \Delta)^{-1}\varrho_k}{\det(\nabla \Psi_{v_{k+1}})}\right) ( (\mathrm{id}-\tau\varepsilon \Delta )^{-1}\varrho_k - \varrho_k) \,dy \\
 &+ \inner[Q]{D_2R_\kappa\left(\eta_{k},\tfrac{\eta_{k+1}-\eta_k}{\tau}\right)}{\phi} + \int_{\Omega_k}\mathbb S(\nabla v_{k+1}):\nabla b\dx + \kappa \int_{\Omega_k}\nabla^{k_0} v_{k+1}:\nabla^{k_0} b\dx \\
 &+ \frac{1}{h} \int_Q \varrho_s\Big(\tfrac{\eta_{k+1}-\eta_k}{\tau} - \zeta_k\Big)\,\phi\dy + \frac{1}{h} \int_{\Omega_k}\left(\varrho_k v_{k+1} - \sqrt{\varrho_k}\sqrt{\det\nabla \Phi_k^{-1}} w_k \circ \Phi_k^{-1}\right)\cdot b \dx \\
 &= \int_Qf_s\cdot \phi\dy + \int_{\Omega_k}\varrho_{k} f_f\cdot b\dx
 \end{aligned}
\end{align}
 for a.a. $t\in[0,h]$ and all $\phi\in W^{k_0,2}(Q;\R^n)$, $b\in W^{k_0,2}_{0}(\Omega;\R^n)$ satisfying
$\phi|_P=0$ and
$b\circ\eta_k=\phi$ in $Q$.
\end{corollary}

 \subsection{The discrete energy inequality}
A key in the analysis is the energy inequality. 
We start with a discrete energy inequality for the solution of the minimisation problem \eqref{eq:DiscreteMinProblem}--\eqref{eq:vrho}.\footnote{Note that this inequality is non-optimal; one would expect an additional factor of two in front of all dissipation terms. Using a more involved argument, one can indeed derive an improved inequality at this point. But since there is no qualitative improvement on the resulting estimates and this inequality will be replaced by a more correct, continuous energy inequality afterwards, there is no need to do so here.}
\begin{lemma}\label{lem:discreteenergy}
Suppose that $\eta_0,v_0$ and $\varrho_0$, as well as $w$, $\zeta$, $f_f$ and $f_s$ are given, where
\begin{align*}
&\eta_0\in \mathcal E \cap W^{k_0,2}(Q;\R^n),\quad v_0\in W^{k_0,2}(\Omega_0;\R^n),\quad  \varrho_0 \in L^\beta(\Omega_0),\\
&\zeta \in L^2([0,h] \times Q;\R^n),  w \in L^2([0,h]\times \Omega_0;\R^n),  f_s \in C([0,h]; L^2(Q;\R^n)),  f_f \in C([0,h]; L^2(\Omega;\R^n)).
\end{align*}
Then the solutions $(\eta_{k},v_{k},\varrho_{k})_{k=1}^N$ to minimisation problem \eqref{eq:DiscreteMinProblem}--\eqref{eq:vrho} satisfy
\begin{align*}
 E_\kappa&(\eta_{N}) + U_{\eta_{N}}^\delta(\varrho_{N})+ \sum_{k=0}^{N-1} \tau\frac{\kappa}{2} \int_{\Omega_k}|\nabla^{k_0} v_{k+1}|^2\dx \nonumber \\
&+ \sum_{k=0}^{N-1} \tau \left[ R_\kappa\left(\eta_k,\tfrac{\eta_{k+1}-\eta_k}{\tau}\right) + \frac{1}{2} \int_{\Omega_k}\mathbb S(\nabla v_{k+1}):\nabla v_{k+1}\dx  +\varepsilon \int_{\Omega_k}H_\delta''(\tilde{\varrho}_k)|\nabla \tilde{\varrho}_k|^2\dx\right] \nonumber \\ 
  & + \sum_{k=0}^{N-1} \frac{\tau}{2h}\left[\varrho_s\int_Q\Big|\tfrac{\eta_{k+1}-\eta_k}{\tau} - \zeta_k\Big|^2\dy + \int_{\Omega_k} \abs{\sqrt{\varrho_k} v_{k+1} - \sqrt{\det\nabla \Phi_k^{-1}}w_k \circ \Phi_k^{-1}}^2 \dx  \right] \nonumber \\
 &\leq E_\kappa(\eta_{0}) + U_{\eta_0}^\delta(\varrho_0) + \sum_{k=0}^{N-1} \frac{\tau}{2h}\left[\varrho_s\int_Q|\zeta_k|^2\dy + \int_{\Omega_0} \abs{ w_k}^2 dx  \right] \nonumber \\ &+  \sum_{k=0}^{N-1} \tau \left[\int_Q \tfrac{\eta_{k+1}-\eta_k}{\tau} \cdot f_s \dy + \int_{\Omega_k} \varrho_k v_{k+1}  \cdot f_f \dx\right] 
\end{align*}
\end{lemma}
\begin{proof}
We compare the value of the actual minimizer in \eqref{eq:DiscreteMinProblem} with the value at $(\eta,v,\varrho) = (\eta_k,0,\tilde\varrho_k)$, where $\tilde{\varrho}_k := (\mathrm{id}-\tau\varepsilon \Delta)^{-1} \varrho_k$. Thus we get
\begin{align}\label{eq:uncorrectedDiscreteEnergyEstimate}
\begin{aligned}
  E_\kappa(\eta_{k+1}) &+ U^{\delta}_{\eta_{k+1}}(\varrho_{k+1}) + \tau \left[ R_\kappa\left(\eta_k,\tfrac{\eta_{k+1}-\eta_k}{\tau}\right) + \frac{1}{2}\int_{\Omega_k}\mathbb S(\nabla v_{k+1}):\nabla v_{k+1}\dx  \right] \\ 
  & + \tau\frac{\kappa}{2} \int_{\Omega_k}|\nabla^{k_0} v_{k+1}|^2\dx+ \frac{\tau}{2h}\left[\varrho_s\int_Q\Big|\tfrac{\eta_{k+1}-\eta_k}{\tau} - \zeta_k\Big|^2\dy\right]\\ &+ \frac{\tau}{2h}\left[ \int_{\Omega_k}  \abs{ \sqrt{\varrho_k} v_{k+1}  - \sqrt{\det \nabla \Phi_k^{-1}} w_k \circ \Phi_k^{-1}}^2 \dx  \right]  \\
  &- \int_Q \tfrac{\eta_{k+1}-\eta_k}{\tau} \cdot f_s \dx - \int_{\Omega_k} \varrho_k v_{k+1}  \cdot f_f \dx \\
 &\leq E_\kappa(\eta_{k}) + U^{\delta}_{\eta_k}(\tilde{\varrho}_k) + \frac{\tau}{2h}\left[\varrho_s\int_Q| \zeta_k|^2\dy + \int_{\Omega_0} \abs{ w_k}^2 \dx \right].
 \end{aligned}
\end{align}
We are now going to estimate the error between $ U^{\delta}_{\eta_k}(\tilde{\varrho}_k)$ and $ U^{\delta}_{\eta_k}(\varrho_k)$. Since $H_\delta$ belongs to $C^2((0,\infty))$ and is convex we have $$H_\delta(\tilde{\varrho}_k - \varepsilon \tau \Delta \tilde{\varrho}_k) \geq H_\delta(\tilde{\varrho}_k) - H_\delta'(\tilde{\varrho}_k)\varepsilon \tau \Delta \tilde{\varrho}_k$$
and thus
\begin{align*}
 U_{\eta_k}^{\delta}(\tilde{\varrho}_k) - U_{\eta_k}^{\delta}(\varrho_k) &= \int_{\Omega_k}\big( H_\delta(\tilde{\varrho}_k) - H_\delta(\varrho_k)\big) \dx = \int_{\Omega_k} \big(H_\delta(\tilde{\varrho}_k) - H_\delta(\tilde{\varrho}_k - \varepsilon \tau \Delta \tilde{\varrho}_k)\big) \dx 
 \\&\leq \int_{\Omega_k} H_\delta'(\tilde{\varrho}_k)\varepsilon \tau \Delta \tilde{\varrho}_k \dx= -\int_{\Omega_k} \varepsilon \tau H_\delta''(\tilde{\varrho}_k) |\nabla \tilde{\varrho}_k|^2 \dx 
\end{align*}
using also \eqref{eq:vrho} and the Neumann boundary condition ofr $\tilde{\varrho}_k$. Plugging this into \eqref{eq:uncorrectedDiscreteEnergyEstimate}
and summing over $k \in \{0,...,N-1\}$ yields the claim.
\end{proof}

 Next we construct piecewise constant and piecewise continuous interpolations of all our quantities by setting
\begin{align*} 
 \bar{\eta}^{(\tau)}(t,y) &= \eta_{k+1}(y) &\text{ for }& \tau k \leq t < \tau (k+1), y\in Q\\
 \eta^{(\tau)}(t,y) &= \eta_k(y) &\text{ for }& \tau k \leq t < \tau (k+1),y\in Q,\\
 \tilde{\eta}^{(\tau)}(t,y) &= \tfrac{\tau (k+1)-t}{\tau} \eta_k(y) + \tfrac{t-\tau k}{\tau} \eta_{k+1}(y) &\text{ for }& \tau k \leq t < \tau (k+1),y\in Q,\\
 v^{(\tau)}(t,x) &= v_{k+1}(x) &\text{ for }& \tau k \leq t < \tau (k+1), x \in \Omega_{k},\\
 v^{(\tau)}(t,x) &= \tfrac{\eta_{k+1}-\eta_k}{\tau}\circ \left(\eta_k\right)^{-1}  &\text{ for }& \tau k \leq t < \tau (k+1), x \in \Omega \setminus \Omega_{k},\\
 \Phi^{(\tau)}(t,x) &= \Phi_{k}(x) &\text{ for }& \tau k \leq t < \tau (k+1), x \in \Omega_0\\
 \bar\varrho^{(\tau)}(t,x) &= \varrho_{k+1}(x) &\text{ for }& \tau k \leq t < \tau (k+1), x \in \Omega_{k},\\ 
 \varrho^{(\tau)}(t,x) &= \varrho_k(x) &\text{ for }& \tau k \leq t < \tau (k+1), x \in \Omega_{k},\\ 
 \tilde{\varrho}^{(\tau)}(t,x) &= \tilde{\varrho}_k(x) &\text{ for }& \tau k \leq t < \tau (k+1), x \in \Omega_{k},\\
 \zeta^{(\tau)}(t,y) &= \zeta_k(y) &\text{ for }& \tau k \leq t < \tau (k+1), y \in Q,\\
 w^{(\tau)}(t,x) &= w_k(x) &\text{ for }& \tau k \leq t < \tau (k+1), x \in \Omega_{0},
\end{align*} 
as well as $\Omega^{(\tau)}(t) = \Omega_{k}$ for $\tau k \leq t < \tau (k+1)$.
Note that from this point on, we extend $v$ to all of $\Omega$, using the corresponding solid velocity, cf. Remark \ref{rem:malte}.

Lemma \ref{lem:discreteenergy} implies the following corollary for the interpolated quantities by setting $t=\tau N$:

\begin{corollary}
Under the assumptions of Lemma \ref{lem:discreteenergy} we have
\begin{align*}
E_\kappa&(\bar{\eta}^{(\tau)}(t_1)) + U_{\eta^{(\tau)}}^\delta(\bar\varrho^{(\tau)}(t_1)) + \frac{\kappa}{2}\int_0^{t_1} \int_{\Omega^{(\tau)}(t)}|\nabla^{k_0} v^{(\tau)}|^2\dx \dt\\
&+ \int_0^{t_1}\bigg[  R_\kappa\left(\eta^{(\tau)},\partial_t \tilde{\eta}^{(\tau)}\right) + \int_{\Omega^{(\tau)}(t)}\frac{1}{2}\mathbb S(\nabla v^{(\tau)}):\nabla v^{(\tau)}\dx  + \varepsilon \int_{\Omega^{(\tau)}(t)}H_\delta''(\tilde{\varrho}^{(\tau)})|\nabla \tilde{\varrho}^{(\tau)}|^2\dx\bigg] \dt \nonumber \\ 
  & + \int_0^{t_1} \frac{1}{2h}\left[\varrho_s\int_Q|\partial_t \tilde{\eta}^{(\tau)} - \zeta^{(\tau)}|^2\dy + \int_{\Omega_k} \big| \sqrt{ \varrho^{(\tau)}} v^{(\tau)}  - \sqrt{\det \nabla (\Phi^{(\tau)})^{-1}}w^{(\tau)}\circ (\Phi^{(\tau)})^{-1}\big|^2 \dx  \right] dt \nonumber \\
 &\leq E_\kappa(\eta_{0}) + U_{\eta_0}^\delta(\varrho_0) + \int_0^{t_1} \frac{1}{2h}\left[\varrho_s\int_Q|\zeta^{(\tau)}|^2\dy + \int_{\Omega_0} | w^{(\tau)}|^2 \dx  \right] \dt \nonumber \\
 &+  \int_0^{t_1} \left[\int_Q \partial_t \tilde{\eta}^{(\tau)} \cdot f_s \dy + \int_{\Omega_k} \varrho^{(\tau)} v^{(\tau)}  \cdot f_f \dx\right] \dt
\end{align*}
for all $t_1 \in \tau \N \cap [0,h]$.
\end{corollary}
Absorbing the forcing terms into the left-hand side by means of Young's inequality we obtain the following estimates
\begin{align} 
\label{apriori1}
 \| \partial_t\tilde\eta^{(\tau)} \|_{L^2([0,h];W^{k_0,2}( Q)) }^2+\sup_{t\in (0,h)}\Big(\|\bar\eta^{(\tau)}\|_{W^{k_0,2}(Q)}^2+\|\eta^{(\tau)}\|_{W^{k_0,2}(Q)}^2+\|\tilde\eta^{(\tau)}\|_{W^{k_0,2}(Q)}^2\Big)\leq c,\\
\label{apriori2}
 \| \nabla v^{(\tau)} \|^2_{L^2((0,h)\times\Omega^{(\tau)}) }+ \| \nabla^{k_0}v^{(\tau)}\|^2_{L^2((0,h)\times\Omega^{(\tau)}) } \leq c,\\
 \label{apriori3}
 \sup_{t\in(0,h)}\Big(\| \bar\varrho^{(\tau)} \|^\beta_{L^\beta(\Omega^{(\tau)}) } + \| \varrho^{(\tau)} \|^\beta_{L^\beta(\Omega^{(\tau)}) }\Big)+ \varepsilon\| (\nabla \tilde{\varrho}^{(\tau)},\nabla (\tilde{\varrho}^{(\tau)})^{\beta/2}) \|^2_{L^2((0,h)\times\Omega^{(\tau)}) }\leq c. 
\end{align}
They are uniform in $\tau$ but may depend on the other parameters $h,\kappa,\varepsilon$ and $\delta$. Note also that \eqref{apriori2} implies
\begin{align} 
\label{apriori4}
\sum_{k=1}^N\tau\|v^{(\tau)}_k\|^2_{C^{1,\alpha}(\Omega_{k-1}) } \leq c
\end{align}
for some $\alpha>0$ using Sobolev's embedding.

 Let us finally deduce uniform bounds for $\det \nabla \Phi_k$. Note the eventhough the general idea stays the same as in \cite[Prop. 4.6]{benesovaVariationalApproachHyperbolic2020}, the estimate is slightly different since we cannot use that $\Div v = 0$ in the compressible regime.

 \begin{lemma}[Bounds on $\Phi$]
  There exists constants $c,C>0$ such that for all small enough $h$ and $\tau$
 \begin{align}\label{apriori5}
 \exp\bigg(- c\int_0^h  \|\nabla v^{(\tau)}\|^n_{L^\infty(\Omega^{(\tau)})}  \dt \bigg) \leq   \det \nabla \Phi_k & \leq \exp\bigg(c\int_0^h  \|\nabla v^{(\tau)}\|^n_{L^\infty(\Omega^{(\tau)})}\dt \bigg),
 \end{align}
 as well as
 \begin{align} \label{apriori6} 
  \|\nabla \Phi_k\|_{L^\infty(\Omega_0)} \leq C \quad \text{ and } \quad \|\nabla^2 \Phi_k\|_{L^\infty(\Omega_0)} \leq C
 \end{align}

 \end{lemma}

 \begin{proof}
We obtain as in \cite[Lemma A.1]{benesovaVariationalApproachHyperbolic2020}
 \begin{align*}
  \det \nabla \Phi_k &= \prod_{i=1}^k \left[ \det (\mathbb{I} + \tau \nabla v_i) \right] \circ \Phi_{i-1} = \prod_{i=1}^k \bigg[ 1+ \sum_{l=1}^n \tau^l M_l(\nabla v_i)   \bigg] \circ \Phi_{i-1},
 \end{align*}
  where $M_l$ are the polynomials of order $l$ stemming from the expansion of the determinant.
  Estimating the arithmetric mean by the geometric mean yields
   \begin{align*}
  \det \nabla \Phi_k 
  &\leq \bigg( 1 + \frac{1}{k} \sum_{i=1}^k \tau \sum_{l=1}^n \tau^{l-1} |M_l(\nabla v_i \circ \Phi_{i-1})| \bigg)^{k}\\
  &\leq \exp\bigg(c\int_0^h \|\nabla v^{(\tau)}\|^n_{L^\infty(\Omega^{(\tau)})} \dt \bigg).
 \end{align*}
In order to get a similar bound form below we use $(1+a)^{-1} \leq 1+2|a|$ for $|a| \ll 1$ as well as uniform boundedness of $\sqrt{\tau}\|\nabla v^{(\tau)}\|_{L^\infty(\Omega_{\eta^{(\tau)}})} $, which follows from restricting \eqref{apriori4} to a single term. We then can infer similarly to the above
 \begin{align}
 \label{eq:detminus1}
 \begin{aligned}
  (\det \nabla \Phi_k)^{-1} &= \prod_{i=1}^k \left[ \det (\mathbb{I} + \tau \nabla v_i) \right]^{-1} \circ \Phi_{i-1} = \prod_{i=1}^k \left[ 1+ \sum_{l=1}^n \tau^l M_l(\nabla v_i)   \right]^{-1} \circ \Phi_{i-1} \\
  &\leq \left( 1 + \frac{2}{k} \sum_{i=1}^k \tau \sum_{l=1}^n \tau^{l-1} |M(\nabla v_i \circ \Phi_{i-1})| \right)^{k}\\
  &\leq \exp\left(c\int_0^h \|\nabla v^{(\tau)}\|^n_{L^\infty(\Omega^{(\tau)})}\dt \right)
  \end{aligned}
 \end{align}
 In total we arrive at \eqref{apriori5}.
 Arguing in the same way, we can show
  \begin{align*}
 \|\nabla \Phi_k\|_{L^\infty(\Omega_0)} & \leq \exp\bigg(c\int_0^h  \|\nabla v^{(\tau)}\|_{L^\infty(\Omega^{(\tau)})}\dt \bigg)\leq \,C
 \end{align*}
uniformly in $k$, $\tau$ and $h$ using also \eqref{apriori2}.\\
Similarly to the above we can also control second order derivative of $\Phi_k$. It holds
 \begin{align*}
  \nabla^2 \Phi_k^{(\tau)} &= \nabla \left(\prod_{i=1}^k (\mathbb{I}+ \tau \nabla v_i) \circ \Phi_i \right) \\
  &=  \sum_{j=1}^k \nabla \left( (\mathbb{I}+\tau \nabla v_j \circ \Phi_j \right) \prod_{i\neq j} (\mathbb{I}+ \tau \nabla v_i) \circ \Phi_i  \\
  &=  \sum_{j=1}^k \tau \nabla^2 v_j \circ \Phi_j  \nabla \Phi_j \prod_{i\neq j} (\mathbb{I}+ \tau \nabla v_i) \circ \Phi_i
  \end{align*}
  such that, using \eqref{apriori6},
  \begin{align}
  \begin{aligned}
   \|\nabla^2 \Phi_k\|_{L^\infty(\Omega_0)} &\leq\,c \sum_{j=1}^k \tau \norm[L^\infty(\Omega_j)]{\nabla^2v_j}\left( 1 + \frac{2}{k} \sum_{i=1}^k \tau \|\nabla v_i\|_{L^\infty(\Omega_i)} \right)^{k}\\
   &\leq\,c \|\nabla^2 v^{(\tau)}\|_{L^2([0,h];L^\infty(\Omega^{(\tau)}))}\exp\bigg(c\int_0^h  \|\nabla v^{(\tau)}\|_{L^\infty(\Omega^{(\tau)})}\dt \bigg)\leq\,C
   \end{aligned}
 \end{align}
uniformly in $k$ by \eqref{apriori2} provided we choose $k_0$ large enough. Again $c$ is independent of $\tau$ and $h$.
\end{proof}

 \subsection{The limit \texorpdfstring{$\tau \to 0$}{tau to 0}}\label{subsec:tau}
Estimates \eqref{apriori1}--\eqref{apriori3} give rise to 
 \begin{align}\label{conv:tau:eat1}
 \bar\eta^{(\tau)}, \,\eta^{(\tau)}, \,\tilde\eta^{(\tau)}&\rightharpoonup^*\eta\quad\text{in}\quad L^\infty(0,h;W^{k_0,2}(Q;\Omega)),\\
\partial_t\tilde\eta^{(\tau)}&\rightharpoonup\partial_t\eta\quad\text{in}\quad L^2(0,h;W^{k_0,2}(Q;\R^n)),\label{conv:tau:eat2}\\
v^{(\tau)}&\rightharpoonup v\quad\text{in}\quad L^2(0,h;W_0^{k_0,2}(\Omega;\R^n)),\label{conv:tau:v}\\
\varrho^{(\tau)}, \bar\varrho^{(\tau)}&\rightharpoonup^{*,\eta} \varrho\quad\text{in}\quad L^\infty(0,h;L^{\beta}(\Omega_\eta)),\label{conv:tau:rho1}\\
\bar\varrho^{(\tau)}&\rightharpoonup^\eta \varrho\quad\text{in}\quad L^2(0,h;W^{1,2}(\Omega_\eta)),
\label{conv:tau:rho2}
 \end{align}
 at least for a (non-relabelled) subsequence. As in \cite[Prop. 2.20]{benesovaVariationalApproachHyperbolic2020} we also have
 \begin{align}
 \tilde\eta^{(\tau)}&\rightarrow \eta\quad\text{in}\quad C^0([0,h];C^{1,\alpha}(Q))\label{conv:tau:eta3}
 \end{align}
 for some $\alpha>0$.
  Furthermore, the definition of $\Phi_k$ and \eqref{apriori5} imply
 \begin{align}
 \label{eq:phidiscrete}
\|\Phi_k-\Phi_{k-1}\|^2_{L^2(\Omega_0)}&=\tau^2\|v_{k}\circ\Phi_{k-1}\|^2_{L^2(\Omega_{0})}\leq c\tau^2\|v_{k}\|^2_{L^2(\Omega_{k-1})}.
 \end{align}
Combining this with \eqref{apriori6} we conclude that
\begin{align}
 \Phi^{(\tau)}&\rightarrow \Phi\quad\text{in}\quad C^0([0,h];C^{1,\alpha}(\Omega_0)),\label{conv:tau:Phi}
 \end{align}
 where the limit $\Phi(t)$ now maps $\Omega_0$ to the limit fluid domain $\Omega(t)$.
 We want to obtain a similar statement for the function $\Psi_{v_k}=\mathrm{id}+\tau v_k$ and write
 $\Psi_{v_k}=\Phi_k\circ\Phi_{k-1}^{-1}$. This motivates the definition
 \begin{align*}
 \Psi^{(\tau)}(t)=\Phi^{(\tau)}(t)\circ \Phi^{\tau}(t-\tau)^{-1}=\mathrm{id}+\tau v^{(\tau)}.
 \end{align*}
  Due to \eqref{apriori5} and \eqref{apriori6}
  we also have
  \begin{align}
 (\Phi^{(\tau)})^{-1}&\rightarrow \Phi^{-1}\quad\text{in}\quad C^0([0,h];C^{\alpha}(\Omega^{(\tau)}))\label{conv:tau:Phi-1}
 \end{align}
  in addition to \eqref{conv:tau:Phi}. Combining \eqref{conv:tau:Phi} and \eqref{conv:tau:Phi-1}
  shows
  \begin{align}
 \Psi^{(\tau)}&\rightarrow \mathrm{id}\quad\text{in}\quad C^0([0,h];C^{\alpha}(\Omega^{(\tau)})).\label{conv:tau:Psi-1}
 \end{align}
  By definition we have $\nabla\Psi^{(\tau)}-I=\tau\nabla v^{(\tau)}$ such that we also have
 \begin{align}\label{conv:tau:Psi} 
 \Psi^{(\tau)}\rightarrow \mathrm{id}\quad\text{in}\quad L^2(0,h;W^{1,2}(\Omega^{(\tau)}))
 \end{align}
 by \eqref{conv:tau:v}.
 The aim is now to pass to the limit in order to obtain 
 \eqref{eq:h1}--\eqref{eq:h3}. As far as the momentum equation is concerned, we rewrite equation \eqref{eq:discreteEL} as
 \begin{align*} 
 &\phantom{{}={}} \int_0^h\inner{DE_\kappa(\bar\eta^{(\tau)})}{\phi^{(\tau)}}\dt -\int_0^{h}\int_{\Omega^{(\tau)}(t+\tau)}  \Div (b \circ (\Psi^{(\tau)})^{-1}) p_\delta(\bar\varrho^{(\tau)}) \dxt \\
 &+ \int_0^h\bigg[\inner{D_2R_\kappa\left(\eta^{(\tau)},\partial_t\tilde\eta^{(\tau)}\right)}{\phi^{(\tau)}} + \int_{\Omega^{(\tau)}}\mathbb S(\nabla v^{(\tau)}):\nabla b\dx + \kappa \int_{\Omega^{(\tau)}}\nabla^{k_0} v^{(\tau)}:\nabla^{k_0} b\dx\bigg]\dt \\
 &+ \frac{1}{h}\int_0^h\int_Q \varrho_s\big(\partial_t\tilde\eta^{(\tau)} - \zeta^{(\tau)}\big)\,\phi^{(\tau)}\dy\dt \\&+ \frac{1}{h}\int_0^h\int_{\Omega^{(\tau)}}\left(\varrho^{(\tau)}v^{(\tau)} - \sqrt{\varrho^{(\tau)}} \sqrt{\det \nabla (\Phi^{(\tau)})^{-1}} w^{(\tau)} \circ (\Phi^{(\tau)})^{-1}\right)\cdot b\dx\dt \\
 &= \int_0^h\int_Qf_s \cdot\phi^{(\tau)}\dy\dt + \int_0^h\int_{\Omega^{(\tau)}}\varrho^{(\tau)} f_f\cdot b\dx\dt.
\end{align*}
Note that due to the coupling condition, which involves $\eta^{(\tau)}$, we cannot pick the same pair of test-functions for all $\tau$. Instead, we fix $\xi \in C^0([0,h];C_0^\infty(\Omega;\R^3))$ and then derive $\phi^{(\tau)} :=  \xi \circ (\eta^{(\tau)})^{-1}$ from there.

We only have to prove that the terms involving the density converge to their correct counterparts, that is
\begin{align}\label{conv:tau:mom1}
&\int_0^{h}\int_{\Omega^{(\tau)}(t+\tau)}  \Div (b \circ (\Psi^{(\tau)})^{-1}) p_\delta(\bar\varrho^{(\tau)}) \dxt\rightarrow \int_0^{h}\int_{\Omega(t)}  \Div b\, p_\delta(\varrho) \dxt,\\
\nonumber
&\frac{1}{h}\int_0^h\int_{\Omega^{(\tau)}(t)}\left(\varrho^{(\tau)}v^{(\tau)} - \sqrt{\varrho^{(\tau)}} \sqrt{\det \nabla (\Phi^{(\tau)})^{-1}} b \circ (\Phi^{(\tau)})^{-1}\right)\cdot b\dx\dt \\
&\qquad\qquad \rightarrow \frac{1}{h}\int_0^h\int_{\Omega(t)}\left(\varrho v - \sqrt{\varrho} \sqrt{\det \nabla (\Phi)^{-1}} b \circ (\Phi)^{-1}\right)\cdot b\dx\dt \label{conv:tau:mom3}\\
&\int_0^h\int_{\Omega^{(\tau)}(t)}\varrho^{(\tau)} f_f\cdot b\dx\dt\rightarrow \int_0^h\int_{\Omega(t)}\varrho f_f\cdot b\dx\dt,\label{conv:tau:mom4}
\end{align}
as $\tau\rightarrow0$.
The limit in the remaining terms can be performed as in \cite[Section 4.1]{benesovaVariationalApproachHyperbolic2020}. The convergence in \eqref{conv:tau:mom4} follows directly from \eqref{conv:tau:rho1} and \eqref{conv:tau:eta3}, whereas \eqref{conv:tau:mom1} 
and \eqref{conv:tau:mom3} require strong convergence of the density. Using \eqref{eq:vrho} and \eqref{eq:rhok} we can write
\begin{align}\label{eq:2203}
\begin{aligned}
\int_{\Omega_k}&\frac{\varrho_{k+1}(x+\tau v_k(x))-\varrho_k(x)}{\tau}\psi(x)\dx\\
&=-\int_{\Omega_k}\varepsilon(\nabla \tilde{\varrho}_{k})(x)\cdot\nabla\psi(x)\dx-\int_{\Omega_k}\frac{1}{\tau}\Big(1-\frac{1}{\det(\nabla \Psi_{v_k})}\Big)\tilde{\varrho}_k(x)\psi(x)\dx
\end{aligned}
\end{align}
 for all $\psi\in W^{1,2}(\Omega_k)$. Now we choose a parabolic cylinder $J\times B$ such that $2B\Subset \Omega_{\eta^{(\tau)}}$ for all $\tau$ small enough. This is possible due to \eqref{conv:tau:eta3}.
 We obtain for $\psi\in W^{1,2}_0(B)$
 \begin{align*}
\int_{B}&\frac{\varrho^{(\tau)}(t+\tau,x)-\varrho^{(\tau)}(t,x)}{\tau}\psi(x)\dx\\
&=-\int_{B}\frac{\varrho^{(\tau)}(t+\tau,x+\tau v^{(\tau)}(x))-\varrho^{(\tau)}(t+\tau,x)}{\tau}\psi(x)\dx\\
&-\int_{B}\varepsilon\nabla \tilde{\varrho}^{(\tau)}(t,x)\cdot\nabla\psi(x)\dx\\&-\int_{B}\frac{1}{\tau}\Big(1-\frac{1}{\det(\nabla \Psi^{(\tau)})}\Big)\tilde{\varrho}^{(\tau)}(t,x)\psi(x)\dx\\
&=-\int_{B}\int_0^1\nabla\varrho^{(\tau)}(t+\tau,x+s\tau v^{(\tau)}(x))\,\dd s\cdot v^{(\tau)}\psi(x)\dx\\
&-\int_{B}\varepsilon\nabla \tilde{\varrho}^{(\tau)}(t,x)\cdot\nabla\psi(x)\dx\\&-\int_{B}\frac{\Div v^{(\tau)}+ o(\tau)}{1+\tau\Div v^{(\tau)}+ o(\tau)}\tilde{\varrho}^{(\tau)}(t,x)\psi(x)\dx.
\end{align*}
Here the quantity $ o(\tau)$ is such that $o(\tau)/\tau$ vanishes in the $L^2(J;L^\infty(B))$-norm as a consequence of \eqref{conv:tau:v}. Using \eqref{eq:vrho} we may deduce from the regularity of $v^{(\tau)}$ in \eqref{apriori1}--\eqref{apriori3} as well as the uniform lower bound for $\det\nabla\Phi_k$ from \eqref{apriori5}) that
\begin{align*}
\varrho^{(\tau)}\in W^{1,2}(I;W^{-1,2}(B))\cap L^2(I;W^{1,2}(B))
\end{align*}
uniformly in $\tau$.
 Combining this with \eqref{conv:tau:rho2}
 we conclude that
 \begin{align}
\begin{aligned}
 \varrho^{(\tau)}&\rightharpoonup\varrho\quad\text{in}\quad W^{1,2}(I;W^{-1,2}(B))\cap L^2(I;W^{1,2}(B)),
 \\
 \text{and} \quad \varrho^{(\tau)}&\rightarrow \varrho\quad\text{in}\quad L^2(J \times B).
 \end{aligned}\label{conv:tau:tilderhostrong} 
 \end{align}
 
 This together with \eqref{conv:tau:Psi} yields
  \begin{align}
 \tilde{\varrho}^{(\tau)},\bar\varrho^{(\tau)}&\rightarrow \varrho\quad\text{in}\quad L^2(J \times B).\label{conv:tau:rhostrong} 
 \end{align}
 Using \eqref{apriori3} and arbitrariness of $J \times B$, the convergences in \eqref{conv:tau:tilderhostrong} and  \eqref{conv:tau:rhostrong} even hold in
 $L^q((0,h)\times\Omega^{(\tau)})$ with $q=\frac{4+3\beta}{3}>\beta$.
 This in combination with \eqref{conv:tau:eta3} and \eqref{conv:tau:Psi} is enough 
 to prove the convergenc \eqref{conv:tau:mom1} 
 Taking into account also \eqref{conv:tau:v}, \eqref{conv:tau:Phi} and \eqref{conv:tau:Phi-1} proves \eqref{conv:tau:mom3} (note also the uniform lower bound for $\det\nabla\Phi_k$ from \eqref{apriori5}). We thus conclude that \eqref{eq:h1} holds.\\
 Now we take a look at the continuity equation. We use a test-function
 $\psi\in C^\infty_c((\tau,h-\tau)\times \Omega)$ and obtain similarly to \eqref{eq:2203}
 \begin{align*}
\int_0^h\int_{\Omega^{(\tau)}}&\frac{\tilde{\varrho}^{(\tau)}(t+\tau)\circ\Psi^{(\tau)}-\tilde{\varrho}^{(\tau)}(t)}{\tau}\psi(t+\tau)\circ\Psi^{(\tau)}\dx\dt\\
&=-\int_0^h\int_{\Omega^{(\tau)}}\varepsilon\nabla \tilde{\varrho}^{(\tau)}\circ\Psi^{(\tau)}\cdot\nabla(\psi(t+\tau)\circ\Psi^{(\tau)})\dxt\\&-\int_0^h\int_{\Omega^{(\tau)}}\frac{1}{\tau}\Big(1-\frac{1}{\det(\nabla \Psi^{(\tau)})}\Big)\tilde{\varrho}^{(\tau)}\psi(t+\tau)\circ\Psi^{(\tau)}\dxt,
\end{align*}
which we denote by $\mathrm{I}=\mathrm{II}+\mathrm{III}$.
The term $\mathrm{I}$ on the left-hand side can be rewritten as
 \begin{align*}
\mathrm{I}&=\int_0^h\int_{\Omega^{(\tau)}}\tilde{\varrho}^{(\tau)}(t)\frac{\psi(t)-\psi(t+\tau)\circ\Psi^{(\tau)}}{\tau}\dx\dt\\
&-\int_0^h\int_{\Omega^{(\tau)}}\frac{\Div v^{(\tau)}(t+\tau)+o(\tau)}{1+\tau\Div v^{(\tau}(t+\tau))+ o(\tau)}\tilde{\varrho}^{(\tau)}(t)\psi(t)\dx\dt\\
&=:\mathrm{I}_1+\mathrm{I}_2
\end{align*}
with $o(\tau)$ as above.
Due to \eqref{conv:tau:Phi}, \eqref{conv:tau:tilderhostrong}, the smoothness of $\psi$ and \eqref{eq:phidiscrete} we find
 \begin{align*}
\mathrm{I}_1\rightarrow\int_0^h\int_{\Omega(t)}\varrho\big(\partial_t\psi+v\cdot\nabla\psi\big)\dx\dt.
\end{align*}
On the other hand,
 \begin{align*}
\mathrm{I}_2&\rightarrow -\int_0^h\int_{\Omega(t)}\Div v(t)\varrho(t)\psi(t)\dx\dt
\end{align*}
using \eqref{conv:tau:v}, \eqref{conv:tau:eta3} and \eqref{conv:tau:tilderhostrong}. Similarly,
 \begin{align*}
\mathrm{III}&\rightarrow -\int_0^h\int_{\Omega(t)}\Div v(t)\varrho(t)\psi(t)\dx\dt
\end{align*}
which cancels with $I_2$. Finally, we have
 \begin{align*}
 \mathrm{II}&\rightarrow-\int_0^h\int_{\Omega(t)}\varepsilon\nabla {\varrho}\cdot\nabla\psi\dxt,
 \end{align*}
 which yields
 \begin{align}\label{lim:continuity}
-\int_0^h\int_{\Omega(t)}\varrho\big(\partial_t\psi+v\cdot\nabla\psi\big)\dx\dt =-\int_0^h\int_{\Omega(t)}\varepsilon\nabla {\varrho}\cdot\nabla\psi\dxt
 \end{align}
 for all $\psi\in C^\infty_c((0,h)\times\Omega)$. We have shown \eqref{eq:h2}.
 
 Finally, for the energy inequality \eqref{eq:h3}, we note that we cannot simply pass to the limit in Lemma \ref{lem:discreteenergy}, as this is missing a factor 2 in front of the dissipation terms. Instead, we obtain it as in \cite[Lemma 4.8]{benesovaVariationalApproachHyperbolic2020} by testing \eqref{eq:h1} with $(\partial_t \eta,v)$. The only substantial addition here is the pressure term which reads as
\begin{align*} 
-\int_{\Omega(t)} \Div  v (H_\delta'(\varrho) \varrho - H_\delta(\varrho)) \dy &= \int_{\Omega(t)}\big( -\Div (v \varrho) H_\delta'(\varrho) + v\cdot \nabla \varrho H_\delta'(\varrho) + \Div v H_\delta(\varrho)\big) \dy \\
&= \int_{\Omega(t)}\big( \partial_t \varrho H_\delta'(\varrho)  + \Div (v H_\delta(\varrho)) \big) \dy  - \int_{\Omega(t)} \varepsilon \Delta \varrho H_\delta'(\varrho) \dy \\
&=\frac{\dd}{\dt} \int_{\Omega(t)} H_\delta(\varrho) \dy+ \varepsilon \int_{\Omega(t)} H_\delta''(\varrho) \abs{\nabla \varrho}^2 \dy
\end{align*}
using \eqref{eq:h2} and Reynold's transport theorem (due to Lemma \ref{lem:warme2} (b) the density is smooth enough to rigorously perform these computations).
For the inertial term we now have
\begin{align*}
 &\phantom{{}={}}\int_{\Omega(t)} \left( \varrho(t) v(t) - \sqrt{\varrho(t)} \sqrt{\det \nabla \Phi^{-1} } w(t) \circ \Phi^{-1} \right) \cdot v(t) \dx \\
 &= \int_{\Omega(t)}\Big(  \varrho(t) \abs{v(t)}^2 - \sqrt{\det \nabla \Phi^{-1} } w(t) \circ \Phi^{-1} \cdot \sqrt{\varrho(t)} v(t) \Big)\dx\\
 & \geq \int_{\Omega(t)}  \frac{\varrho(t)}{2} \abs{v(t)}^2 - \frac{\det \nabla \Phi^{-1}}{2} \abs{w(t) \circ \Phi^{-1}}^2\ \dx = \int_{\Omega(t)}  \varrho(t) \frac{\abs{v(t)}^2}{2} \dx - \int_{\Omega_0}   \frac{\abs{w(t)}^2}{2} \dx
\end{align*}
by Young's inequality and a change of variables. This also shows the need for the factor $\sqrt{\det\nabla \Phi^{-1}}$ in the equations.

 \section{Inertial problem (\texorpdfstring{$h \to 0$}{h to 0})}\label{sec:h}
The aim of the present section is to pass to the limit $h\rightarrow 0$ in \eqref{eq:h1}-\eqref{eq:h3} and to obtain a solution to the regularised system (where $\kappa,\varepsilon$ and $\delta$ are fixed). We beginn with a definition of the latter.

We introduce the function spaces
\begin{align*}
Y^I_{k_0}&:=\{\zeta\in W^{1,2}(I;W^{k_0,2}(Q;\R^n))\cap L^\infty(I;\mathcal{E})\,\},\\
X_{\eta,k_0}^I&:=L^2(I;W^{k_0,2}(\Omega_{\eta};\R^n)),\\
\hat{Z}_{\eta}^I&:=C_w(I;W^{1,2}(\Omega_{\eta};\R^n) \cap L^\beta(\Omega_\eta)),
\end{align*}
which replace the spaces $Y^I$, $X_\eta^I$ and $Z_\eta^I$ (defined in Section \ref{sec:weak}) on the $\kappa$ and $\varepsilon$-level respectively.

A weak solution to the regularised system is a triple $(\eta,v,\varrho)\in Y^I_{k_0}\times X_{\eta,k_0}^I\times \hat{Z}_{\eta,\varepsilon}^I$ that satisfies the following.
\begin{itemize}
\item\label{E1} The momentum equation holds in the sense that\begin{align}\label{eq:mom:kappa}
\begin{aligned}
&\int_I\frac{\dd}{\dt}\int_{\Omega_{ \eta}}\varrho v \cdot b\dx-\int_{\Omega_{\eta}} \Big(\varrho v\cdot \partial_t b +\varrho v\otimes v:\nabla b\Big)\dxt+\frac{\varepsilon}{2}\int_I\int_{\Omega(t)}\nabla\varrho\cdot(\nabla v b+\nabla b v)\dxt
\\
&+\int_I\int_{\Omega_\eta}\mathbb S(\nabla v):\nabla b \dxt+\kappa\int_I\int_{\Omega_\eta}\nabla^{k_0} v:\nabla^{k_0}b \dxt-\int_I\int_{\Omega_{ \eta }}
p_\delta(\varrho)\,\Div b\dxt\\
&+\int_I\bigg(-\int_Q\varrho_s \partial_t\eta\,\partial_t \phi\dy + \langle DE_\kappa(\eta),\phi\rangle+\langle D_2R_\kappa(\eta,\partial_t\eta),\phi\rangle\bigg)\dt
\\&=\int_I\int_{\Omega_{\eta}}\varrho f_f\cdot b\dxt+\int_I\int_Q f_s\cdot \phi\,\dd x\dt
\end{aligned}
\end{align} 
for all $(\phi,b)\in W^{1,2}(I; W^{k_0,2}(Q;\R^n) )\times C_c^\infty(\overline{I}\times \Omega; \R^n)$ with $b(t)\circ \eta(t) = \phi(t)$ in $Q$ and $\phi(t)=0$ on $P$. Moreover, we have $(\varrho v)(0)=q_0$, $\eta(0)=\eta_0$ and $\partial_t\eta(0)=\eta_1$ we well as $\partial_t\eta(t)=v(t)\circ\eta(t)$ in $Q$, $\eta(t)\in\mathcal E$ and $v(t)=0$ on $\partial\Omega$ for a.a. $t\in I$.
\item\label{E2}  The continuity equation holds in the sense that
\begin{align}\label{eq:con:kappa} 
\begin{aligned}
\partial_t\varrho+\Div(\varrho v)=\varepsilon\Delta \varrho
\end{aligned}
\end{align}
holds in $I\times\Omega_\eta$ and we have $\varrho(0)=\varrho_0$ as well as $\partial_{\nu_\eta}\varrho=0$. 
\item \label{E3} The energy inequality is satisfied in the sense that
\begin{align} \label{eq:ene:kappa}
\begin{aligned}
- \int_I &\partial_t \psi \,
\mathscr E_{\delta,\kappa} \dt+\int_I\psi\int_{\Omega_\eta}\Big(\mathbb S(\nabla v):\nabla v+\kappa|\nabla^{k_0} v|^2\Big)\dxt\\
&+2\int_I\psi R_\kappa(\eta,\partial_t\eta)\ds +\varepsilon\int_I\psi\int_{\Omega_\eta}H_\delta''(\varrho)|\nabla\varrho|^2\dxt\\
&\leq
\psi(0) \mathscr E_\delta(0)+\int_I\int_{\Omega_{\eta}}\varrho f_f\cdot v\dxt+\int_I\psi\int_Q f_s\,\partial_t\eta\,\dd y\dt
\end{aligned}
\end{align}
holds for any $\psi \in C^\infty_c([0, T))$.
Here, we abbreviated
$$\mathscr E_{\delta,\kappa}(t)= \int_{\Omega_\eta(t)}\Big(\frac{1}{2} \varrho(t) | {v}(t) |^2 + H_\delta(\varrho(t))\Big)\dx+\int_Q\varrho_s\frac{|\partial_t\eta|^2}{2}\dy+ E_\kappa(\eta(t)).$$
\end{itemize}

\begin{theorem}\label{thm:kappa}
Assume that we have for some $\alpha\in(0,1)$
\begin{align*}
\frac{|q_0|^2}{\varrho_0}&\in L^1(\Omega_{\eta_0}),\ \varrho_0\in C^{2,\alpha}(\overline\Omega_{\eta_0}), \ \eta_0\in \mathcal E,\ \eta_1\in L^2(Q;\R^n),\\
f_f&\in C([0,T];L^2(\R^n;\R^n))\cap L^2(I;L^\infty(\R^3;\R^n)),\ f_s\in L^2(I\times Q).
\end{align*}
Furthermore suppose that $\varrho_0$ is strictly positive. Then there is a solution $(\eta,v,\varrho)\in Y^I_{k_0}\times X^I_{\eta,k_0}\times \hat{Z}_{\eta}^I$ to \eqref{eq:mom:kappa}--\eqref{eq:ene:kappa}. 
Here, we have $I=(0,T)$, where $T \in (0, \infty)$ is arbitrary.
\end{theorem}

\subsection{A priori analysis} \label{subsec:h-construction}
We proceed as follows. First we need to produce an $h$-approximation on the whole interval $I$. For $h\ll 1$ we decompose the interval $I$ into subintervals $(0,h),(h,2h),\dots$. For any given $h$ we obtain
from Theorem \ref{thm:h} the existence of a solution to \eqref{eq:h1}--\eqref{eq:h3} in $(0,h)$. We will then use the resulting $\eta(h)$, $\varrho(h)$, $\partial_t \eta$ and a suitably modified version of $v \circ \Phi(t) \circ \Phi(h)^{-1}$ as $\eta_0$, $\varrho_0$ $\zeta$ and $w$ on $(h,2h)$. We can prove that these are valid initial data, because of the energy inequality \eqref{eq:h3}. We repeat this procedure on the following time-intervals to get a global solution
$(\eta_h,v_h,\varrho_h)$ which solves a variant of \eqref{eq:h1} on $I$.

Specifically, given a solution $(\eta_l^{(h)},v_l^{(h)},\varrho_l^{(h)},\Phi_l^{(h)})$ to \eqref{eq:h1}--\eqref{eq:h3} on the interval $[0,h]$ (note that later this will correspond to $[(l-1)h,lh]$ in $I$, but for now, we will consider each of these intervals as $[0,h]$ and distinguish them by the index $l$), we apply Theorem \ref{thm:h} with 
\begin{align*}
 \eta_l^{(h)}(h) &\text{ as  }\eta_0&\\
 \varrho_l^{(h)}(h) &\text{ as }\varrho_0 &&\text{ (defined on $\Omega\setminus \eta_l^{(h)}(h,Q) = \Omega \setminus \eta_0(Q) =: \Omega_0$)}\\
 \partial_t \eta_l^{(h)}(t) & \text{ as } \zeta(t) && \text{ for all } t \in [0,h]
\end{align*}
and \begin{align*}
 \left( \sqrt{\varrho_l^{(h)}(t) } v_l^{(h)}(t)\right) \circ \Phi_l^{(h)}(t) \circ \Phi_l^{(h)}(h)^{-1} \sqrt{\det\nabla (\Phi_l^{(h)}(t) \circ \Phi_l^{(h)}(h)^{-1}) }
\end{align*}
 as $w(t)$  for all $t \in [0,h]$ (defined on $\Omega\setminus \eta_l^{(h)}(h)(Q) = \Omega_0$). We also write as usual $\Omega_l^{(h)}(t) := \Omega \setminus \eta_l^{(h)}(t,Q)$.
 
These assignements are obvious, except for $w$. Here one should keep in mind that $w(t)$ needs to be defined on the new initial fluid domain $\Omega_0$, but needs to correspond to the quantity $\sqrt{\varrho} v$ for the matching fluid particle at an earlier time. We thus employ the flow map to move it there. As this flow map was only defined with respect to the previous reference configuration, we take a slight detour there, resulting in $\Phi(t) \circ \Phi(h)^{-1}$ and finally we need to correct the distortion due to the extra term $\varepsilon \Delta \varrho$ in the continuity equation such that, in particular,
\begin{align*}
\int_{\Omega_0} \abs{w(t)}^2 \dx = \int_{\Omega_l^{(h)}(t)} \varrho_l^{(h)}(t) \abs{v_l^{(h)}}^2 \dx
\end{align*}

From the energy inequality \eqref{eq:h3} applied to the previous solution, we can conclude that our new initial data fulfills the conditions for Theorem \ref{thm:h}. The solutions constructed by this theorem will then be denoted by $(\eta_{l+1}^{(h)},v_{l+1}^{(h)},\varrho_{l+1}^{(h)},\Phi_{l+1}^{(h)})$.

With these solutions in hand, we can now construct an $h$-approximation on the whole interval $I$. Specifically, we set
\begin{align*}
\eta^{(h)}(t) &:= \eta_l^{(h)}(t-(l-1)h) & \text{ for } & t \in [(l-1)h,lh] \\
v^{(h)}(t) &:= v_l^{(h)}(t-(l-1)h) & \text{ for } & t \in [(l-1)h,lh] \\
\varrho^{(h)}(t) &:= \varrho_l^{(h)}(t-(l-1)h) & \text{ for } & t \in [(l-1)h,lh] \\
\Omega^{(h)}(t) &:= \Omega^{(h)}_l(t-(l-1)h) = \Omega \setminus \eta^{(h)}(t,Q) & \text{ for } & t \in [(l-1)h,lh] 
\intertext{ as well as a redefined flow map for $t \in [(l-1)h,lh]$}
\Phi_s^{(h)}(t) &:= \Phi_l^{(h)}(s+t-(l-1)h) \circ \Phi_l^{(h)}(t-(l-1)h)^{-1} & \text{ for } &t+s \in [(l-1)h,lh] \\
\Phi_s^{(h)}(t) &:= \Phi_{l+1}^{(h)}(s+t-lh) \circ \Phi_l^{(h)}(h) \circ \Phi_l^{(h)}(t-(l-1)h)^{-1} & \text{ for } &t+s \in [lh,(l+1)h] \\
\Phi_s^{(h)}(t) &:= \Phi_{l-1}^{(h)}(s+t-(l-2)h) \circ \Phi_{l-1}^{(h)}(h) \circ \Phi_{l-1}^{(h)}(t-(l-1)h)^{-1} & \text{ for } &t+s \in [(l-2)h,(l-1)h] 
\end{align*}
and so on. Here the new flow map needs to be read in the style of semi-groups, i.e. $\Phi_s^{(h)}(t)$ maps the fluid domain at time $t$ to the fluid domain at time $t+s$, in such a way that we follow the fluid flow. Note that in concert with the uniform (only $\kappa$-dependent) bounds on $\det \nabla \Phi(t)$ derived in the last section in \eqref{apriori5} and the $C^1$-regularity of $\eta$, this also implies that $\Phi(t)$ is a diffeomorphism up to the boundary and thus there cannot be a collision.

If we now translate \eqref{eq:h1}--\eqref{eq:h3} into this new notation, then we have that the $h$-approximation fulfills the following:
\begin{itemize}
 \item The momentum equation \eqref{eq:h1} translates to
  \begin{align}\label{eq:hlong1}
  \begin{aligned}
&\qquad\inner{DE_\kappa(\eta^{(h)}(t))}{\phi} -\int_{\Omega^{(h)}(t)} p_\delta(\varrho^{(h)})\Div b \dx +\inner{D_2R_\kappa\left(\eta^{(h)},\partial_t \eta^{(h)}\right)}{\phi}  \\
 &+  \int_{\Omega^{(h)}(t)}\mathbb S(\nabla v^{(h)}):\nabla b\dx + \kappa \int_{\Omega^{(h)}(t)}\nabla^{k_0} v^{(h)}:\nabla^{k_0}b \dx  +  \int_Q \varrho_s\tfrac{\partial_t \eta^{(h)}(t) - \partial_t \eta^{(h)}(t-h)}{h}\cdot\phi\dy \\
 &+ \frac{1}{h} \int_{\Omega^{(h)}(t)}\left(\varrho^{(h)}(t)v^{(h)}(t)-\sqrt{\varrho^{(h)}(t) \varrho^{(h)}(t-h)\circ\Phi^{(h)}_{-h}\det \nabla \Phi^{(h)}_{-h}}v^{(h)}(t-h)\circ\Phi^{(h)}_{-h}\right)\cdot b\dx \\
 &= \int_Qf_s\cdot\phi\dy + \int_{\Omega^{(h)}(t)}\varrho f_f\cdot b\dx
 \end{aligned}
  \end{align}
\item The continuity equation \eqref{eq:h2} holds unchanged, i.e.
  \begin{align}
   \partial_t \varrho^{(h)} = -\Div(v^{(h)} \varrho^{(h)}) + \varepsilon \Delta \varrho^{(h)}\label{eq:hlong2}
  \end{align}
  in $I\times \Omega_\eta$ and $\partial_{\nu^{(h)}(t)}\varrho^{(h)}(t)=0$ in $\partial \Omega^{(h)}(t)$ for all $t \in (0,h)$ as well as $\varrho^{(h)}(0)=\varrho_0$.
  \item the energy balance holds in the sense that
  \begin{align}\label{eq:hlong3}
  \begin{aligned}
&\qquad E_\kappa(\eta^{(h)}(t)) + U_{\eta}^\delta(\varrho^{(h)})  + \kappa\int_0^{t_1}\int_{\Omega^{(h)}(t)}|\nabla^{k_0} v^{(h)}|^2\dxt \\
&+ \int_0^{t_1}  \Big(2R_\kappa\left(\eta^{(h)},\partial_t \eta^{(h)}\right) + \int_{\Omega^{(h)}(t)}\mathbb S(\nabla v^{(h)}):\nabla v^{(h)}\dx + \varepsilon \int_{\Omega^{(h)}(t)}H_\delta''(\varrho^{(h)})|\nabla\varrho^{(h)}|^2\dx\Big) \dt \\ 
  & + \int_{t_1-h}^{t_1} \frac{1}{2h}\left[\varrho_s\int_Q|\partial_t \eta^{(h)}|^2\dy + \int_{\Omega^{(h)}(t)} \varrho^{(h)} \abs{v^{(h)}}^2 \dx  \right] \dt\\
 &\leq E_\kappa(\eta_{0}) + U_{\eta_0}^\delta(\varrho_0) + \frac{1}{2}\left[\varrho_s\int_Q| \eta_1|^2\dy + \int_{\Omega_0} \varrho_0 \abs{v_0 }^2 \dx  \right] \dt \\
 &+  \int_0^{t_1} \left[\int_Q \partial_t \eta^{(h)} \, f_s \dy + \int_{\Omega(t)} \varrho^{(h)} v^{(h)} \cdot f_f \dx\right] \dt
 \end{aligned}
\end{align}
for a.a. $t_1\in I$.
\end{itemize}

Most of this is a straigthforward replacement of the new definitions together with a telescope argument for the energy balances. Note that in the intertial term of the momentum-equation, the different flow maps from the equation and the definition of $w$ combine to $\Phi_h$, as do their Jacobian determinants.

Additionally, we recover the expected properties of the flow map:
\begin{corollary} \label{cor:Phi_hProperties}
Let the assumptions of Theorem \ref{thm:kappa} be valid.
 For any $t,t+s \in [0,T]$, $\Phi^{(h)}_s(t)$ is a diffeomorphism between $\Omega^{(h)}(t)$ and $\Omega^{(h)}(t+s)$ such that $\Phi^{(h)}_0 = \mathrm{id}$ and
\begin{align*}
 \partial_s \Phi^{(h)}_s(t) &= v^{(h)}(t+s) \circ \Phi^{(h)}_s(t).
\end{align*}
As a consequence, we have 
\begin{align*}
 \partial_s \det \nabla \Phi_s^{(h)} = \Div v^{(h)}(t+s) \circ \Phi^{(h)}_s \det \nabla \Phi_s^{(h)},
\end{align*}
which, in particular, implies
\begin{align*}
 \partial_s \left(\varrho^{(h)}(t+s) \circ \Phi^{(h)}_s(t)\right) &= \left(\varepsilon \Delta \varrho^{(h)}(t+s) -\varrho^{(h)}(t+s) \Div v^{(h)}(t+s)\right)\circ \Phi^{(h)}_s(t)\\
 \partial_s \left(\varrho^{(h)}(t+s) \circ \Phi^{(h)}_s(t) \det \nabla \Phi_s^{(h)}(t) \right) &= \varepsilon \Delta \varrho^{(h)}(t+s)\circ \Phi^{(h)}_s(t) \det \nabla \Phi_s^{(h)}(t).
\end{align*}
\end{corollary}

\begin{proof}
The first set of assertions follows directly from the definition and the properties of the short-time solutions. Additionally, we can calculate
\begin{align*}
 \partial_s \det \nabla \Phi_s^{(h)}(t) &= \tr( \nabla \partial_s \Phi_s^{(h)}(t) \cof \nabla \Phi_s^{(h)}(t)) = \tr(\nabla (v^{(h)}(t+s) \circ \Phi_s^{(h)}(t)) (\nabla \Phi_s^{(h)}(t))^{-1}) \det \nabla \Phi_s^{(h)}(t) \\
 &= \tr(\nabla v^{(h)})(t+s) \circ \Phi_s^{(h)}(t) \det \nabla \Phi_s^{(h)}(t) = \Div v^{(h)}(t+s) \circ \Phi^{(h)}_s(t) \det \nabla \Phi_s^{(h)}(t)
\end{align*}
as well as
\begin{align*}
 \partial_s &\left(\varrho^{(h)}(t+s) \circ \Phi^{(h)}_s(t)\right)\\ &= \partial_t \varrho^{(h)}(t+s) \circ \Phi^{(h)}_s(t) + (\nabla \varrho^{(h)}(t+s)) \circ \Phi^{(h)}_s(t) \cdot (\partial_s  \Phi^{(h)}_s(t)) \\
 &=\left(\varepsilon \Delta \varrho^{(h)}(t+s) - \Div(\varrho^{(h)} v^{(h)}) (t+s)+ v^{(h)}(t+s) \cdot \nabla \varrho^{(h)}(t+s) \right) \circ \Phi^{(h)}_s(t) \\
 &=\left(\varepsilon \Delta \varrho^{(h)}(t+s) - \varrho^{(h)} (t+s)\Div v^{(h)}(t+s) \right) \circ \Phi^{(h)}_s(t)\end{align*}
using the contunuity equation. Combining the two relations above results in the final assertion.
\end{proof}
%

From \eqref{eq:hlong3} we obtain the following uniform bounds:
\begin{align} 
\label{est:h1}
 \| \partial_t\nabla\eta^{(h)} \|_{L^2(I\times Q) }^2+\sup_{t\in I}\|\eta^{(h)}\|_{W^{2,q}(Q)}^q\leq c,\\
\label{est:h2}
\sup_{t \in I} \| \varrho^{(h)} \|^\beta_{L^\beta(\Omega^{(h)})}    \leq c,\\
 \label{est:h3}
 \| \nabla v^{(h)} \|^2_{L^2(I\times \Omega^{(h)}) }+\| \nabla \varrho^{(h)} \|^2_{L^2(I\times\Omega^{(h)}) } + \| \nabla (\varrho^{(h)})^{\beta/2} \|^2_{L^2(I\times\Omega^{(h)}) } \leq c,\\
  \label{est:h4}
\| v^{(h)} \|^2_{L^2(I;W^{k_0,2}(\Omega^{(h)})) }+\sup_{t\in I}\| \eta^{(h)} \|^2_{W^{k_0,2}(Q) }+\| \partial_t\eta^{(h)} \|^2_{L^2(I;W^{k_0,2}(Q)) } \leq c,
\end{align}
which are uniform in $h$. As a consequence we have the following convergences
 for some $\alpha\in (0,1)$ 
\begin{align}
\eta^{(h)}&\rightharpoonup^\ast\eta\quad\text{in}\quad L^\infty(I;W^{k_0,2}(Q;\Omega))\label{eqh:conveta1},\\
\eta^{(h)}&\rightharpoonup\eta\quad\text{in}\quad W^{1,2}(I;W^{k_0,2}(Q;\R^n)),
\label{eqh:conetat2}\\
\eta^{(h)}&\to\eta\quad\text{in}\quad C^\alpha(\overline I\times Q;\Omega)),
\label{eqh:conetatal}\\
v^{(h)}&\rightharpoonup^\eta v\quad\text{in}\quad L^2(I;W^{k_0,2}(\Omega;\R^n)),\label{eqh:convu1}\\
\varrho^{(h)}&\rightharpoonup^{\ast,\eta}\varrho\quad\text{in}\quad L^\infty(I;L^\beta(\Omega_{\eta})),\label{eqh:convrho1}\\
\varrho^{(h)}&\rightarrow^\eta \varrho\quad\text{in}\quad L^2(I;W^{1,2}(\Omega_\eta)).\label{eqh:van1}
\end{align}
after taking a (non-relabelled) subsequence. Taking further Lemma \ref{lem:warme2} into account we have
\begin{align}
\varrho^{(h)}&\rightharpoonup^{\ast,\eta}\varrho\quad\text{in}\quad L^\infty(I;W^{1,2}(\Omega_\eta)),\label{eqh:convrho1'}\\
\varrho^{(h)}&\rightharpoonup^\eta \varrho\quad\text{in}\quad L^2(I;W^{2,2}(\Omega_\eta)).\label{eqh:van1'}
\end{align}
Using this together with \eqref{eqh:convu1} in the equation of continuity we also have
\begin{align}
\partial_t\varrho^{(h)}&\rightharpoonup^\eta \partial_t\varrho\quad\text{in}\quad L^2(I;L^{2}(\Omega_\eta)).\label{eqh:dtrho}
\end{align}
Moreover, $\varrho^{(h)}$ stays bounded and strictly positive, that is
\begin{align}\label{eq:rhopositive}
\underline\varrho\leq\varrho^{(h)}(t,x)\leq\overline\varrho\quad\text{for a.a.}\quad (t,x)\in I\times\Omega^{(h)}
\end{align}
for some $\underline\varrho, \overline\varrho>0$ which do not depend on $h$.\\
Finally, passing to the limit in \eqref{apriori5}--\eqref{apriori6} and using \eqref{eqh:convu1} yields uniform bounds for $\Phi^{(h)}$. In particular, we have
 \begin{align}\label{apriori52} 
\underline c\leq   \det \nabla \Phi^{(h)}_s & \leq \overline c
 \end{align}
for some $\underline c,\overline c>0$ independent of $h$ and $s$ (however diverging with $T \to \infty$ or $\kappa \to 0$)
as well as
\begin{align}
 \Phi^{(h)}_s&\rightarrow \Phi\quad\text{in}\quad C^0([0,T];C^{1,\alpha}(\Omega_0))\label{conv:h:Phi}
 \end{align}
for a.a.\ $s$ as $h\rightarrow 0$.

\subsection{Strong convergence}
To pass to the limit in \eqref{eq:hlong1}--\eqref{eq:hlong3}, we need to improve some of the convergence from the last subsection to strong convergences. First we will argue similarly to \cite[Lem. 4.14, Prop. 4.15]{benesovaVariationalApproachHyperbolic2020} and use \eqref{eq:hlong1} to estimate
\begin{lemma}[$W^{-m,2}$-estimates] \label{lem:hmEst} There exists a constant $C > 0$ independent of $h$ such that for an $m\in \N$ large enough
\begin{align}
 \norm[L^2(I;W^{-m,2}(Q))]{\frac{\partial_t \eta^{(h)}- \partial_t \eta^{(h)} (\cdot-h)}{h} } &\leq C, \label{hmEst1}\\
 \int_0^T \int_{\Omega^{(h)}(t)}  \tfrac{\sqrt{\varrho^{(h)}} v^{(h)}-\sqrt{\varrho^{(h)}(t-h) \det \nabla \Phi_{-h}}v^{(h)}(t-h)}{h} \cdot \sqrt{\varrho^{(h)}}b \dxt &\leq C \norm[L^2(I;W^{m,2}(\Omega))]{b}.\label{hmEst2}
\end{align}
for all $b \in C_c(I;W_0^{m,2}(\Omega))$.
\end{lemma}

\begin{proof}
We need the following key estimate to correct the flow map. For any Lipschitz continuous function $a:I \times \Omega \to \R$ and any $s \in [0,h]$ we have
\begin{align} \label{Phi-Shift}
\int_{\Omega(t)} &\varrho^{(h)}(t) \abs{ a(t) - a(t-s) \circ \Phi_{-s}^{(h)}}^2 \dx\\ \nonumber & = \int_{\Omega(t)} \varrho^{(h)}(t) \abs{s \fint_{-s}^0 \frac{\dd}{\dd r}\left( a(t+r) \circ \Phi_{r}^{(h)}\right) \,\dd r }^2 \dx \\ \nonumber
 &\leq s^2 \int_{\Omega(t)} \varrho^{(h)}(t) \fint_{-s}^0 \abs{ \partial_t a(t+r) \circ \Phi_r^{(h)} + \nabla a(t+r) \cdot v^{(h)} \circ \Phi_r^{(h)}}^2 \,\dd r \dx \\ \nonumber
 &\leq s h \fint_{-h}^0 \int_{\Omega(t+r)} \varrho^{(h)}(t+r) c\left( (\Lip_t a)^2 + \abs{\smash{v^{(h)}}}^2 (\Lip_x a)^2 \right) \dx \,\dd r \leq sh C (\Lip_{t,x} a)^2,
\end{align}
where we used that $\int_{\Omega(t+r)} \varrho^{(h)}(t) \dx$ is constant (this follows immediately from \eqref{eq:h2}) and the fact that $\fint_{t-h}^t \int_{\Omega(t)} \varrho^{(h)}(t+r) \abs{v^{(h)}}^2 \dx \,\dd r$ is uniformly bounded by the energy estimate. By $\Lip_t$, $\Lip_x$ and $\Lip_{t,x}$ we denote the Lipschitz constants with respect to time, space and space-time respectively. This estimate varies from the corresponding estimate in \cite{benesovaVariationalApproachHyperbolic2020} by the necessary inclusion of the density, as the flow is no longer volume preserving.

Now we consider (with $\varrho^{(h)}$ per convention extended by $0$ to all of $\Omega$)
\begin{align*}
  \int_0^T &\int_{\Omega}\tfrac{\sqrt{\varrho^{(h)}}v^{(h)}-\sqrt{\varrho^{(h)}(t-h)\det \nabla \Phi^{(h)}_{-h}}v^{(h)}(t-h)}{h} \cdot \sqrt{\varrho^{(h)}}b \dxt  \\ 
  &= \int_0^T \int_{\Omega} \tfrac{\sqrt{\varrho^{(h)}}v^{(h)}-\sqrt{\varrho^{(h)}(t-h)\circ \Phi_{-h}^{(h)}\det \nabla \Phi^{(h)}_{-h}}v^{(h)}(t-h)\circ \Phi_{-h}^{(h)}}{h} \cdot \sqrt{\varrho^{(h)}}b \dxt \\
  &+ \int_0^T \int_{\Omega}{\tfrac{\sqrt{\det \nabla \Phi^{(h)}_{-h}}\big(\sqrt{\varrho^{(h)}(t-h) \circ \Phi_{-h}^{(h)}}v^{(h)}(t-h)\circ \Phi_{-h}^{(h)}-\sqrt{\varrho^{(h)}(t-h)}v^{(h)}(t-h)\big)}{h} }{\sqrt{\varrho^{(h)}}b} \dxt,
\end{align*}
where the first term can be estimated using equation \eqref{eq:hlong1}. For the second one we perform a change of variables with $\Phi_h(t-h)$ in its first term to obtain
\begin{align*}
 \int_0^T& \int_{\Omega}{\tfrac{\sqrt{\det \nabla \Phi^{(h)}_{-h}}\big(\sqrt{\varrho^{(h)}(t-h) \circ \Phi_{-h}^{(h)}}v^{(h)}(t-h)\circ \Phi_{-h}^{(h)}-\sqrt{\varrho^{(h)}(t-h)}v^{(h)}(t-h)\big)}{h} }{\sqrt{\varrho^{(h)}}b} \dxt \\
 &= \int_0^T  \int_\Omega \tfrac{\det \nabla \Phi^{(h)}_h(t-h) \sqrt{\det \nabla \Phi^{(h)}_{-h} \circ\Phi^{(h)}_h(t-h)  \varrho^{(h)}(t-h)}v^{(h)}(t-h) \sqrt{\varrho^{(h)}\circ\Phi^{(h)}_h(t-h)}b \circ\Phi^{(h)}_h(t-h)}{h} \dx \\
 &-  \int_{\Omega}{\tfrac{\sqrt{\det \nabla \Phi^{(h)}_{-h}}\sqrt{\varrho^{(h)}(t-h)}v^{(h)}(t-h)}{h} }{\sqrt{\varrho^{(h)}}b} \dxt \\
 &= \int_0^T  \int_\Omega \tfrac{ \sqrt{\det \nabla \Phi^{(h)}_h(t-h)} \sqrt{\varrho^{(h)}\circ\Phi^{(h)}_h(t-h)}b \circ\Phi_h(t-h) -\sqrt{\det \nabla \Phi^{(h)}_{-h}}\sqrt{\varrho^{(h)}}b}{h} \sqrt{\varrho^{(h)}(t-h)} v^{(h)}(t-h)  \dxt \\
 &= \int_0^T  \int_\Omega \sqrt{\det \nabla \Phi^{(h)}_{-h}}\sqrt{\varrho^{(h)}} \tfrac{ b \circ\Phi^{(h)}_h(t-h) -b}{h} \sqrt{\varrho^{(h)}(t-h)} v^{(h)}(t-h)  \dxt \\
 &+ \int_0^T \int_\Omega \tfrac{ \sqrt{\det \nabla \Phi^{(h)}_h(t-h)} \sqrt{\varrho^{(h)}\circ\Phi^{(h)}_h(t-h)} -\sqrt{\det \nabla \Phi^{(h)}_{-h}}\sqrt{\varrho^{(h)}}}{h} {\scriptstyle b \circ\Phi^{(h)}_h(t-h)\sqrt{\varrho^{(h)}(t-h)} v^{(h)}(t-h) } \dxt \\
 &=: \mathrm{I}+ \mathrm{II}
 \end{align*}
 Now the absolute value of the first integral can be estimated by
 \begin{align*}
 &\abs{\mathrm{I}} \leq  \int_0^T \sqrt{ \int_{\Omega} \varrho^{(h)} \abs{v^{(h)}}^2 \dx  } \sqrt{ \int_{\Omega} \varrho^{(h)} \abs{ \tfrac{b(t+h) \circ \Phi_{h}^{(h)} - b}{h}}^2 \dx  } \dt 
 \leq C \Lip_{t,x} b 
\end{align*}
by using that 
\begin{align*} \int_0^T \int_{\Omega} \varrho^{(h)} |v^{(h)}|^2 \dxt = \sum_{l=0}^{T/h-1} h \fint_{lh}^{(l+1)h} \int_{\Omega} \varrho^{(h)} |v^{(h)}|^2 \dxt
 \end{align*}
 consists of $T/h$ terms uniformly bounded by a multiple of $h$, due to the energy-inequality and by applying \eqref{Phi-Shift} with $s = h$ and $a=b$

 For the second integral we have by Corollary \ref{cor:Phi_hProperties} and the various estimates
\begin{align*}
 \abs{\mathrm{II}} &= \bigg| \int_0^T \int_{\Omega} \fint_{t-h} \partial_s\left( \sqrt{\det\nabla \Phi_s^{(h)} \varrho^{(h)} \circ \Phi_s}_{}  \right) \dd s\, b \circ\Phi_h(t-h)\sqrt{\varrho^{(h)}(t-h)} v^{(h)}(t-h)\dxt \bigg| \\
 &\leq \norm[L^2(I;W^{2,2}(\Omega))]{\varrho^{(h)}} \norm[\infty]{\varrho^{(h)}}\norm[\infty]{b} \norm[L^2(I\times \Omega))]{v^{(h)}} \leq C \norm[\infty]{b}. \qedhere
\end{align*}
\end{proof}

Finally we need to prove strong convergence of the density. For this purpose
we choose a parabolic cylinder $J\times B$ such that $2B\Subset \Omega^{(h)}$ for all $t\in J$ and all $h$ small enough. This is possible due to \eqref{eqh:conetatal}. From the continuity equation \eqref{eq:hlong2} we obtain
\begin{align*}
\partial_t \varrho^{(h)}\in L^2(J;W^{-1,2}(B))
\end{align*}
uniformly in $h$. This, in combination with the gradient estimate from \eqref{est:h2}, yields
\begin{align*}
 \varrho^{(h)}&\rightarrow \varrho\quad\text{in}\quad L^2(J\times B).
 \end{align*}
 Using \eqref{eq:rhopositive} and the arbitrariness of $J\times B$ we obtain
 \begin{align}\label{eq:0106rhostrong}
 \varrho^{(h)}&\rightarrow^\eta \varrho\quad\text{in}\quad L^p(I\times\Omega_\eta)
 \end{align}
  for some $p>\beta$.  Using
Theorem \ref{lem:warme} (b) with $\theta(s)=s^2$ (which is admissible by approximation)
we obtain
\begin{align}
\label{eq:renormzh}
\begin{aligned}
\int_{\Omega^{(h)}} |\varrho^{(h)}(t_1)|^2\dx+\int_0^{t_1}&\int_{\Omega^{(h)}}2\varepsilon |\nabla\varrho^{(h)}|^2\dx\dt\\
&=\int_{\Omega_0}\varrho_0^2\dx-\int_0^{t_1}\int_{\Omega^{(h)}}2\varrho^{(h)}\Div v^{(h)}\dx\dt.
\end{aligned}
\end{align}
Now We apply Theorem \ref{lem:warme} (b) to the limit version of the continuity equation (that is, equation \eqref{eq:con:kappa}) which results in a counterpart of \eqref{eq:renormzh}. On account of \eqref{eq:0106rhostrong} all terms converge except for $\int_0^t\int_{\Omega_{\eta^{(h)}}}2\varepsilon |\nabla\varrho^{(h)}|^2\dx\ds$, which yields convergence of the latter. Hence we obtain
 \begin{align}\label{eq:0106rhostrongh}
 \varrho^{(h)}&\rightarrow^\eta \varrho\quad\text{in}\quad L^2(I;W^{1,2}(\Omega_\eta)).
 \end{align}
  
   \subsection{Derivation of the material derivative}
  We now take the inertial term from~\eqref{eq:hlong1} and shift its second half in time by a change of variables. This gives us (note the symmetry of the $\sqrt{\varrho}$-terms)
  \begin{align*}
    \int_0^T &\frac{1}{h}\int_{\Omega^{(h)}} \left(\varrho^{(h)} v^{(h)}-\sqrt{\varrho^{(h)}} \sqrt{\varrho^{(h)}(t-h) \circ \Phi^{(h)}_{-h} \det \nabla \Phi^{(h)}_{-h}} v^{(h)}(t-h) \circ \Phi^{(h)}_{-h} \right) \cdot b\dx \dt \\ 
   &=  \int_0^T \frac{1}{h}\int_{\Omega^{(h)}} \left(\varrho^{(h)} b-\sqrt{\varrho^{(h)}} \sqrt{\varrho^{(h)}(t+h) \circ \Phi^{(h)}_{h} \det \nabla \Phi^{(h)}_{h}} b(t+h) \circ \Phi^{(h)}_h\right) \cdot  v^{(h)}\dx \dt
   \end{align*}
   recalling that $(\Phi^{(h)}_{h})^{-1}(t,x)=\Phi^{(h)}_{-h}(t+h))$.
Using the fundamental theorem of calculus and \autoref{cor:Phi_hProperties} the integrand can be rewritten as\footnote{Note that $\det \nabla \Phi^{(h)}_{s}(t,x)$ signifies that this term is always evaluated at $x$, while most other terms (those at time $t+s$ are evaluated at $\Phi^{(h)}_{s}(t,x)$. In general, all terms are evaluated at their ``natural'' point.}
\begin{align*}
  -\int_{0}^h& \sqrt{\varrho^{(h)}(t)} \partial_s \left(\sqrt{\varrho^{(h)}(t+s) \circ \Phi^{(h)}_s(t) \det \nabla \Phi^{(h)}_{s}(t)} b(t+s) \circ \Phi^{(h)}_s(t) \right) \cdot v^{(h)}(t) \, \dd s  \\
   &= -\int_0^h \sqrt{\varrho^{(h)}(t)} \tfrac{\partial_s \left( [\varrho^{(h)}(t+s)\det \nabla \Phi^{(h)}_{s}(t,x)] \circ \Phi^{(h)}_s(t)\right) }{2\sqrt{\varrho^{(h)}(t+s) \circ \Phi^{(h)}_s(t) \det \nabla \Phi^{(h)}_{s}(t,x)}}b(t+s)  \circ \Phi^{(h)}_s(t) \cdot v^{(h)}(t) \,\dd s  \\
   &  - \int_0^h \sqrt{\varrho^{(h)}(t)}\left[\sqrt{\varrho^{(h)}(t+s) \det \nabla \Phi^{(h)}_{s}(t,x)} \partial_t b(t+s)\right] \circ \Phi^{(h)}_s(t) \cdot v^{(h)}(t)\,\dd s \\
      & - \int_0^h \sqrt{\varrho^{(h)}(t)}\left[\sqrt{\varrho^{(h)}(t+s) \det \nabla \Phi^{(h)}_{s}(t,x)}  v^{(h)}(t+s) \cdot \nabla b(t+s)  \right]\circ \Phi^{(h)}_s(t) \cdot v^{(h)}(t) \,\dd s \\
   &=-\int_0^h \sqrt{\varrho^{(h)}(t)}\left[ \sqrt{\det \nabla \Phi^{(h)}_{s}(t,x)} \tfrac{\varepsilon \Delta \varrho^{(h)}(t+s)   }{2\sqrt{\varrho^{(h)}(t+s)} }  b(t+s)\right]  \circ \Phi^{(h)}_s(t) \cdot v^{(h)}(t)\, \dd s \\
   &  - \int_0^h \sqrt{\varrho^{(h)}(t)} \left[\sqrt{\det \nabla \Phi^{(h)}_{s}(t,x)} \sqrt{\varrho^{(h)}(t+s)} \partial_t b(t+s)\right]\circ\Phi^{(h)}_s(t) \cdot v^{(h)}(t)\,\dd s \\
   &- \int_0^h \sqrt{\varrho^{(h)}(t)}\left[ \sqrt{\det \nabla \Phi^{(h)}_{s}(t,x)}\sqrt{\varrho^{(h)}(t+s)} v^{(h)}(t+s) \cdot \nabla b(t+s)\right]  \circ \Phi^{(h)}_s(t) \cdot v^{(h)}(t)\, \dd s\\
   &=:-(\mathrm{I})_h-(\mathrm{II})_h-(\mathrm{III})_h
  \end{align*} 
  using in particular Corollary \ref{cor:Phi_hProperties}.

  We can now deal with each of these terms one after the other. Using $(\Phi^{(h)}_{s})^{-1}(t,x)=\Phi^{(h)}_{-s}(t+s)$ and $\partial_{\nu^{(h)}(t)}\varrho^{(h)}(t)=0$ on $\partial\Omega^{(h)}(t)$ we find
  \begin{align*}
    \int_0^T \frac{1}{h}\int_{\Omega^{(h)}}(\mathrm{I})_h\dx\dt&= 
   \int_0^T \fint_0^h \int_{\Omega^{(h)}(t+s)} \sqrt{\varrho^{(h)}(t)\circ \Phi_{-s}^{(h)}(t+s,x)} \sqrt{\det \nabla \Phi^{(h)}_{-s}(t,x)}  \tfrac{\varepsilon \Delta \varrho^{(h)} (t+s) }{2\sqrt{\varrho^{(h)}(t+s) } }  b(t+s)\\\nonumber&\qquad\qquad\qquad\qquad\qquad\qquad \qquad\qquad\qquad\cdot v^{(h)}(t) \circ \Phi_{-s}^{(h)}(t+s,x) \,\dd s \dx \dt \\
      &= 
   \int_0^T \fint_0^h \int_{\Omega^{(h)}(t+s)} \Big(\tfrac{\sqrt{\varrho^{(h)}(t)\circ \Phi_{-s}^{(h)}(t+s,x)} \sqrt{\det \nabla \Phi^{(h)}_{-s}(t,x)}   }{2\sqrt{\varrho^{(h)}(t+s) } }-1\Big) \\&\qquad\qquad\qquad \qquad\quad\quad\cdot\varepsilon \Delta \varrho^{(h)} (t+s) b(t+s)\cdot v^{(h)}(t) \circ \Phi_{-s}^{(h)}(t+s,x) \,\dd s \dx \dt\\
      &-
   \int_0^T \fint_0^h \int_{\Omega^{(h)}(t+s)}  \frac{\varepsilon}{2} \nabla \varrho^{(h)} (t+s)  \nabla \big(v^{(h)}(t) \circ \Phi_{-s}^{(h)}(t+s,x)\big)b(t+s) \,\dd s \dx \dt\\
      &-
   \int_0^T \fint_0^h \int_{\Omega^{(h)}(t+s)}  \frac{\varepsilon}{2} \nabla \varrho^{(h)} (t+s)\cdot  \nabla b(t+s) v^{(h)}(t) \circ \Phi_{-s}^{(h)}(t+s,x) \,\dd s \dx \dt\\
   & =:(\mathrm{IV})_h+(\mathrm{V})_h+(\mathrm{VI})_h.
  \end{align*}
  As $h\rightarrow0$ we have
       \begin{align*}
    (\mathrm{IV})_h \to0
  \end{align*}
    on account of \eqref{eqh:van1'}, \eqref{eq:rhopositive}
and
 \begin{align*}
&\int_{\Omega^{(h)}(t+s)}|v^{(h)}(t)\circ\Phi_{-s}^{(h)}(t+s)|^2\dx\leq\,c\int_{\Omega^{(h)}(t)}|v^{(h)}(t)|^2\dx,\\  &\sup_{l\in\{0,\dots, T/h-1\}}\dashint_{lh}^{(l+1)h} \int_{\Omega^{(h)}}|v^{(h)}|^2\dx\leq\,c,
\end{align*}
  the latter one being a consequence of \eqref{eq:hlong3} and \eqref{eq:rhopositive}.
  Moreover, it holds
      \begin{align*}
(\mathrm{V})_h\to- \int_0^T \int_{\Omega(t)} \frac{ \varepsilon}{2} \nabla \varrho(t)\cdot\nabla v(t) b(t) \dx \dt
  \end{align*}
  as well as
       \begin{align*}
(\mathrm{VI})_h \to- \int_0^T \int_{\Omega(t)} \frac{ \varepsilon}{2} \nabla \varrho(t)\cdot \nabla b(t) \cdot v(t)\dx \dt.
  \end{align*}
  by \eqref{conv:h:Phi}, \eqref{eq:0106rhostrong} and \eqref{eq:0106rhostrongh}.
  We conlcude that
    \begin{align}
    \int_0^T \frac{1}{h}\int_{\Omega^{(h)}}(\mathrm{I})_h\dx\dt \to- \int_0^T \int_{\Omega(t)} \frac{ \varepsilon}{2} \nabla \varrho(t)\cdot\Big(\nabla b(t) \cdot v(t)+\nabla v(t) b(t)\Big) \dx \dt.\label{limregh}
  \end{align}
  Similarly, we have
  \begin{align*}
   \int_0^T \frac{1}{h}\int_{\Omega^{(h)}}(\mathrm{II})_h\dx\dt &\to \int_0^T \int_{\Omega(t)} \varrho(t) v(t) \partial_t b(t) \dx \dt
  \end{align*}
    as $h\rightarrow0$. The remaining term involving $(\mathrm{III})_h$ is more critical since $v^{(h)}$, which is based on our a-priori estimates only weakly converging, appears in a product with itself. However, we may use the discrete time derivative for the momentum. 
    It holds
    \begin{align*}
    \partial_t \fint^{t+h}_t \sqrt{\varrho^{(h)}(s)}&\sqrt{\det \nabla \Phi^{(h)}_{s}(t,x)}   v^{(h)}(s) \,\dd s \\&= \frac{\sqrt{\varrho^{(h)}(t+h)}\sqrt{\det \nabla \Phi^{(h)}_{s}(t,x)}   v^{(h)}(t+h) - \sqrt{\varrho^{(h)}(t)} v^{(h)}(t)}{h} 
    \end{align*}
    such that
    \begin{align*}
   \partial_t\bigg[\sqrt{\varrho^{(h)}(t)}& \fint^{t+h}_t \sqrt{\varrho^{(h)}(s)}\sqrt{\det \nabla \Phi^{(h)}_{s}(t,x)}  v^{(h)}(s)\,\dd s\bigg]\\ &=\sqrt{\varrho^{(h)}(t)} \frac{\sqrt{\varrho^{(h)}(t+h)}\sqrt{\det \nabla \Phi^{(h)}_{s}(t,x)}  v^{(h)}(t+h) - \sqrt{\varrho^{(h)}(t)} v^{(h)}(t)}{h}\\
    &+\frac{\partial_t\varrho^{(h)}(t)}{2\sqrt{\varrho^{(h)}(t)}} \fint^{t+h}_t \sqrt{\varrho^{(h)}(s)}\sqrt{\det \nabla \Phi^{(h)}_{s}(t,x)}  v^{(h)}(s)\,\dd s.
   \end{align*}
For the first term we have 
a $W^{-m,2}$-estimate given in \eqref{hmEst2}. The second term can be estimated in $L^1(I;L^2(\Omega^{(h)}))$ by \eqref{eqh:van1'}, \eqref{eqh:dtrho}, \eqref{eq:rhopositive} and \eqref{apriori52}. We conclude that
\begin{align*}
 \int_0^T \int_{\Omega(t)}    \partial_t\bigg[\sqrt{\varrho^{(h)}(t)}  \fint^{t+h}_t \sqrt{\varrho^{(h)}(s)} \sqrt{\det \nabla \Phi^{(h)}_{s}(t,x)} v^{(h)}(s)\,\dd s\bigg] \cdot b  \dxt \leq C \norm[L^1(I;W^{m,2}(\Omega))]{b}
\end{align*}
uniformly in $h$.
Now, Lemma \ref{thm:weakstrong} yields
   \begin{align}
\nonumber
\int_0^T\int_{\Omega^{(h)}(t)} \fint_0^h &\sqrt{\varrho^{(h)}(t)}\sqrt{\det \nabla \Phi^{(h)}_{s}(t,x)}  \left[  \sqrt{\varrho^{(h)}(t+s)} v^{(h)}(t+s) \cdot \nabla b(t+s) \right]  \cdot v^{(h)}(t) \,\dd s \dx  \dt\\
   &\to \int_0^T \int_{\Omega(t)} \varrho(t) v(t) \cdot \nabla b(t) v(t) \dx \dt\label{eq:0106b}
  \end{align}    
 as $h\rightarrow0$ taking also \eqref{eqh:van1'}, \eqref{conv:h:Phi} and \eqref{eq:0106rhostrong} into account. It remains to ``add'' the shift
 ${}\circ \Phi^{(h)}_s(t)$. For this purpose we decompose
  \begin{align*}
  \int_0^T\int_{\Omega^{(h)}(t)}& \fint_0^h \sqrt{\varrho^{(h)}(t)}\sqrt{\det \nabla \Phi^{(h)}_{s}(t,x)}  \Big[  \dots \Big] \circ \Phi^{(h)}_s(t) \cdot v^{(h)}(t) \,\dd s \dx  \dt\\
    &=\int_0^T\int_{\Omega^{(h)}(t)} \fint_0^h \sqrt{\varrho^{(h)}(t)} \sqrt{\det \nabla \Phi^{(h)}_{s}(t,x)} \bigg(\Big[  \dots \Big]_\xi \circ \Phi^{(h)}_s(t)-\Big[  \dots \Big]_{\xi}\bigg) \cdot v^{(h)}(t) \,\dd s \dx  \dt\\
       &+\int_0^T\int_{\Omega^{(h)}(t)} \fint_0^h \sqrt{\varrho^{(h)}(t)}\sqrt{\det \nabla \Phi^{(h)}_{s}(t,x)}  \bigg(\Big[  \dots \Big]_{\xi}-\Big[  \dots \Big]\bigg) \cdot v^{(h)}(t) \,\dd s \dx  \dt\\
         &+\int_0^T\int_{\Omega^{(h)}(t)} \fint_0^h \sqrt{\varrho^{(h)}(t)}\sqrt{\det \nabla \Phi^{(h)}_{s}(t,x)}  \Big[  \dots \Big] \cdot v^{(h)}(t) \,\dd s \dx  \dt
  \end{align*}
  Here $[\cdot]_\xi$ denotes a regularisation in space defined by an extension to the whole space\footnote{Since $\partial\Omega^{(h)}$ is a Lipschitz boundary uniformly in time by \eqref{est:h1} standard results on the extension of Sobolev functions apply.} and a mollification with a smooth kernel.
 In the above we use $[  \dots]$ as a shorthand for
 \begin{align*}
 \Big[  \dots \Big]=\left[  \sqrt{\varrho^{(h)}(t+s)} v^{(h)}(t+s) \cdot \nabla b(t+s) \right].
 \end{align*}  For fixed $\xi$ we can use smoothness of $[\cdot]_\xi$ to conclude that the first term vanishes as $h\rightarrow0$. This is a consequence of \eqref{Phi-Shift} and the a priori estimates. Also, the second term converges to zero as $\xi\rightarrow 0$ (uniformly with respect to $h$, recall \eqref{eqh:convrho1'},  \eqref{eqh:van1'} and \eqref{eq:rhopositive}) by standard properties of the mollification. The last term converges to the expected limit as we have seen in \eqref{eq:0106b}.
  In conclusion, we have
    \begin{align*}
 \int_0^T\int_{\Omega^{(h)}} \fint_0^h \sqrt{\varrho^{(h)}(t)}&\sqrt{\det \nabla \Phi^{(h)}_{s}(t,x)} \Big[  \sqrt{\varrho^{(h)}(t+s)} v^{(h)}(t+s) \cdot \nabla \phi(t+s) \Big] \circ \Phi^{(h)}_s(t) \cdot v^{(h)}(t) \,\dd s \! \dx \!  \dt\\
   &\to \int_0^T \int_{\Omega(t)} \varrho(t) v(t) \cdot \nabla \phi(t) v(t) \dx \dt
  \end{align*}    
 as $h\rightarrow0$, which finishes the proof of \eqref{eq:mom:kappa}.

 \section{Removal of the remaining approximation parameters}
 
\label{sec:5}
In this section we pass to the limit in the approximate equations. For technical reasons
the limits $\kappa\rightarrow0$, $\varepsilon\rightarrow0$ and $\delta\rightarrow0$ have to be performed independently from each other.  The limit $\kappa\rightarrow0$ is rather straightforward as the density remains compact for $\varepsilon>0$.
For the greater part of this section we study the limit $\varepsilon\rightarrow0$ and only highlight the difference in the $\delta$-limit.

\subsection{The limit system for \texorpdfstring{$\kappa\to0$}{k to 0}}
\label{subsec:eps}

Recalling the definition of the function spaces from Section \ref{sec:weak} we seek a triple $(\eta,v,\varrho)\in Y^I\times X_{\eta}^I\times \hat{Z}_\eta^I$ that satisfies the following.
\begin{itemize}
\item The momentum equation holds in the sense that\begin{align}\label{eq:mom:eps}
\begin{aligned}
&\int_I\frac{\dd}{\dt}\int_{\Omega(t)}\varrho v \cdot b\dx-\int_{\Omega(t)} \Big(\varrho v\cdot \partial_t b +\varrho v\otimes v:\nabla b\Big)\dxt
\\
&+\int_I\int_{\Omega(t)}\mathbb S(\nabla v):\nabla b \dxt-\int_I\int_{\Omega(t)}
p_\delta(\varrho)\,\Div b\dxt+\frac{\varepsilon}{2}\int_I\int_{\Omega(t)}\nabla\varrho\cdot(\nabla v b+\nabla b v)\dxt\\
&+\int_I\bigg(-\int_Q \varrho_s\partial_t\eta\,\partial_t \phi \dy + \langle DE(\eta), \phi\rangle+\langle D_2R(\eta,\partial_t\eta),\phi\rangle\bigg)\dt
\\&=\int_I\int_{\Omega(t)}\varrho f_f\cdot b\dxt+\int_I\int_Q f_s \cdot \phi\,\dd x\dt
\end{aligned}
\end{align} 
for all $(\phi,b)\in L^2(I;W^{2,q}(Q;\Omega)) \cap W^{1,2}(I;L^2(Q;\R^n)) \times C^\infty_c(\overline{I}\times \Omega; \R^n)$ with $\phi(t) = b(t) \circ \eta(t)$ in $Q$ and $\phi(t)=0$ on $P$, where $\Omega(t) := \Omega \setminus \eta(t,Q)$. Moreover, we have $(\varrho v)(0)=q_0$, $\eta(0)=\eta_0$ and $\partial_t\eta(0)=\eta_1$ we well as $\partial_t\eta(t)=v(t)\circ\eta(t)$ in $Q$, $\eta(t)\in\mathcal E$ and $v(t)=0$ on $\partial\Omega$ for a.a. $t\in I$.
\item The continuity equation holds in the sense that
\begin{align}\label{eq:con:eps}
\begin{aligned}
&\int_I\frac{\dd}{\dt}\int_{\Omega(t)}\varrho \psi\dxt-\int_I\int_{\Omega(t)}\Big(\varrho\partial_t\psi
+\varrho v\cdot\nabla\psi\Big)\dxt=\varepsilon\int_I\int_{\Omega(t)}\nabla\varrho\cdot\nabla\psi\dxt
\end{aligned}
\end{align}
for all $\psi\in C^\infty(\overline{I}\times\R^3)$, $\partial_{\nu(t)}\varrho(t)=0$ in $\partial \Omega(t)$ for a.a. $t \in I$ and we have $\varrho(0)=\varrho_0$. 
\item  The energy inequality is satisfied in the sense that
\begin{align} \label{eq:ene:eps}
\begin{aligned}
- \int_I \partial_t &\psi \,
\mathscr E_\delta \dt+\int_I\psi\int_{\Omega(t)}\mathbb S(\nabla v):\nabla v\dxs\\&+2\int_I\psi R(\eta,\partial_t\eta)\ds+\varepsilon\int_I\psi\int_{\Omega(t)}H_\delta''(\varrho)|\nabla\varrho|^2\dxs \\&\leq
\psi(0) \mathscr E_\delta(0)+\int_I\psi\int_{\Omega(t)}\varrho f_f\cdot v\dxt+\int_I\psi\int_Q g_s\,\partial_t\eta\,\dd y\dt
\end{aligned}
\end{align}
holds for any $\psi \in C^\infty_c([0, T))$.
Here, we abbreviated
$$\mathscr E_\delta(t)= \int_{\Omega(t)}\Big(\frac{1}{2} \varrho(t) | {v}(t) |^2 + H_\delta(\varrho(t))\Big)\dx+\int_Q\varrho_s\frac{|\partial_t\eta|^2}{2}\dy+ E(\eta(t)).$$
\end{itemize}
\begin{theorem}\label{thm:eps}
Assume that we have for some $\alpha\in(0,1)$
\begin{align*}
\frac{|q_0|^2}{\varrho_0}&\in L^1(\Omega_{\eta_0}),\ \varrho_0\in C^{2,\alpha}(\overline\Omega_{\eta_0}), \ \eta_0\in \mathcal E,\ \eta_1\in L^2(Q;\R^n),\\
f_f&\in L^2(I;L^\infty(\Omega;\R^n)),\ f_s\in L^2(I\times Q;\R^n).
\end{align*}
Furthermore suppose that $\varrho_0$ is strictly positive. There is a solution $(\eta,v,\varrho)\in Y^I\times X_\eta^I\times \hat{Z}_\eta^I$ to \eqref{eq:mom:eps}--\eqref{eq:ene:eps}. 
Here, we have $I=(0,T_*)$, where $T_*<T$ only if the time $T_*$ is the time of the first contact
of the free boundary of the solid body either with itself or $\partial\Omega$ (i.e. $\eta(T_*)\in\partial\mathscr E$). 
\end{theorem}

\begin{proof}
For $T>0$ to be fixed later and any given $\kappa$ we obtain a solution $(\eta^{(\kappa)},v^{(\kappa)},\varrho^{(\kappa)})$ to \eqref{eq:mom:kappa}--\eqref{eq:ene:kappa} by Theorem \ref{thm:kappa}.
In particular, we have
\begin{align} \label{NTDBkappa}\begin{aligned}
&
 \int_{\Omega^{(\kappa)}} \Big[ \frac{1}{2} \varrho^{(\kappa)} | {v^{(\kappa)}}(t_1) |^2 + H_{\delta}(\varrho^{(\kappa)}(t_1))
  \Big] \dx+\int_Q\varrho_s\tfrac{|\partial_t\eta^{(\kappa)}(t_1)|^2}{2}\dy+ E(\eta^{(\kappa)}(t_1))\\&+\int_0^{t_1} \left[ \int_{\Omega^{(\kappa)}}\mathbb S(\nabla v^{(\kappa)}):\nabla v^{(\kappa)}\dx
+2 \int_{Q}R(\eta^{(\kappa)},\partial_t\eta^{(\kappa)})\dy+\varepsilon \int_{\Omega^{(\kappa)}}H''_\delta(\varrho^{(\kappa)})|\nabla\varrho^{(\kappa)}|^2\dx \right.\\
&\left.+\kappa\int_{\Omega^{(\kappa)}}|\nabla^{k_0} v^{(\kappa)}|^2\dx
+2\kappa \int_{Q}|\nabla^{k_0}\eta^{(\kappa)}|^2\dy\right] \dt +\int_Q \kappa|\nabla^{k_0}\eta^{(\kappa)}(t_1)|^2\dy\\
 &\leq \,\int_{\Omega_0}\Big[ \frac{1}{2} \varrho_0 |v_0|^2  + H_{\delta}(\varrho_0) \Big] \dx+\int_Q\varrho_s\frac{|\eta_1|^2}{2}\dy+E(\eta_0) \\  &+ \int_0^{t_1}\left[ \int_{\Omega^{(\kappa)}} \varrho^{(\kappa)} f_f \cdot v^{(\kappa)} \dx + \int_Q f_s \cdot \partial_t \eta^{(\kappa)} \dy  \right] \dt
 \end{aligned}
\end{align}
for almost all $0 \leq t_1 \leq T$.
We deduce the bounds
\begin{align} 
\label{est:k1}
 \| \partial_t\nabla\eta^{(\kappa)} \|_{L^2(I\times Q) }^2+\sup_{t\in I}\|\partial_t\eta^{(\kappa)}\|_{L^{2}(Q)}^2+\sup_{t\in I}\|\eta^{(\kappa)}\|_{W^{2,q}(Q)}^q\leq c,\\
\label{est:k2}
\sup_{t \in I} \| \varrho^{(\kappa)} \|^\beta_{L^\beta(\Omega^{(\kappa)})}  +
 \sup_{t \in I} \| \varrho^{(\kappa)} v^{(\kappa)} \|_{L^{\frac{2 \beta}{\beta + 1}}(\Omega^{(\kappa)})}^{\frac{2 \beta}{\beta + 1}}  \leq c,\\
 \label{est:k3}
 \| \nabla v^{(\kappa)} \|^2_{L^2(I\times \Omega^{(\kappa)}) }+ \| \nabla \varrho_\varepsilon \|^2_{L^2(I\times\Omega^{(\kappa)}) } + \| \nabla (\varrho^{(\kappa)})^{\beta/2} \|^2_{L^2(I\times\Omega^{(\kappa)}) } \leq c,\\
   \label{est:k4}
\sqrt{\kappa}\Big(\| v^{(\kappa)} \|^2_{L^2(I;W^{k_0,2}(\Omega^{(\kappa)})) }+\sup_{t\in I}\| \eta^{(\kappa)} \|^2_{W^{k_0,2}(Q) }+\| \partial_t\eta^{(\kappa)} \|^2_{L^2(I;W^{k_0,2}(Q)) }\Big) \leq c,
\end{align}
where in particular the bound on $\| \partial_t\nabla\eta^{(\kappa)} \|_{L^2(I\times Q) }$ converges to $0$ if we send $T\to 0$. Using \cite[Prop. 2.7]{benesovaVariationalApproachHyperbolic2020}, we can then choose $T$ small enough, so that there will be no collision, even in the limit $\kappa \to 0$.

Passing to a subsequence we obtain
 for some $\alpha\in (0,1)$
\begin{align}
\eta^{(\kappa)}&\rightharpoonup^\ast\eta\quad\text{in}\quad L^\infty(I;W^{2,q}(Q;\Omega))\label{kappaeq:conveta1},\\
\eta^{(\kappa)}&\rightharpoonup^\ast\eta\quad\text{in}\quad W^{1,\infty}(I;L^2(Q;\R^n)),
\label{kappaeq:conetat1}\\
\eta^{(\kappa)}&\rightharpoonup\eta\quad\text{in}\quad W^{1,2}(I;W^{1,2}(Q;\R^n)),
\label{kappaeq:conetat2}\\
\eta^{(\kappa)}&\to\eta\quad\text{in}\quad C^\alpha(\overline I\times Q;\Omega)),
\label{kappaeq:conetatal}\\
\kappa\eta^{(\kappa)}&\rightharpoonup^* 0\quad\text{in}\quad L^\infty(I;W^{k_0,2}(Q;\Omega)),\label{kappaeq:conetat3}\\
\kappa\partial_t\eta^{(\kappa)}&\rightharpoonup 0\quad\text{in}\quad L^2(I;W^{k_0,2}(Q;\R^n)),\label{kappaeq:conetat4}\\
v^{(\kappa)}&\rightharpoonup^\eta v\quad\text{in}\quad L^2(I;W^{1,2}(\Omega;\R^n)),\label{kappaeq:convu1}\\
\kappa v^{(\kappa)}&\rightharpoonup^\eta 0\quad\text{in}\quad L^2(I;W^{k_0,2}(\Omega;\R^n)),\label{kappaeq:convu2}\\
\varrho^{(\kappa)}&\rightharpoonup^{\ast,\eta}\varrho\quad\text{in}\quad L^\infty(I;L^\beta(\Omega_{\eta^{(\kappa)}})),\label{kappaeq:convrho1}\\
\varrho^{(\kappa)}&\rightharpoonup^{\eta}\varrho\quad\text{in}\quad L^2(I;W^{1,2}(\Omega_{\eta^{(\kappa)}}))\label{kappaeq:convrho2}.
\end{align}
Clearly, the $\kappa$-terms in \eqref{eq:mom:kappa} vanishes as $\kappa\rightarrow0$ as a consequence of \eqref{kappaeq:conetat3}, \eqref{kappaeq:conetat4} and \eqref{kappaeq:convu2}.
Arguing as in \cite[Prop. 2.23]{benesovaVariationalApproachHyperbolic2020} one can use assumption S6 to deduce
\begin{align}\label{15:02kappa}
\eta^{(\kappa)}&\rightarrow\eta\quad\text{in}\quad L^q(I;W^{2,q}(Q;\Omega))
\end{align}
such that for a.e. $t\in I$
\begin{align}\label{15:02'kappa}
DE(\eta^{(\kappa)}(t))&\rightarrow DE(\eta(t))\quad\text{in}\quad W^{-2,q}(Q;\Omega).
\end{align}
In order to pass to the limit in various terms
in the equations
we are concerned with the compactness of $\varrho^{(\kappa)}$. 
Due to \eqref{kappaeq:convrho2} and \eqref{eq:con:kappa} we can apply
Corollary \ref{rem:strong} to conclude
\begin{align}\label{0103}
\varrho^{(\kappa)}\rightarrow^{\eta}\varrho\quad\text{in}\quad L^2(I;L^2(\Omega)).
\end{align}
 In combination with \eqref{est:k3} this can be improved to
\begin{align}\label{0103a}
\varrho^{(\kappa)}\rightarrow^{\eta}\varrho\quad\text{in}\quad L^p(I;L^p(\Omega_\eta)),
\end{align}
for some $p>\beta$. It is easy to see that \eqref{0103a} allows to pass to the limit in all nonlinear terms of \eqref{eq:mom:kappa} and \eqref{eq:con:kappa} except of
\begin{align*}
\frac{\varepsilon}{2}\int_I\int_{\Omega^{(\kappa)}}\nabla\varrho^{(\kappa)}\cdot(\nabla v^{(\kappa)} \phi+\nabla\phi v^{(\kappa)})\dxt.
\end{align*}
Due to \eqref{kappaeq:convu1} 
we obtain the expected limit as $\kappa\rightarrow0$ provided we have
\begin{align}\label{0104}
\nabla\varrho^{(\kappa)}\rightarrow^{\eta}\nabla\varrho\quad\text{in}\quad L^2(I;L^2(\Omega;\R^n)).
\end{align}
This can be proved exactly as in \eqref{eq:0106rhostrongh}
 using Theorem \ref{lem:warme} (b).
Finally we complete proof of Theorem \eqref{thm:eps}, by extending the solution to long times: Assume that $I = [0,T)$ is a maximal interval of existence with $T< \infty$. Then using the energy-inequality, we conclude existence of limits $\eta(T),\partial_t \eta(T), \varrho(T)$ and $v(T)$. Now $\eta(T)$ has to have a collision, which proves the theorem. otherwise we could construct an extended solution by applying the theorem with these as initial data, which would be a contradiction.
 \end{proof}

\subsection{The limit system for \texorpdfstring{$\varepsilon\to0$}{ε to 0}}
\label{subsec:del}
We wish to establish the existence of a weak solution $(\eta,v,\varrho)$ to the system with artificial pressure in the following sense: We define
$$
\widetilde Z_\eta^I= C_w(\overline{I};L^\beta(\Omega_\eta))
$$
as the function space for the density, whereas the other function spaces are defined in Section \ref{sec:weak}.
A weak solution is a triple $(\eta,v,\varrho)\in Y^I\times X_{\eta}^I\times\widetilde Z_\eta^I$ that satisfies the following.
\begin{itemize}
\item\label{D1} The momentum equation holds in the sense that\begin{align}\label{eq:mom:delta}
\begin{aligned}
&\int_I\frac{\dd}{\dt}\int_{\Omega(t)}\varrho v \cdot b\dx-\int_{\Omega(t)} \Big(\varrho v\cdot \partial_t b +\varrho v\otimes v:\nabla b\Big)\dxt
\\
&+\int_I\int_{\Omega(t)}\mathbb S(\nabla v):\nabla b \dxt-\int_I\int_{\Omega(t)}
p_\delta(\varrho)\,\Div b \dxt\\
&+\int_I\bigg(-\int_Q \varrho_s\partial_t\eta\,\partial_t \phi\dy + \langle DE(\eta),\phi\rangle+\langle D_2R(\eta,\partial_t\eta),\phi\rangle\bigg)\dt
\\&=\int_I\int_{\Omega(t)}\varrho f_f\cdot b\dxt+\int_I\int_Q f_s\cdot \phi\,\dd x\dt
\end{aligned}
\end{align} 
for all $(\phi,b)\in L^2(I;W^{2,q}(Q;\R^n)) \cap W^{1,2}(I;L^2(Q;\R^n))\times C^\infty_c(\overline{I}\times \Omega; \R^n)$ with $\phi(t) = b(t) \circ \eta(t)$ in $Q$ and $\phi(t)=0$ on $P$, where $\Omega(t) := \Omega \setminus \eta(t,Q)$. Moreover, we have $(\varrho v)(0)=q_0$, $\eta(0)=\eta_0$ and $\partial_t\eta(0)=\eta_1$ we well as $\partial_t\eta(t)=v(t)\circ\eta(t)$ in $Q$, $\eta(t)\in\mathcal E$ and $v(t)=0$ on $\partial\Omega$ for a.a. $t\in I$.
\item\label{D2}  The continuity equation holds in the sense that
\begin{align}\label{eq:con:delta}
\begin{aligned}
&\int_I\frac{\dd}{\dt}\int_{\Omega(t)}\varrho \psi\dxt-\int_I\int_{\Omega(t)}\Big(\varrho\partial_t\psi
+\varrho v\cdot\nabla\psi\Big)\dxt=0
\end{aligned}
\end{align}
for all $\psi\in C^\infty(\overline{I}\times\Omega)$ and we have $\varrho(0)=\varrho_0$. 
\item \label{D3} The energy inequality is satisfied in the sense that
\begin{align} \label{eq:ene:delta}
\begin{aligned}
- \int_I &\partial_t \psi \,
\mathscr E_\delta \dt+\int_I\psi\int_{\Omega_\eta}\mathbb S(\nabla v):\nabla v\dxt+2\int_I R(\eta,\partial_t\eta)\dt \\&\leq
\psi(0) \mathscr E_\delta(0)+\int_I\psi\int_{\Omega_{\eta}}\varrho f_f\cdot v\dxt+\int_I\psi\int_Q f_s\,\partial_t\eta\,\dd y\dt
\end{aligned}
\end{align}
holds for any $\psi \in C^\infty_c([0, T))$.
Here, we abbreviated
$$\mathscr E_\delta(t)= \int_{\Omega(t)}\Big(\frac{1}{2} \varrho(t) | {v}(t) |^2 + H_\delta(\varrho(t))\Big)\dx+\int_Q\varrho_s\frac{|\partial_t\eta|^2}{2}\dy+ E(\eta(t)).$$
\end{itemize}
\begin{theorem}\label{thm:delta}
Assume that we have for some $\alpha\in(0,1)$ and $s>0$
\begin{align*}
\frac{|q_0|^2}{\varrho_0}&\in L^1(\Omega_{\eta_0}),\ \varrho_0\in C^{2,\alpha}(\overline\Omega_{\eta_0}), \ \eta_0\in \mathcal E,\ \eta_1\in L^2(Q;\R^n),\\
f_f&\in L^2(I;L^\infty(\R^n)),\ f_s\in L^2(I\times Q;\R^n).
\end{align*}
Furthermore suppose that $\varrho_0$ is strictly positive. There is a solution $(\eta,v,\varrho)\in Y^I\times X_\eta^I\times \widetilde Z_\eta^I$ to \eqref{eq:mom:delta}--\eqref{eq:ene:delta}. 
Here, we have $I=(0,T_*)$, where $T_*<T$ only if the time $T_*$ is the time of the first contact
of the free boundary of the solid body either with itself or $\partial\Omega$ (i.e. $\eta(T_*)\in\partial\mathscr E$). 
\end{theorem}

\begin{lemma}
\label{cor:ap1}
Under the assumptions of Theorem \ref{thm:delta} the continuity equation holds in the renormalized sense as specified in Definition~\ref{def:ren}.
\end{lemma}

For a given $\varepsilon$ we obtain a solution $(\eta_\varepsilon,v_\varepsilon,\varrho_\varepsilon)$ to \eqref{eq:mom:eps}--\eqref{eq:ene:eps} by Theorem \ref{thm:eps}.
In particular, we have
\begin{align} 
\nonumber
 \int_{\Omega^{(\varepsilon)}} &\Big[ \frac{1}{2} \varrho^{(\varepsilon)}(t_1) | {v^{(\varepsilon)}(t_1)} |^2 + H_{\delta}(\varrho^{(\varepsilon)}(t_1))
  \Big] \dx+\int_Q\varrho_s\frac{|\partial_t\eta(t_1)|^2}{2}\dy+ E(\eta^{(\varepsilon)}(t_1))\\\int_0^{t_1}\int_{\Omega^{(\varepsilon)}}&\mathbb S(\nabla v^{(\varepsilon)}):\nabla v^{(\varepsilon)}\dxt
+2\int_0^{t_1} \int_{Q}R(\eta^{(\varepsilon)},\partial_t\eta^{(\varepsilon)})\ds+\varepsilon\int_0^{t_1} \int_{\Omega^{(\varepsilon)}}P''_\delta(\varrho^{(\varepsilon)})|\nabla\varrho^{(\varepsilon)}|^2\dxt\nonumber\\
&\leq \,\int_{\Omega_0}\Big[ \frac{1}{2} \varrho_0 |v_0|^2  + H_{\delta}(\varrho_0) \Big] \dx+\int_Q\varrho_s\frac{|\eta_1|^2}{2}\dy+E(\eta_0)
\label{NTDB}
\end{align}
for any $0 \leq t_1 \leq T$, where $\Omega^{(\varepsilon)} := \Omega \setminus \eta^{(\varepsilon)}(t,Q)$.
We deduce the bounds
\begin{align} 
\label{Nbv1}
 \| \partial_t\nabla\eta^{(\varepsilon)} \|_{L^2(I\times Q) }^2+\sup_{t\in I}\|\partial_t\eta^{(\varepsilon)}\|_{L^{2}(Q)}^2+\sup_{t\in I}\|\eta^{(\varepsilon)}\|_{W^{2,q}(Q)}^q\leq c,\\
\label{Nbv2}
\sup_{t \in I} \| \varrho^{(\varepsilon)} \|^\beta_{L^\beta(\Omega^{(\varepsilon)})}  +
 \sup_{t \in I} \| \varrho^{(\varepsilon)} v^{(\varepsilon)} \|_{L^{\frac{2 \beta}{\beta + 1}}(\Omega^{(\varepsilon)})}^{\frac{2 \beta}{\beta + 1}}  \leq c,\\
 \label{Nbv4}
 \| \nabla v^{(\varepsilon)} \|^2_{L^2(I\times \Omega^{(\varepsilon)}) }+ \varepsilon\| \nabla \varrho^{(\varepsilon)} \|^2_{L^2(I\times\Omega^{(\varepsilon)}) } + \varepsilon\| \nabla (\varrho^{(\varepsilon)})^{\beta/2} \|^2_{L^2(I\times\Omega^{(\varepsilon)}) } \leq c.
\end{align}
Finally, we deduce from the equation of continuity (\ref{eq:con:eps}) (using the renormalized formulation from Theorem~\ref{lem:warme} (b) with $\theta(z)=z^2$ and testing with $\psi\equiv 1$)) that
\begin{equation} \label{Nbv6}
\int_{\Omega^{(\varepsilon)}(t)} \varrho^{(\varepsilon)}(t, \cdot) \dx = \int_{\Omega_0} \varrho_0 \dx,\
 \| \sqrt{\varepsilon} \nabla \varrho^{(\varepsilon)} \|_{L^2(I\times\Omega^{(\varepsilon)})}  \leq c.
\end{equation}
Note that all estimates are independent of $\varepsilon$.
Hence, we may take a subsequence such that
 for some $\alpha\in (0,1)$ we have
\begin{align}
\eta^{(\varepsilon)}&\rightharpoonup^\ast\eta\quad\text{in}\quad L^\infty(I;W^{2,q}(Q;\Omega))\label{eq:conveta1},\\
\eta^{(\varepsilon)}&\rightharpoonup^\ast\eta\quad\text{in}\quad W^{1,\infty}(I;L^2(Q;\R^n)),
\label{eq:conetat1}\\
\eta^{(\varepsilon)}&\rightharpoonup\eta\quad\text{in}\quad W^{1,2}(I;W^{1,2}(Q;\R^n)),
\label{eq:conetat2}\\
\eta^{(\varepsilon)}&\to\eta\quad\text{in}\quad C^\alpha(\overline I\times Q;\Omega)),
\label{eq:conetatal}\\
v^{(\varepsilon)}&\rightharpoonup^\eta v\quad\text{in}\quad L^2(I;W^{1,2}(\Omega;\R^n)),\label{eq:convu1}\\
\varrho^{(\varepsilon)}&\rightharpoonup^{\ast,\eta}\varrho\quad\text{in}\quad L^\infty(I;L^\beta(\Omega)),\label{eq:convrho1}\\
\varepsilon\nabla\varrho^{(\varepsilon)}&\rightarrow^\eta 0\quad\text{in}\quad L^2(I\times\Omega;\R^n).\label{eq:van1}
\end{align}
Arguing as in \cite[Prop. 2.23]{benesovaVariationalApproachHyperbolic2020} one can again benefit from assumption S6 to deduce
\begin{align}\label{15:02}
\eta^{(\varepsilon)}&\rightarrow\eta\quad\text{in}\quad L^q(I;W^{2,q}(Q;\Omega))
\end{align}
such that for a.e.\ $t\in I$
\begin{align}\label{15:02'}
DE(\eta^{(\varepsilon)}(t))&\rightarrow DE(\eta(t))\quad\text{in}\quad W^{-2,q}(Q;\Omega)
\end{align}
We observe that the a priori estimates \eqref{Nbv2} imply uniform bounds of
$\varrho^{(\varepsilon)} v^{(\varepsilon)}$ in $ L^\infty(I,L^\frac{2\beta}{\beta+1}(\Omega^{(\varepsilon)}))$. Therefore, we may apply Lemma \ref{thm:weakstrong} with the choice $v_i\equiv v^{(\varepsilon)}$, $r_i=\varrho^{(\varepsilon)}$, $p=s=2$, $b=\beta$ and $m$ sufficiently large to obtain
\begin{align}
\varrho^{(\varepsilon)}v^{(\varepsilon)}&\rightharpoonup^\eta  {\varrho}  {v}\qquad\text{in}\qquad L^q(I, L^a(\Omega^{(\varepsilon)};\R^n)),\label{conv:rhov2}
\end{align}
where $a\in (1,\frac{2\beta}{\beta+1})$ and $q\in (1,2)$.
We apply Lemma \ref{thm:weakstrong} once more with the choice $v_i\equiv v^{(\varepsilon)}$, $r_i=\varrho^{(\varepsilon)} v^{(\varepsilon)}$,  $p=s=2$, $b=\frac{2\beta}{\beta+1}$ and $m$ sufficiently large to find that
\begin{align}
{\varrho}^{(\varepsilon)}  {v}^{(\varepsilon)}\otimes  {v}^{(\varepsilon)}&\rightharpoonup^\eta  {\varrho}  {v}\otimes  {v}\quad\text{in}\quad L^1(I\times\Omega;\R^{n\times n}).\label{conv:rhovv2}
\end{align}
We also obtain
\begin{align}
\varrho^{(\varepsilon)}v^{(\varepsilon)}&\rightarrow^\eta  {\varrho}  {v}\quad\text{in}\quad L^q(I, L^q(\Omega;\R^n))\label{conv:rhov2B},\\
\varrho^{(\varepsilon)}v^{(\varepsilon)}&\rightharpoonup^{\eta,*}  {\varrho}  {v}\quad\text{in}\quad L^\infty(I, L^{\frac{2\beta}{\beta+1}}(\Omega;\R^n))\label{conv:rhov2BB},
\end{align}
for all $q<\frac{6\beta}{\beta+6}$.
At this stage of the proof the pressure is only bounded in $L^1$, so we have to exclude its concentrations. The standard approach only works locally, where the moving shell is not seen and we obtain the following Lemma (see \cite[Lemma 6.3]{BrSc} for details).
\begin{lemma}\label{prop:higher}
Let $J\times B\Subset I\times\Omega(\cdot)$ be a parabolic cube. The following holds for any $\varepsilon\leq\varepsilon_0(J\times B)$
\begin{equation}\label{eq:gamma+1}
\int_{J\times B}p_\delta(\varrho^{(\varepsilon)})\varrho^{(\varepsilon)}\,\dd x\,\dd t\leq C(J\times B)
\end{equation}
with a constant independent of $\varepsilon$.
\end{lemma}
We still have to exclude concentrations of the pressure at the boundary, which is the object of the following lemma.
\begin{lemma}\label{prop:higherb}
Let $\xi>0$ be arbitrary. There is a measurable set $A_\xi\Subset I\times\Omega(\cdot)$ such that we have for all $\varepsilon\leq \varepsilon_0(\xi)$
\begin{equation}\label{eq:gamma+1b}
\int_{I\times\Omega\setminus A_\xi}p_\delta(\varrho^{(\varepsilon)})\varrho^{(\varepsilon)}\chi_{\Omega^{(\varepsilon)}}\,\dd x\,\dd t\leq \xi.
\end{equation}
\end{lemma}
\begin{proof}
The proof is exactly as in \cite[Lemma 6.4]{BrSc} (which is inspired by \cite{Kuk}) provided we know that 
\begin{align*}
\partial_t\eta^{(\varepsilon)}\in L^2(I;L^2(M;\R^n))
\end{align*}
uniformly in $\varepsilon$. This follows from \eqref{eq:conetat2} due to the trace theorem.
\end{proof}

We connect Lemma~\ref{prop:higher} and Lemma~\ref{prop:higherb} to obtain the following corollary.
\begin{corollary}\label{prop:higherb0}Under the assumptions of Theorem~\ref{thm:delta} there exists a function $\overline p$ such that
\[
p_\delta(\varrho^{(\varepsilon)})\weakto^\eta \overline p\text{ in }L^1(I;L^1(\Omega)),
\]
at least for a subsequence.
Additionally, for $\xi>0$  arbitrary, there is a measurable set $A_\xi\Subset I\times\Omega(\cdot)$ such that $\overline{p}\varrho\in L^1(A_\xi)$ and
\begin{equation}\label{eq:gamma+1b0}
\int_{(I\times\Omega(\cdot))\setminus A_\xi}\overline p\dxt\leq \xi.
\end{equation}
\end{corollary}
Combining Corollary \ref{prop:higherb0} with the convergences \eqref{eq:conveta1}--\eqref{15:02'} we can pass to the limit in \eqref{eq:mom:eps} and \eqref{eq:con:eps} and obtain the following. There is
$$(\eta,v,\varrho,\overline p)\in Y^I\times X_{\eta}^I\times\widetilde W_\eta^I  \times L^1(I\times \Omega_\eta)$$ that satisfies $v(t) = \partial_t \eta(t) \circ \eta(t)$ in $Q$, $v(t)=0$ on $\partial\Omega$ for a.a. $t\in I$, the continuity equation 
\begin{align}\label{eq:apvarrho}
\begin{aligned}
&\int_I\frac{\dd}{\dt}\int_{\Omega(t)}\varrho \psi\dxt-\int_I\int_{\Omega(t)}\Big(\varrho\partial_t\psi
+\varrho v\cdot\nabla\psi\Big)\dxt=0
\end{aligned}
\end{align}
for all $\psi\in C^\infty(\overline I \times \R^3)$ and the coupled weak momentum equation
\begin{align}\label{eq:apulim}
\begin{aligned}
&\int_I\frac{\dd}{\dt}\int_{\Omega(t)}\varrho v\cdot b\dxt-\int_I\int_{\Omega(t)} \Big(\varrho v\cdot \partial_t b +\varrho v\otimes v:\nabla b\Big)\dxt\\
&+\int_{\Omega(t)}\mathbb S(\nabla v):\nabla b\dxt-\int_I\int_{\Omega(t)}
\overline{p}\,\Div b\dxt
\\
&
+\int_I\bigg(-\int_Q \varrho_s\partial_t\eta\,\partial_t \phi\dy + \langle DE(\eta),\phi\rangle+\langle D_2R(\eta,\partial_t\eta), \phi\rangle\bigg)\dt
\\&=\int_I\int_{\Omega(t)}\varrho f_f\cdot b\dxt+\int_I\int_Q f_s\cdot \phi\,\dd x\dt
\end{aligned}
\end{align} 
 for all $(\phi,b)\in L^2(I;W^{2,q}(Q;\Omega) \cap W^{1,2}(L^2(Q;\R^n)) \times C_c^\infty(\overline{I}\times \R^3)$ with $b(t) \circ \eta(t) = \phi(t)$ in $Q$, $b(t)=0$ on $P$ for a.a. $t\in I$.
It remains to show strong convergence of $\varrho^{(\varepsilon)}$.
The proof of strong convergence of the density is based on the effective viscous flux identity introduced in \cite{Li2} and the concept of renormalized solutions from \cite{DL}. Arguing locally, there is no difference to the standard setting. We consider a parabolic cube $J\times B$ with
$J \times B \Subset A \Subset  I\times\Omega(\cdot)$. Due to \eqref{eq:conetatal} we can assume
that $A \Subset  I\times\Omega^{(\varepsilon)}$ (by taking $\varepsilon$ small enough). 
For non-negative $\psi\in C^\infty_c(A)$ we obtain
\begin{align}\label{eq:fluxpsi}
\begin{aligned}
\int_{I\times \Omega} &\psi\big(p_\delta(\varrho^{(\varepsilon)})-(\lambda+2\mu)\Div v^{(\varepsilon)}\big)\,\varrho^{(\varepsilon)}\dxt\\&\longrightarrow\int_{I\times\Omega}\psi \big( \overline{p}-(\lambda+2\mu)\Div v\big)\,\varrho\dxt
\end{aligned}
\end{align}
as $\varepsilon\rightarrow0$.
The proof of Lemma \ref{cor:ap1} follows exactly as in \cite[Lemma 6.2]{BrSc}. 
So, for $\psi\in C^\infty(\overline{I}\times \Omega(\cdot))$ we have
\begin{align}\label{8.4}
\begin{aligned}
\int_{I}\frac{\dd}{\dt}&\int_{\Omega(t)} \theta(\varrho)\,\psi\dxt-\int_I \int_{\Omega(t)}\theta(\varrho)\,\partial_t\psi\dxt
+\int_I\int_{\Omega(t)}\big(\varrho\theta'(\varrho)-\theta(\varrho)\big)\Div v\,\psi\dxt\\
&=\int_{I} \int_{\Omega(t)}\theta(\varrho) v\cdot\nabla\psi \dxt .
\end{aligned}
\end{align} 
In order to deal with the local nature of \eqref{eq:fluxpsi} we use ideas from
\cite{Fe}. First of all, by the monotonicity of the mapping $\varrho\mapsto p(\varrho)$, we find for arbitrary non-negative $\psi\in C^\infty_c(A)$
\begin{align*}
(\lambda+2\mu)\liminf_{\varepsilon\rightarrow0}&\int_{I\times\R^n}\psi \big(\Div v^{(\varepsilon)}\,\varrho^{(\varepsilon)} -\Div v\,\varrho\big)\dxt\\
=&\liminf_{\varepsilon\rightarrow0}\int_{I\times\Omega_{\eta^{(\varepsilon)}}}\psi \big(p(\varrho^{(\varepsilon)})- \overline{p}\big)\big(\varrho^{(\varepsilon)}-\varrho\big)\dxt\geq 0
\end{align*}
using \eqref{eq:fluxpsi} (together with the uniform bounds \eqref{Nbv2} and \eqref{Nbv4}).
As $\psi$ is arbitrary we conclude
\begin{align}\label{8.12}
\overline{\Div v\,\varrho}\geq \Div v\,\varrho \quad\text{a.e. in }\quad I\times\Omega(\cdot),
\end{align}
where 
\begin{align*}
\Div v^{(\varepsilon)}\,\varrho^{(\varepsilon)}\weakto^\eta \overline{\Div v\,\varrho}\quad\text{in}\quad L^1(\Omega;L^1(\Omega_{\eta^{(\varepsilon)}})),
\end{align*}
recall \eqref{eq:convu1} and \eqref{eq:convrho1}. Now, we compute both sides of
\eqref{8.12} by means of the corresponding continuity equations. Due to Theorem
\ref{lem:warme} (b) on the interval $[0,t_1]$ with $\theta(z)=z\ln z$ and $\psi=1$ we have
\begin{align}\label{8.15}
\int_0^{t_1}\int_{\Omega^{(\varepsilon)}}\Div v^{(\varepsilon)}\,\varrho^{(\varepsilon)}\dxt \leq\int_{\Omega_0}\varrho_0\ln(\varrho_0)\dx
-\int_{\Omega^{(\varepsilon)}(t_1)}\varrho^{(\varepsilon)}(t_1)\ln(\varrho^{(\varepsilon)}(t_1))\dx
\end{align}
for almost all $0 \leq t_1 < T$.
Similarly, equation \eqref{8.4} yields
\begin{align}\label{8.14}
\int_0^{t_1}\int_{\Omega(t)}\Div v\,\varrho\dxt=\int_{\Omega_0}\varrho_0\ln(\varrho_0)\dx
-\int_{\Omega(t_1)}\varrho(t)\ln(\varrho(t))\dx.
\end{align}
Combining \eqref{8.12}--\eqref{8.14} shows
\begin{align*}
\limsup_{\varepsilon\rightarrow0}\int_{\Omega^{(\varepsilon)}(t)}\varrho^{(\varepsilon)}(t)\ln(\varrho^{(\varepsilon)}(t))\dx\leq \int_{\Omega(t)}\varrho(t)\ln(\varrho(t))\dx
\end{align*}
for any $t\in I$.
This gives the claimed convergence $\varrho^{(\varepsilon)}\rightarrow\varrho$ in $L^1(I\times\R^3)$ by convexity of $z\mapsto z\ln z$. Consequently, we have $\overline p=p(\varrho)$ and the proof of Theorem \ref{thm:delta} is complete.

\subsection{Proof of Theorem~\ref{thm:main}.}

\label{sec:6}
In this section we are ready to prove the main result of this paper by passing to the limit $\delta\rightarrow0$ in the system \eqref{eq:mom:delta}--\eqref{eq:ene:delta} from Section \ref{subsec:del}.
Given initial data $(q_0,\varrho_0)$ belonging to the function spaces stated in Theorem \ref{thm:main}
it is standard to find regularised versions $q_0^{(\delta)}$ and $\varrho_0^{(\delta)}$ such that for all $\delta>0$
\begin{align*}
 \varrho_0^{(\delta)}\in C^{2,\alpha}(\overline\Omega_0),\ \varrho_0^{(\delta)}\ \text{strictly positive}, \int_{\Omega_0} \varrho_0^{(\delta)} \dx = \int_{\Omega_0} \varrho_0 \dx
 \end{align*}
 as well as $q_0^{(\delta)} \to q_0$ in $L^{\frac{2\gamma}{\gamma+1}}(\Omega_0;\R^n)$, $\varrho_0^{(\delta)} \to \varrho_0$ in $L^\gamma(\Omega_0)$ and
 \begin{align*}
 \int_{\Omega_0}\Big(\frac{1}{2} \frac{| {q}_0^{(\delta)} |^2}{\varrho_0^{(\delta)}} &+ H_{\delta} (\varrho_0^{(\delta)})\Big)\dx\rightarrow  \int_{\Omega_0}\Big(\frac{1}{2} \frac{| {q}_{0} |^2}{\varrho_0} + H(\varrho_{0})\Big)\dx,
 \end{align*}
 as $\delta\rightarrow0$.
For a given $\delta$ we gain a weak solution $(\eta^{(\delta)},v^{(\delta)},\varrho^{(\delta)})$ to \eqref{eq:mom:delta}--\eqref{eq:ene:delta} with this data by Theorem \ref{thm:delta}. It is defined in the interval $(0,T_*)$, where $T_*$ is restricted by the data only. Exactly as in Section \ref{subsec:del}, we write $\Omega^{(\delta)}(t):= \Omega(t) \setminus \eta^{(\delta)}(t,Q)$ and deduce the following uniform bounds from the energy inequality:
\begin{equation} \label{wWS4eta}
 \| \partial_t\nabla\eta^{(\delta)} \|_{L^2(I\times Q) }^2+\sup_{t\in I}\|\partial_t\eta^{(\delta)}\|_{L^{2}(Q)}^2+\sup_{t\in I}\|\eta^{(\delta)}\|_{W^{2,q}(Q)}^q\leq c,\\
\end{equation}
\begin{equation} \label{wWS47}
 \sup_{t \in I} \| \varrho^{(\delta)} \|^{\gamma}_{L^{\gamma}(\Omega^{(\delta)})}  +
  \sup_{t \in I} \delta \|  \varrho^{(\delta)} \|^\beta_{L^\beta(\Omega^{(\delta)})}  
\leq c,
\end{equation}
\begin{equation} \label{wWS48}
\begin{split}
   \sup_{t \in I} \big\| \varrho^{(\delta)} |v^{(\delta)}|^2 \big\|_{L^1(\Omega^{(\delta)})} +
 \sup_{t \in I} \big\| \varrho^{(\delta)} v^{(\delta)} \big\|^\frac{2\gamma}{\gamma+1}_{L^{\frac{2\gamma}{\gamma+1}}(\Omega^{(\delta)})}   \leq c,
\end{split}
\end{equation}
\begin{equation} \label{wWS49}
 \big\| v^{(\delta)} \big\|^{2}_{L^2(I;W^{1,2}(\Omega^{(\delta)}))}  \leq c.
\end{equation}
Finally, we have the conservation of mass principle resulting from the continuity equation, i.e.,
\begin{equation} \label{wWS411}
\| \varrho^{(\delta)}(\tau, \cdot) \|_{L^1(\Omega^{(\delta)})} = \int_{\Omega_{\eta^{(\delta)}}} \varrho(\tau, \cdot) \dx = \int_{\Omega} \varrho_0 \dx  \quad \mbox{for all}\ \tau\in[0,T].
\end{equation}
Hence we may take a subsequence, such that
for some $\alpha\in (0,1)$ we have
\begin{align}
\eta^{(\delta)}&\rightharpoonup^\ast\eta\quad\text{in}\quad L^\infty(I;W^{2,q}(Q;\Omega))\label{eq:conveta}\\
\eta^{(\delta)}&\rightharpoonup^\ast\eta\quad\text{in}\quad W^{1,\infty}(I;L^2(Q;\R^n)),
\label{eq:conetat}\\
\eta^{(\delta)}&\rightharpoonup^\ast\eta\quad\text{in}\quad W^{1,2}(I;W^{1,2}(Q;\R^n)),
\label{eq:conetatt}\\
\eta^{(\delta)}&\to\eta\quad\text{in}\quad C^\alpha(\overline{I}\times Q;\Omega),
\label{eq:conetata}\\
v^{(\delta)}&\rightharpoonup^\eta v\quad\text{in}\quad L^2(I;W^{1,2}(\Omega;\R^n)),\label{eq:convu}\\
\varrho^{(\delta)}&\rightharpoonup^{\ast,\eta}\varrho\quad\text{in}\quad L^\infty(I;L^\gamma(\Omega)).\label{eq:convrho}
\end{align}
Also, we obtain as before again that
\begin{align}\label{15:02''}
DE(\eta^{(\delta)}(t))&\rightarrow DE(\eta(t))\quad\text{in}\quad W^{-2,q}(Q;\Omega).
\end{align}
By Lemma \ref{thm:weakstrong}, arguing as in Section \ref{subsec:del}, we find for all $q\in (1,\frac{6\gamma}{\gamma+6})$ that
\begin{align}
\varrho^{(\delta)}v^{(\delta)}&\rightharpoonup^\eta  {\varrho}  {v}\quad\text{in}\quad L^2(I, L^q(\Omega;\R^n))\label{conv:rhov2delta}\\
{\varrho}^{(\delta)}  {v}^{(\delta)}\otimes  {v}^{(\delta)}&\rightarrow^\eta  {\varrho}  {v}\otimes  {v}\quad\text{in}\quad L^1(I;L^1(\Omega;\R^{n\times n})).\label{conv:rhovv2delta}\\
\sqrt{{\varrho}^{(\delta)}}  {v}^{(\delta)}&\rightarrow^\eta  \sqrt{\varrho}  {v}\quad\text{in}\quad L^1(I;L^1(\Omega)).\label{conv:rhovv2delta'}
\end{align}
As before in Proposition \ref{prop:higher} we have higher integrability of the density (see \cite[Lemma 7.3]{BrSc} for the proof).
\begin{lemma}\label{prop:higher'}
Let $\gamma> \frac{n}{2}$.
Let $J\times B\Subset I\times\Omega(\cdot)$ be a parabolic cube and $0<\Theta\leq\frac{2}{n}\gamma-1$. The following holds for any $\delta\leq \delta_0(J\times B)$
\begin{equation}\label{eq:gamma+1'}
\int_{J\times B}p_\delta(\varrho^{(\delta)})(\varrho^{(\delta)})^{\Theta}\,\dd x\,\dd t\leq C(J\times B)
\end{equation}
with constant independent of $\delta$.
\end{lemma}

In order to exclude concentrations of the pressure at the moving boundary we need the assumption $\gamma>\frac{12}{7}$.
\begin{lemma}\label{prop:higherb'}
Let $\gamma> \frac{2n(n-1)}{3n-2}$.
Let $\xi>0$ be arbitrary. There is a measurable set $A_\xi\Subset I\times\Omega(\cdot)$ such that we have for all $\delta\leq\delta_0$
\begin{equation}\label{eq:gamma+1b'}
\int_{I\times\R^3\setminus A_\xi}p_\delta(\varrho^{(\delta)})
\chi_{\Omega_{\eta^{(\delta)}}}\dd x\,\dd t\leq \xi.
\end{equation}
\end{lemma}
\begin{proof}
The proof is exactly as in \cite[Lemma 7.4]{BrSc} provided we know that  for all $q<4$
\begin{align*}
\partial_t\eta^{(\varepsilon)}\in L^2(I;L^q(M))
\end{align*}
uniformly in $\varepsilon$. This follows from \eqref{eq:conetatt} due to the trace theorem.
\end{proof}
Lemma~\ref{prop:higher'} and Lemma~\ref{prop:higherb'} imply equi-integrability of the sequence  $p_\delta(\varrho^{(\delta)})\chi_{\Omega^{(\delta)}}$. This yields the existence of a function $\overline p$ such that (for a subsequence)
\begin{align}\label{eq:limp'}
p_\delta(\varrho^{(\delta)})\rightharpoonup\overline p\quad\text{in}\quad L^{1}(I\times\Omega),\\
\label{1301}
\delta(\varrho^{(\delta)})^{\beta}+\delta(\varrho^{(\delta)})^{2}\rightarrow0\quad\text{in}\quad L^1(I\times\Omega).
\end{align}
Similarly to Corollary \ref{prop:higherb0} we have the following. 
\begin{corollary}\label{prop:higherb0'}
Let $\xi>0$ be arbitrary. There is a measurable set $A_\xi\Subset I\times\Omega(\cdot)$ such that
\begin{equation}\label{eq:gamma+1b0'}
\int_{I\times\Omega(.) \setminus A_\xi}\overline p\,\dd x\,\dd t\leq \xi.
\end{equation}
\end{corollary}

Using \eqref{eq:limp'} and the convergences \eqref{eq:conveta}--\eqref{15:02''} we can pass to the limit in \eqref{eq:mom:delta} and \eqref{eq:con:delta} and obtain
\begin{align}\label{eq:apulim'}
\begin{aligned}
&\int_I\frac{\dd}{\dt}\int_{\Omega(t)}\varrho v \cdot b\dx-\int_{\Omega(t)} \Big(\varrho v\cdot \partial_t b +\varrho v\otimes v:\nabla b\Big)\dxt
\\
&+\int_I\int_{\Omega(t)}\mathbb S(\nabla v):\nabla b \dxt-\int_I\int_{\Omega(t)}
\overline p\,\Div b\dxt\\
&+\int_I\bigg(-\int_Q \varrho_s\partial_t\eta\,\partial_t \phi\dy + \langle DE(\eta),\phi\rangle+\langle D_2R(\eta,\partial_t\eta),\phi\rangle\bigg)\dt
\\&=\int_I\int_{\Omega(t)}\varrho f_f\cdot b\dxt+\int_I\int_Q f_s\cdot \phi\,\dd x\dt
\end{aligned}
\end{align} 
for all test-functions $(\phi,b)$ with $\phi= b \circ \eta$, $\phi(T,\cdot)=0$ and $b(T,\cdot)=0$. Moreover, the following holds:
\begin{align}\label{eq:apvarrholim}
&\int_I\frac{\dd}{\dt}\int_{\Omega(t)}\varrho \psi\dxt-\int_I\int_{\Omega(t)}\Big(\varrho\partial_t\psi
+\varrho v\cdot\nabla\psi\Big)\dxt=0
\end{align}
for all $\psi\in C^\infty(\overline{I}\times\R^3)$.

It remains to show strong convergence of $\varrho^{(\delta)}$. 
We define the $L^\infty$-truncation
\begin{align}\label{eq:Tk'}
T_k(z):=k\,T\Big(\frac{z}{k}\Big)\quad z\in\R,\,\, k\in\N.
\end{align}
Here $T$ is a smooth concave function on $\R$ such that $T(z)=z$ for $z\leq 1$ and $T(z)=2$ for $z\geq3$. we clearly have
\begin{align}
 T_k(\varrho^{(\delta)})&\weakto {T}^{1,k}\quad\text{in}\quad C_w(I;L^p( \R^3))\quad\forall p\in[1,\infty),\label{eq:Tk1'}\\
\big(T_k'(\varrho^{(\delta)})\varrho^{(\delta)}-T_k(\varrho^{(\delta)})\big)\Div v^{(\delta)}&\rightharpoonup{T}^{2,k}
\quad\text{in}\quad L^2(I\times\R^3),\label{eq:Tk2'}
\end{align}
for some limit functions ${T}^{1,k}$ and ${T}^{2,k}$.
Now we have to show that 
\begin{align}\label{eq:flux'}
\begin{aligned}
\int_{I\times\Omega_{\eta^{(\delta)}}}&\big( p_\delta(\varrho^{(\delta)})-(\lambda+2\mu)\Div v^{(\delta)}\big)\,T_k(\varrho^{(\delta)})\dxt\\&\longrightarrow\int_{I\times\Omega_{\eta}} \big( \overline{p}-(\lambda+2\mu)\Div v\big)\,T^{1,k}\dxt.
\end{aligned}
\end{align}
For this step we are able to use the theory established in \cite{Li2} on a local level. As in
\cite[Subsection 7.1]{BrSc} one can first prove a localised version of \eqref{eq:flux'}
and then use Lemma \ref{prop:higherb'} and Corollary \ref{prop:higherb0'} to deduce the global version.

The next aim is to prove that $\varrho$ is a renormalized solution (in the sense of Definition~\ref{def:ren}).
In order to do so it suffices to use the continuity equation and \eqref{eq:flux'} again on the whole space.
Following line by line the arguments from \cite[Subsection 7.2]{BrSc} we have
\begin{align}\label{eq:Tk''}
\partial_t T^{1,k}+\Div\big( T^{1,k}v\big)+T^{2,k}= 0
\end{align}
in the sense of distributions on $I\times\R^n$. Note that we extended
$\varrho$ by zero to $\R^n$. 
The next step is to show
\begin{align}\label{eq.amplosc''}
\limsup_{\delta\rightarrow0}\int_{I\times\Omega}|T_k(\varrho^{(\delta)})-T_k(\varrho)|^{q}\dxt\leq C,
\end{align}
where $C$ does not depend on $k$ and $q>2$ will be specified later. The proof of \eqref{eq.amplosc''} follows exactly the arguments from the classical setting with fixed boundary
 (see \cite{F}) using \eqref{eq:flux'} and the uniform bounds on $v^{(\delta)}$ (with the only exception that we do not localise).
Using  \eqref{eq.amplosc''} and arguing as in \cite[Sec. 7.2]{BrSc} we obtain the renormalised continuity equation.
As in \cite[Sec. 7.3]{BrSc} we can use the latter one to show strong convergence of the density and as in the end of Theorem \ref{thm:eps} we then extend the existence interval until the first collision, which finishes the proof.
\

\section*{Compliance with Ethical Standards}\label{conflicts}

\smallskip
\par\noindent 
{\bf Funding}. This research was partly funded by:  
\\ (i) Primus Research Programme of  Charles University, Prague (grant number PRIMUS/19/SCI/01).\\
(ii) The program GJ19-11707Y of the Czech national grant agency
(GA\v{C}R).

\smallskip
\par\noindent
{\bf Conflict of Interest}. The authors declare that they have no conflict of interest.
 
\bibliographystyle{alpha}
\bibliography{biblio}

\begin{thebibliography}{CDEG05}

\bibitem[BGN14]{BGN}
T.~Bodn\'ar, G.~P. Galdi, and S.~Ne\v{c}asov\'a.
\newblock {\em Fluid-Structure Interaction and Biomedical Applications}, volume
  2002.
\newblock Springer, Basel, 2014.

\bibitem[BKS20]{benesovaVariationalApproachHyperbolic2020}
Barbora Bene{\v s}ov{\'a}, Malte Kampschulte, and Sebastian Schwarzacher.
\newblock A variational approach to hyperbolic evolutions and fluid-structure
  interactions.
\newblock {\em arXiv:2008.04796 [math]}, August 2020.

\bibitem[BS18]{BrSc}
Dominic Breit and Sebastian Schwarzacher.
\newblock Compressible fluids interacting with a linear-elastic shell.
\newblock {\em Arch. Ration. Mech. Anal.}, 228(2):495--562, 2018.

\bibitem[BS21]{BrSc2}
Dominic Breit and Sebastian Schwarzacher.
\newblock Navier--stokes--fourier fluids interacting with elastic shells.
\newblock {\em arXiv:2101.00824 [math]}, January 2021.

\bibitem[CDEG05]{Ch}
Antonin Chambolle, Beno\^{\i}t Desjardins, Maria~J. Esteban, and C\'{e}line
  Grandmont.
\newblock Existence of weak solutions for the unsteady interaction of a viscous
  fluid with an elastic plate.
\newblock {\em J. Math. Fluid Mech.}, 7(3):368--404, 2005.

\bibitem[Cha15]{Su}
S.~K. Chakrabarti.
\newblock {\em The theory and practice of hydrodynamics and vibration}, volume
  2002.
\newblock River Edge, World Scientific, New York, 2015.

\bibitem[DL89]{DL}
R.~J. DiPerna and P.-L. Lions.
\newblock Ordinary differential equations, transport theory and {S}obolev
  spaces.
\newblock {\em Invent. Math.}, 98(3):511--547, 1989.

\bibitem[Dow15]{Do}
W.~Dowell.
\newblock {\em A Modern Course in Aeroelasticity}, volume 2017 of {\em Solid
  Mechanics and Its Applications}.
\newblock Springer, New York, 2015.
\newblock Fifth Revised and Enlarged Edition.

\bibitem[Fei03]{Fe}
Eduard Feireisl.
\newblock On the motion of rigid bodies in a viscous compressible fluid.
\newblock {\em Arch. Ration. Mech. Anal.}, 167(4):281--308, 2003.

\bibitem[FNP01]{F}
Eduard Feireisl, Anton\'{\i}n Novotn\'{y}, and Hana Petzeltov\'{a}.
\newblock On the existence of globally defined weak solutions to the
  {N}avier-{S}tokes equations.
\newblock {\em J. Math. Fluid Mech.}, 3(4):358--392, 2001.

\bibitem[Gra08]{Gr}
C\'{e}line Grandmont.
\newblock Existence of weak solutions for the unsteady interaction of a viscous
  fluid with an elastic plate.
\newblock {\em SIAM J. Math. Anal.}, 40(2):716--737, 2008.

\bibitem[HK09]{HealeyKroemer}
Timothy~J. Healey and Stefan Kr\"{o}mer.
\newblock Injective weak solutions in second-gradient nonlinear elasticity.
\newblock {\em ESAIM Control Optim. Calc. Var.}, 15(4):863--871, 2009.

\bibitem[Kuk09]{Kuk}
Peter Kuku\v{c}ka.
\newblock On the existence of finite energy weak solutions to the
  {N}avier-{S}tokes equations in irregular domains.
\newblock {\em Math. Methods Appl. Sci.}, 32(11):1428--1451, 2009.

\bibitem[Lio98]{Li2}
Pierre-Louis Lions.
\newblock {\em Mathematical topics in fluid mechanics. {V}ol. 2}, volume~10 of
  {\em Oxford Lecture Series in Mathematics and its Applications}.
\newblock The Clarendon Press, Oxford University Press, New York, 1998.
\newblock Compressible models, Oxford Science Publications.

\bibitem[LR14]{LeRu}
Daniel Lengeler and Michael R\r{u}\v{z}i\v{c}ka.
\newblock Weak solutions for an incompressible {N}ewtonian fluid interacting
  with a {K}oiter type shell.
\newblock {\em Arch. Ration. Mech. Anal.}, 211(1):205--255, 2014.

\bibitem[MC13]{MuCa1}
Boris Muha and Suncica Cani\'{c}.
\newblock Existence of a weak solution to a nonlinear fluid-structure
  interaction problem modeling the flow of an incompressible, viscous fluid in
  a cylinder with deformable walls.
\newblock {\em Arch. Ration. Mech. Anal.}, 207(3):919--968, 2013.

\bibitem[Mit20]{Mi}
Sourav Mitra.
\newblock Local existence of strong solutions of a fluid-structure interaction
  model.
\newblock {\em J. Math. Fluid Mech.}, 22(4):Paper No. 60, 38, 2020.

\bibitem[MRT21]{Ma}
D.~Maity, A.~Roy, and T.~Takahashi.
\newblock Existence of strong solutions for a system of interaction between a
  compressible viscous fluid and a wave equation.
\newblock {\em Nonlinearity}, 34(4):Paper No. 2659, 2021.

\bibitem[MS19]{MuSc}
Boris Muha and Sebastian Schwarzacher.
\newblock Existence and regularity for weak solutions for a fluid interacting
  with a non-linear shell in 3d.
\newblock {\em arXiv:1906.01962 [math]}, 2019.

\bibitem[MT20]{Ma2}
D.~Maity and T.~Takahashi.
\newblock Existence and uniqueness of strong solutions for the system of
  interaction between a compressible navier-stokes-fourier fluid and a damped
  plate equation.
\newblock {\em arXiv:2006.00488 [math]}, 2020.

\bibitem[Mv14]{MuCa}
Boris Muha and Sun\v{c}ica \v{C}ani\'{c}.
\newblock Existence of a solution to a fluid-multi-layered-structure
  interaction problem.
\newblock {\em J. Differential Equations}, 256(2):658--706, 2014.

\bibitem[ST20]{Tr}
Y.-G.~Wang S.~Trifunovi\'c.
\newblock On the interaction problem between a compressible viscous fluid and a
  nonlinear thermoelastic plate.
\newblock {\em arXiv:2010.01639 [math]}, October 2020.

\end{thebibliography}

\end{document}